\documentclass[11pt]{amsart}
\usepackage[margin=1.05in]{geometry}
\usepackage{amsmath,amssymb,amsthm,mathtools,mathrsfs}
\usepackage{enumitem,booktabs,longtable,array}
\usepackage{xcolor}
\usepackage{tikz-cd}
\usepackage[expansion=false]{microtype}
\usepackage[colorlinks=true,linkcolor=blue!50!black,citecolor=green!40!black,urlcolor=blue!60!black]{hyperref}
\usepackage[capitalise,noabbrev]{cleveref}
\usepackage{url}

\numberwithin{equation}{section}
\theoremstyle{plain}
\newtheorem{theorem}{Theorem}[section]
\newtheorem{proposition}[theorem]{Proposition}
\newtheorem{lemma}[theorem]{Lemma}
\newtheorem{corollary}[theorem]{Corollary}
\newtheorem{conjecture}[theorem]{Conjecture}

\newtheorem{mainthm}{Theorem}

\newtheorem{maincor}[mainthm]{Corollary}
\theoremstyle{definition}
\newtheorem{definition}[theorem]{Definition}
\newtheorem{example}[theorem]{Example}
\newtheorem{question}[theorem]{Question}
\theoremstyle{remark}
\newtheorem{remark}[theorem]{Remark}

\crefname{citedthm}{Theorem}{Theorems}
\crefname{mainthm}{Theorem}{Theorems}
\crefname{mainconj}{Conjecture}{Conjectures}
\crefname{maincor}{Corollary}{Corollaries}
\crefname{conjecture}{Conjecture}{Conjectures}
\crefname{question}{Question}{Questions}

\newcommand{\Hilb}{\operatorname{Hilb}}
\newcommand{\Quot}{\operatorname{Quot}}
\newcommand{\col}{\operatorname{col}}
\newcommand{\rk}{\operatorname{rk}}
\newcommand{\Spec}{\operatorname{Spec}}
\newcommand{\ord}{\operatorname{ord}}

\newcommand{\Gr}{\operatorname{Gr}}
\newcommand{\GL}{\operatorname{GL}}
\newcommand{\Bot}{\operatorname{Bot}}
\newcommand{\Nor}{\operatorname{Nor}}

\newcommand{\LL}{\mathbb{L}}
\newcommand{\FF}{\mathbb{F}}
\newcommand{\CC}{\mathbb{C}}
\newcommand{\NN}{\mathbb{N}}
\newcommand{\ZZ}{\mathbb{Z}}
\newcommand{\QQ}{\mathbb{Q}}
\newcommand{\PP}{\mathbb{P}}
\newcommand{\AAA}{\mathbb{A}}
\newcommand{\GG}{\mathbb{G}}
\newcommand{\mm}{\mathfrak{m}}
\newcommand{\cc}{\mathfrak{c}}
\newcommand{\Zs}{\mathsf{Z}}
\newcommand{\cA}{\mathcal{A}}
\newcommand{\cE}{\mathcal{E}}

\newcommand{\bfA}{\mathbf a}
\newcommand{\Mat}{\mathrm{Mat}}
\newcommand{\kk}{\Bbbk}
\newcommand{\Kvar}[1][\kk]{K_0(\mathrm{Var}_{#1})}
\newcommand{\qbin}[3]{\genfrac{[}{]}{0pt}{}{#1}{#2}_{#3}}
\newcommand{\qb}[2]{\qbin{#1}{#2}{Q\LL}}
\newcommand{\pbin}[2]{\qbin{#1}{#2}{q}}
\newcommand{\poch}[2]{(#1;Q\LL)_{#2}}

\begin{document}

\title[Sheaves on nodal singularities I]{Sheaves on nodal singularities I}

\author{Yifeng Huang}
\address{Dept.\ of Mathematics, Emory University}
\email{yifeng.huang@emory.edu}

\author{Ruofan Jiang}
\address{Dept.\ of Mathematics, University of California, Berkeley}
\email{ruofanjiang@berkeley.edu}

\author{Alexei Oblomkov}
\address{Dept.\ of Mathematics and Statistics, University of Massachusetts, Amherst}
\email{oblomkov@umass.edu}

\date{October 7, 2026. Draft.}

\begin{abstract}
We study two families of moduli spaces of sheaves attached to the node, namely the punctual Hilbert schemes of the non-reduced plane curve $x^uy^v=0$, and the punctual Quot schemes of the module $\kk[[x]]^u\oplus\kk[[y]]^v$ over the reduced node $xy=0$. We compute their motivic generating functions for all $u,v$ and over every field. Both are specializations of a single rational function whose numerator counts matrices by rank, and they are related by the substitution $\LL\mapsto Q\LL$. This proves the Hilb--Quot correspondence for the node with independent ranks on the two branches. The main geometric ingredient is an elementary modification of an ideal along a branch of the curve, whose fibres are affine spaces of dimension equal to a contact length. We also construct stratifications of the Hilbert schemes into products of affine spaces and tori, and we formulate a branchwise Hilb--Quot--HOMFLY conjecture, which we test on the colored Hopf link.
\end{abstract}

\maketitle
\tableofcontents
\section{Introduction}\label{sec:intro}

This paper is about two moduli spaces of sheaves supported at the node, namely ideals of finite colength in the local ring of a non-reduced nodal curve, and submodules of finite colength of a module over the reduced node. We recall the context in \cref{subsec:context}, state the main formulas in \cref{subsec:formula}, and explain the proofs in \cref{subsec:strategy,subsec:geometry-intro}. The central geometric idea is the elementary modification along a branch, described in \cref{subsec:strategy}. \Cref{subsec:intro-conj} explains the conjectural link-theoretic context.

\subsection{Hilbert schemes and Quot schemes of plane curve singularities}\label{subsec:context}
Let us introduce the two families of moduli spaces studied in the paper, the punctual Hilbert schemes of plane curve germs and the punctual Quot schemes of their normalization modules, recall the Hilb--Quot conjecture that relates them, and describe the fat node $C_{u,v}=\{x^uy^v=0\}$ together with the results previously known about it.

Let $\kk$ be a field, $R=\kk[[x,y]]$ and $\mm=(x,y)$. For $g\in\mm$ the \emph{punctual Hilbert scheme} of the plane curve germ $C=\{g=0\}$ parametrizes the ideals of colength $n$ of the local ring $\mathcal O_{C,0}=R/(g)$, equivalently
\[
\Hilb_n(C,0)=\{\,I\subseteq R\ \text{an ideal}:\ \dim_{\kk}R/I=n,\ g\in I\,\}.
\]
For reduced $C$ these spaces are classical. They are related to compactified Jacobians by the Macdonald-type formulas of Maulik--Yun and Migliorini--Shende \cite{MY,MS}, and their motivic generating functions are rational \cite{BRV}.

By the conjecture of Oblomkov and Shende \cite{OS}, refined by Oblomkov, Rasmussen and Shende \cite{ORS}, their cohomology should compute the HOMFLY-PT homology of the link of the singularity. The Euler-characteristic version was proved by Maulik \cite{Maulik}. Affine pavings, which give the Betti numbers, are known in several families \cite{LS,Piontkowski,GMO}.

Alongside the Hilbert scheme of the germ stand the punctual Quot schemes of its normalization. Let $g$ be reduced with irreducible factors $g_1,\dots,g_\ell$, let $R_g=R/(g)$, and let $\widetilde R_g=\prod_i\widetilde R_i$, $\widetilde R_i\simeq\kk[[u_i]]$, be the normalization. For a vector of ranks $\bfA=(a_1,\dots,a_\ell)\in\ZZ_{\ge0}^\ell$ put
\[
E_{\bfA}=\bigoplus_{i=1}^{\ell}\widetilde R_i^{\,a_i},
\]
an $R_g$-module, and let $\Quot^{R_g}_m(E_{\bfA})$ be the scheme of $R_g$-submodules of colength $m$ of $E_{\bfA}$, with its reduced structure. For $\bfA=(1,\dots,1)$ this is the local model of the compactified Jacobian.

Kivinen and Trinh \cite{KT} conjectured that for every reduced germ the motivic generating functions of $\Hilb_n(C,0)$ and of $\Quot^{R_g}_m(\widetilde R_g)$ agree after the single substitution $\LL\mapsto Q\LL$, $\LL=[\AAA^1]$ being the class of the affine line and $Q$ the colength variable.

In \cite{HJO} the Quot side was raised to arbitrary rank, $\Quot^{R_g}_m(\widetilde R_g^{\,r})$, and it was proposed that the matching object on the Hilbert side is the punctual Hilbert scheme of the \emph{thickened} curve $\{g^r=0\}$. The single substitution $\LL\mapsto Q\LL$ should relate the two for every $r$. The present paper allows the rank to vary from branch to branch, comparing $\Quot^{R_g}_m(E_{\bfA})$ with the Hilbert schemes of $\{\prod_ig_i^{a_i}=0\}$, and proves the resulting identity in the simplest multibranch case, the node, in every pair of ranks.

For non-reduced curves, a component of multiplicity $m$ should correspond to a link component colored by the symmetric power $\mathrm{Sym}^m$ \cite[\S1.2]{Luan}, \cite{Cautis,HM}. The simplest example is the \emph{fat line} $\{y^v=0\}$. Bejleri computed its motivic Hilbert zeta function \cite{Bej} to be
\begin{equation}\label{eq:bejleri}
\sum_{n\ge0}[\Hilb_n(\{y^v=0\},0)]\,Q^n=\prod_{i=1}^{v}\frac1{1-Q^i\LL^{i-1}} .
\end{equation}
The simplest singular example is the \emph{fat node}
\[
C_{u,v}=\{x^uy^v=0\},\qquad u,v\ge0,
\]
a node whose two branches carry multiplicities $u$ and $v$. Its link is the Hopf link with components colored by $\mathrm{Sym}^u$ and $\mathrm{Sym}^v$, and its Quot-side counterpart is the module
\[
E_{u,v}=\kk[[x]]^u\oplus\kk[[y]]^v\qquad\text{over}\qquad R_{xy}:=\kk[[x,y]]/(xy),
\]
where $\kk[[x]]=R_{xy}/(y)$ and $\kk[[y]]=R_{xy}/(x)$ are the two branches of the normalization.

In a recent paper \cite{Luan}, Luan introduced filtrations of ideals adapted to $C_{u,v}$ and a non-reduced valuation. Luan obtained the following results.
\begin{itemize}[leftmargin=2em]
\item Affine pavings and closed formulas for $u\in\{1,2\}$ \cite[Thms.~1.2, 1.4]{Luan}.
\item Stratifications into pieces $\GG_m^{a}\times\AAA^{b}$, and transfer-matrix formulas for the
virtual Poincar\'e polynomials, for $u\in\{v,v-1,v-2\}$ \cite[Thms.~1.5--1.7]{Luan}.
\item The colored ORS conjecture for the $(\mathrm{Sym}^v,\mathrm{Sym}^1)$-colored Hopf link
\cite[Cor.~1.3]{Luan}.
\end{itemize}
Luan points out that the general case leads to a rapidly growing case analysis \cite[p.~39]{Luan}. The purpose of this paper is to treat all $(u,v)$, on both sides.

\subsection{One rational function, two specializations}\label{subsec:formula}
Now we can state the main results. Both generating functions of the paper are specializations of a single rational function $F_{u,v}(q,Q)$ whose numerator counts $u\times v$ matrices by rank. The Hilbert series of $C_{u,v}$ is its value at $q=Q\LL$, and the Quot series of $E_{u,v}$ is its value at $q=\LL$ (\cref{thmA}). We also state the refinement by contact length (\cref{thmB}) and record the basic features of the numerator.

Throughout we use the variable $Q$ for the colength and $\LL$ for the class of the affine line, together with the standard $q$-notation
\begin{equation}\label{eq:qnotation}
\begin{gathered}
(Q;q)_n:=\prod_{i=0}^{n-1}(1-Qq^{i}),\quad n\in\NN\cup\{\infty\},\qquad\qbin{n}{k}{q}:=\frac{(q;q)_n}{(q;q)_k\,(q;q)_{n-k}}\in\NN[q],\\
(Q;q)_{n_1,\dots,n_l}:=(Q;q)_{n_1}\cdots(Q;q)_{n_l}.
\end{gathered}
\end{equation}
Here $q$ and $Q$ are formal variables. In the generating functions below, $Q$ is the colength variable and $q$ is specialized to $\LL$ or to $Q\LL$; the same letter $q$ denotes the cardinality of a finite field when point counts are taken, since the counting measure sends $\LL$ to $q$. Thus $1/(Q;Q\LL)_v$ is Bejleri's Hilbert zeta function \eqref{eq:bejleri} of the fat line $\{y^v=0\}$ \cite{Bej}. For integers $u,v,m\ge0$ let
\[
R_m(u,v;q):=\prod_{i=0}^{m-1}\frac{(q^u-q^i)(q^v-q^i)}{q^m-q^i},
\]
a polynomial in $q$; for a prime power $q$ it is the number of $u\times v$ matrices of rank $m$ over $\FF_q$ (Landsberg \cite{Landsberg}, see \cref{lem:rank}). Everything in this paper is governed by one rational function in the two independent variables $q,Q$,
\begin{equation}\label{eq:F}
C_{u,v}(q,Q):=\sum_{m=0}^{\min(u,v)}R_m(u,v;q)\,Q^m,\qquad F_{u,v}(q,Q):=\frac{C_{u,v}(q,Q)}{(Q;q)_{u,v}} .
\end{equation}
For a prime power $q$, $C_{u,v}(q,Q)$ is the generating polynomial of $u\times v$ matrices over $\FF_q$, weighted by rank, and by \cref{lem:rank}
\begin{equation}\label{eq:NP}
C_{u,v}(q,Q)=\sum_{m=0}^{\min(u,v)}(-1)^m\,q^{\binom m2}\,(q;q)_m\qbin{u}{m}{q}\qbin{v}{m}{q}\,Q^m .
\end{equation}

\begin{mainthm}[\Cref{thm:main,thm:torus-all,thm:node-main}]\label{thmA}
Let $u,v\ge0$ and let $\kk$ be any field. The following identities hold in $\Kvar{}[[Q]]$.
\begin{enumerate}[label=(\alph*)]
\item (Hilbert side)
\begin{equation}\label{eq:main}
\sum_{n\ge0}[\Hilb_n(C_{u,v},0)]\,Q^n=F_{u,v}(Q\LL,Q)=\frac{C_{u,v}(Q\LL,Q)}{(Q;Q\LL)_{u,v}}.
\end{equation}
\item (Quot side)
\begin{equation}\label{eq:main-quot}
\sum_{m\ge0}[\Quot^{R_{xy}}_m(E_{u,v})]\,Q^m=F_{u,v}(\LL,Q)=\frac{C_{u,v}(\LL,Q)}{(Q;\LL)_{u,v}}.
\end{equation}
\item Consequently the two series are related by the substitution $\LL\mapsto Q\LL$, that is,
\[
\sum_{n\ge0}[\Hilb_n(C_{u,v},0)]\,Q^n=\Big(\sum_{m\ge0}[\Quot^{R_{xy}}_m(E_{u,v})]\,Q^m\Big)\Big|_{\LL\mapsto Q\LL}.
\]
\end{enumerate}
In particular the numbers of $\FF_q$-points are obtained by setting $\LL=q$. For $\kk=\CC$ the E-polynomials are obtained by setting $\LL=ab$, and the virtual Poincar\'e polynomials by setting $\LL=T^2$.
\end{mainthm}

Part (c) follows from (a) and (b), since both sides are specializations of the same rational function and the coefficients of the Quot series are polynomials in $\LL$ (\cref{subsec:hilb-quot}). It is the branchwise Hilb--Quot identity of \cref{conj:main} below for the node, in every pair of ranks. For $u=v=1$ it is the conjecture of Kivinen--Trinh \cite{KT} for the node, and for $u=v$ it is the thickened Hilb--Quot conjecture of \cite[Conj.~1.12]{HJO}. For equal ranks, part (b) belongs to the quadratic-order calculations of \cite{HJ}. See also \cite[Prop.~6.1]{CH}. The proof given here is self-contained and allows independent branch ranks.

Part (a) is the Hilbert-side prediction of the Quot calculation (b). Once (b) is known, the identity $\LL\mapsto Q\LL$ \emph{predicts} $F_{u,v}(Q\LL,Q)$ for the Hilbert series, and Luan's formulas confirm the prediction for $\min(u,v)\le2$ \cite[Thms.~1.2, 1.4]{Luan}. The proof of (a) given here is uniform in $(u,v)$ and independent of \cite{Luan}.

We write $\Zs_{u,v}(Q)$ for the Hilbert series in (a) and $\Zs^{\Quot}_{u,v}(Q)$ for the Quot series in (b). The formula is symmetric in $u$ and $v$. Its numerator has the following features.
\begin{itemize}[leftmargin=2em]
\item Its degree in $Q$ is $uv+\min(u,v)$.
\item At $Q=1$ it equals $C_{u,v}(\LL,1)=\LL^{uv}$, which counts all $u\times v$ matrices.
\item For $u=0$ it gives \eqref{eq:bejleri}, and for $u,v\to\infty$ it gives the
Ellingsrud--Str{\o}mme formula for $\Hilb_n(\AAA^2,0)$ \cite{ES}.
\item For $u\le2$ it specializes to Luan's formulas \cite[Thms.~1.2, 1.4]{Luan} (\cref{ex:u1,ex:g2}).
\item The coefficients of $F_{u,v}(Q\LL,Q)$ lie in $\NN[\LL]$ (\cref{cor:positivity}).
\item On the Quot side the numerator $C_{u,v}(\LL,Q)$ has an unconditional geometric meaning, as the
generating polynomial of the \emph{extension fibres} of $E_{u,v}$, the submodules $L$ with $\widetilde R L=E_{u,v}$. The fibre of colength $k$ is the variety of codimension-$k$ subspaces of $\kk^u\oplus\kk^v$ surjecting onto both summands (\cref{sec:quot}). 
\end{itemize}
The denominator is a product over the two branches.

The factors $1/(Q;Q\LL)_u$ and $1/(Q;Q\LL)_v$ are the Bejleri zeta functions \eqref{eq:bejleri} \cite{Bej} of the fat branches $\{x^u=0\}$ and $\{y^v=0\}$, and in the proof each branch contributes its factor one multiplicity at a time (\cref{thm:B}). On the Quot side the denominator $(Q;\LL)_{u,v}$ is the product of the lattice zeta functions of the two branches (\cref{lem:dvr}). The numerator is the part of the answer in which the two branches interact. For example
\[
C_{1,v}(Q\LL,Q)=1-Q+Q^{v+1}\LL^v,\qquad C_{2,2}(Q\LL,Q)=1-(1-Q\LL)(1+Q\LL)^2Q+Q\LL(1-Q\LL)(1-Q^2\LL^2)Q^2 .
\]
The first case outside the ranges of \cite[Thms.~1.2--1.7]{Luan} is $(u,v)=(3,6)$.

The refined statement keeps track of the \emph{contact length} $j_x(I)=\dim_{\kk}R/(I+(x))$, the length of the intersection of $\Spec R/I$ with the branch $\{x=0\}$. Let $\Hilb_n(C_{u,v},0)_{j_x=j}\subseteq\Hilb_n(C_{u,v},0)$ be the locally closed subvariety of ideals $I$ with $j_x(I)=j$.

\begin{mainthm}[\Cref{thm:refined,thm:torus-all}]\label{thmB}
For all $u,v\ge0$ and every field $\kk$, in $\Kvar{}[z][[Q]]$,
\[
\sum_{n,j\ge0}[\Hilb_n(C_{u,v},0)_{j_x=j}]\,Q^nz^j=\frac{1}{\poch{Qz}{u}}\sum_{k=0}^{v}\qb{v}{k}\frac{(Q(Q\LL)^uz)^k}{\poch{Q(Q\LL)^{v-k}}{k}} .
\]
In particular the contact strata $\Hilb_n(C_{0,v},0)_{j_x=k}=\{I\in\Hilb_n(\{y^v=0\},0):\ I+(x)=(x,y^k)\}$ of the fat line satisfy
\[
\sum_n[\Hilb_n(C_{0,v},0)_{j_x=k}]\,Q^n=Q^k\qb{v}{k}\prod_{i=v-k}^{v-1}\frac1{1-Q(Q\LL)^{i}} .
\]
\end{mainthm}

\subsection{How the formulas are proved}\label{subsec:strategy}
The Quot side reduces to a relative form of Goursat's lemma \cite{Goursat}. On the Hilbert side we explain in detail the elementary modification along a branch, which is the main geometric ingredient of the paper, and the bootstrap, which determines the one remaining unknown, the \textit{contact distribution} of a fat line.
\subsubsection{The Quot side}
The Quot side is relatively easy based on the extension-fibre argument established in our previous work \cite{HJ,HJO}, which applies with unequal ranks on the branches without change. This is one reason for bringing the Quot side into the picture of Hilbert schemes versus colored link homology: it is proved in two pages, for any reduced germ with split normalization, and it explains why the rank counts $R_k(u,v;q)$ appear at all, whereas on the Hilbert side the rank expansion is so far only an identity of generating functions. A submodule $L\subseteq E_{u,v}$ of finite colength projects onto lattices $M\subseteq\kk[[x]]^u$ and $N\subseteq\kk[[y]]^v$, and since $xy=0$,
\[
xM\oplus yN\subseteq L\subseteq M\oplus N .
\]
So $L$ is the inverse image of a subspace of $(M/xM)\oplus(N/yN)\cong\kk^u\oplus\kk^v$ surjecting onto both summands, a relative form of Goursat's lemma \cite{Goursat}. These subspaces are classified by their two kernels and the graph of an isomorphism between the two quotients, which gives the rank-generating polynomial $C_{u,v}$, and the lattices contribute the two Bejleri-type factors \cite{Bej} $1/(Q;\LL)_u$ and $1/(Q;\LL)_v$ (\cref{sec:quot}).

The same argument, for any reduced germ with split normalization, expresses the Quot series of $E_{\bfA}$ as a polynomial in $Q$ with coefficients the classes of the extension fibres, divided by $\prod_i(Q;\LL)_{a_i}$ (\cref{prop:extension}).

\subsubsection{The Hilb side} The Hilbert side has two ingredients. The first is geometric, the elementary modification along a branch, and it carries all the geometry. The second is formal, the bootstrap, which exploits the symmetry between elementary modifications along two branches, and reduces the computation to a formal manipulation of polynomials.\\

\noindent\emph{Elementary modifications along a branch.}
This is the main geometric ingredient of the paper, and the rest of the Hilbert-side argument is built on it. Let $I$ be an ideal of finite colength and $D=\{x=0\}$ one branch of the node. The story begins with the dynamics between the following three subschemes: 
$$ Z=\Spec R/I,\qquad Z\setminus D=\Spec R/(I:x),\qquad Z\cap D=\Spec R/(I+(x)).$$  
One can make sense of the ``thickness'' of the above schemes as the colength of the corresponding ideals (or $\kk$-dimension of the corresponding rings). In particular, the ``thickness'' of $Z\cap D$ is the \emph{contact length} $j_x(I):=\dim_{\kk}R/(I+(x))$.

Multiplication by $x$ gives an exact sequence
\[
0\longrightarrow R/(I:x)\xrightarrow{\ \cdot x\ }R/I\longrightarrow R/\big(I+(x)\big)\longrightarrow0 .
\]
It can be understood geometrically as follows: The passage $I\mapsto(I:x)$ ``peels off'' from $Z$ one layer along the branch $D$. The layer being peeled off is $Z\cap D$, and what remains is $Z\setminus D$. After peeling, the ``thickness'' of $Z$ drops by exactly $j_x(I)$.

The peeling operation is compatible with the curve. For $f\in R$ one has $xf\in I$ if and only if $f\in(I:x)$, so $x^uy^v\in I$ if and only if $x^{u-1}y^v\in(I:x)$. Thus peeling $Z$ along $D$ maps punctual subschemes of $C_{u,v}$ to punctual subschemes of $C_{u-1,v}$. It lowers the multiplicity of the fat branch $\{x^u=0\}$ by one and leaves the other branch untouched.

The operation can be inverted explicitly, and its fibres are as simple as one could hope. Fix an ideal $J$ and an integer $j$. If $(I:x)=J$ and $I+(x)=(x,y^j)$, then $xJ\subseteq I\subseteq J$, the ideal $I$ contains an element of the form $y^j+xc$, and in fact
\[
I=xJ+R\,(y^j+xc).
\]
So $I$ is determined by its \emph{tail} $\bar c\in R/J$.

The only condition on the tail is $y^j+xc\in J$, i.e.\ $X\bar c=-\bar y^{\,j}$, where $X$ is multiplication by $x$ on $R/J$. This is an inhomogeneous linear equation. It is solvable if and only if $j\ge j_x(J)$, and then its solutions form a torsor under $\ker X=(J:x)/J$, a vector space of dimension $j_x(J)$. Hence we obtain the following statement (\cref{thm:A}):\\[5pt]
\begin{quote}
\emph{For every ideal $J$, the ideals $I$ with $(I:x)=J$ and contact length $j$ form an affine space of dimension $j_x(J)$ if $j\ge j_x(J)$, and there are none if $j<j_x(J)$.}\\[5pt]
\end{quote}

The dimension of the affine space depends on $J$ only through its contact length. 
Because the tail is cut out by a linear equation on the module $R/J$, the description persists in families: over any family of ideals $J$ with constant contact length, the ideals $I$ form a Zariski-locally trivial affine bundle (\cref{prop:family}). So classes in $\Kvar$ are simply multiplied by $\LL^{j_x(J)}$, and the recursion closes on the pair (colength, contact length). 

To exploit this, let $G^x_{u,v}(Q,z)=\sum_{n,j}[\Hilb_n(C_{u,v},0)_{j_x=j}]\,Q^nz^j$ be the generating function refined by the contact length, in the notation of \cref{thmB}. Then the above statement gives
\[
[\Hilb_n(C_{u,v},0)_{j_x=j}]=\sum_{j'\le j}\LL^{j'}\,[\Hilb_{n-j}(C_{u-1,v},0)_{j_x=j'}].
\]
Plugging back to $G^x_{u,v}(Q,z)$ and summing over $j\ge j'$ gives the factor $\LL^{j'}(Qz)^{j'}/(1-Qz)$. Therefore
\[
G^x_{u,v}(Q,z)=\frac{G^x_{u-1,v}(Q,Q\LL z)}{1-Qz},\qquad u\ge1,\ v\ge0,
\]
and peeling the branch $\{x=0\}$ $u$ times reduces $C_{u,v}$ to the fat line $\{y^v=0\}$. Explicitly,
\begin{equation}\label{eq:xchain-intro}
\Zs_{u,v}(Q)=\frac{g_v(Q,(Q\LL)^u)}{(Q;Q\LL)_u},\qquad g_v(Q,z):=\sum_{I\ni y^v}Q^{\col I}z^{j_x(I)}\qquad\text{(\cref{thm:B})}.
\end{equation}
Each modification contributes one factor $(1-Q(Q\LL)^i)^{-1}$ of Bejleri's series $1/(Q;Q\LL)_u$ \cite{Bej}, and the entire interaction between the two branches is concentrated in the \textit{contact distribution} $g_v$ of the other branch, evaluated at $z=(Q\LL)^u$.

Nothing here uses the special form of the second branch. For an arbitrary germ $\{g=0\}$, adding a fat branch $\{x^u=0\}$ multiplies the zeta function by $1/(Q;Q\LL)_u$ and evaluates the contact-refined zeta function of $g$ at $(Q\LL)^u$ (\cref{cor:F}).

The same modification along the other branch, $I\mapsto(I:y)$, is Luan's vertical filtration \cite[Lem.~5.1]{Luan}. It reduces $C_{u,v}$ to the fat line $\{x^u=0\}$, whose contact distribution is hard to compute directly for $u\ge3$ (\cref{subsec:whyhard}). Rather than computing it, we combine the two chains in the bootstrap below. In \cref{sec:shadow} we analyse how the contact length with one branch changes along the chain for the other branch. This refinement of the elementary modification is what lifts the formula from point counts to $\Kvar$.\\

\noindent\emph{The bootstrap.} The contact distribution $g_v(Q,z)$ of the fat line is a polynomial in $z$ of degree $\le v$ with $g_v(Q,0)=1$. The curve can be peeled along either branch, so \eqref{eq:xchain-intro} and its mirror give
\[
\frac{g_v(Q,(Q\LL)^u)}{(Q;Q\LL)_u}=\Zs_{u,v}(Q)=\frac{g_u(Q,(Q\LL)^v)}{(Q;Q\LL)_v}\qquad\text{for all }u,v.
\]
By induction on $v$, the right side for $u<v$ gives the values of $g_v(Q,z)$ at the $v$ distinct special points $z=1,Q\LL,\dots,(Q\LL)^{v-1}$. Together with $g_v(Q,0)=1$, interpolation then determines $g_v$, once $\LL$ is inverted (\cref{thm:C}). We found the closed form by solving these interpolation problems for $v\le6$ and recognizing the Gaussian binomials in the formula
\[
g_v(Q,z)=\frac1{(Q;Q\LL)_v}\sum_{k=0}^{v}\qb{v}{k}\poch{Q}{v-k}\,(Qz)^k .
\]

We then proved that it satisfies the functional equation, by the $q$-binomial theorem. The symmetric form of the result is exactly the rank count of \cref{thmA} (\cref{thm:D}). This proves \cref{thmA}(a) for point counts over every finite field, and in $\Kvar{}[\LL^{-1}][[Q]]$; the identity in $\Kvar{}[[Q]]$ itself needs the geometry of \cref{sec:shadow}. The identity behind the symmetric form, read in the independent variables $(q,Q)$, is
\[
C_{u,v}(q,Q)=\sum_{k=0}^{v}\pbin{v}{k}(Q;q)_{v-k}\,Q^kq^{uk},
\]
and it is also the reason for the positivity properties of $F_{u,v}$ discussed in \cref{sec:positivity}.

\subsection{The geometry behind the Hilbert-scheme formula}\label{subsec:geometry-intro}
The lift rests on a dichotomy for the image of an ideal in the reduced node (\textit{split} or \textit{glued}, see definition below) which controls how the contact length changes along the second chain, and it produces stratifications of the Hilbert schemes into pieces $\AAA^\delta\times\GG_m^\tau$ (\cref{thmC}).

To lift \cref{thmA}(a) to $\Kvar$, where $\LL$ is a zero divisor \cite{Borisov}, we study how the contact length varies along the chain $I\mapsto(I:y)$ for the other branch (Luan's vertical filtration \cite[Def.~4.1]{Luan}). Write $j_y(I)=\dim R/(I+(y))$ for the contact length with the other branch. The image of $I$ in the reduced node $\{xy=0\}$ is of one of two kinds (\cref{lem:node}), namely
\[
I+(xy)=(xy,\,x^{j_y(I)},\,y^{j_x(I)})\quad\textit{(split)},\qquad\text{or}\qquad
I+(xy)=(xy,\,x^{j_y(I)}+\gamma\,y^{j_x(I)}),\ \gamma\ne0\quad\textit{(glued)}.
\]
\begin{mainthm}[\Cref{prop:jump,prop:fibre,thm:torus-fatline,thm:torus-all}]\label{thmC}
~\begin{enumerate}[label=(\alph*)]
\item For every ideal $I\ne R$ of finite colength, the contact jumps along the two chains agree, i.e., $j_x(I)-j_x\big((I:y)\big)=j_y(I)-j_y\big((I:x)\big)$, and this number is $1$ if $I$ is split and $0$ if $I$ is glued.
\item Along the chain $I\mapsto(I:y)$, the split ideals in a fibre form either an affine
hyperplane, or the whole fibre, or nothing. In the first case the gluing parameter $\gamma$ is an affine-linear coordinate on the complement of the hyperplane. In the last case $\gamma$ is inherited from $(I:y)$.
\item Consequently every $\Hilb_n(C_{u,v},0)$ is stratified by the contact lengths of $(I:x^ay^b)$ (equivalently, by the ranks of certain operators $X^aY^b$ on $R/I$), into locally closed pieces isomorphic to $\AAA^\delta\times\GG_m^\tau$, on which the contact length is constant. This gives \cref{thmA}(a) in $\Kvar$.
\end{enumerate}
\end{mainthm}

The coefficients of $F_{u,v}(Q\LL,Q)$ lie in $\NN[\LL]$ (\cref{cor:positivity}), whereas the pieces of (c) have classes $\LL^\delta(\LL-1)^\tau$, so that the stratification gives a signed expansion of $[\Hilb_n(C_{u,v},0)]$, different from the termwise positive expansion of \cref{thmB}. Whether these varieties admit affine pavings, and whether their cohomology is pure, so that the coefficients of \cref{thmA}(a) at $\LL=T^2$ are Betti numbers, is not addressed here. For $\min(u,v)\le2$ both hold by Luan's affine pavings \cite[Thms.~1.2, 1.4]{Luan}.

\subsection{The branchwise conjecture and the colored Hopf link}\label{subsec:intro-conj}
Next we describe the conjectural link-theoretic context, namely the branchwise Hilb--Quot--HOMFLY conjecture for reduced plane curve germs with independent ranks on the branches, the prediction it gives for the colored Hopf link (\cref{corD}), and the normalization questions raised by links, which we settle for the node (\cref{thmE}).

Let $g=g_1\cdots g_\ell$ be reduced, let $L_g$ be its link, and color the component of $L_g$ belonging to $g_i$ by the one-row partition $(a_i)$. Let $Z^{\Hilb}_{g;\bfA}(q,Q)$ and $Z^{\Quot}_{E_{\bfA}}(q,Q)$ be the generating series of the compactly supported virtual weight polynomials of $\Hilb_n(\{\prod_ig_i^{a_i}=0\},0)$ and of $\Quot^{R_g}_m(E_{\bfA})$, with $W_c(\AAA^1;q)=q$ (\cref{subsec:geometric-conventions}), and let $B^{L_g}_{\bfA}(Q,T)$ be the bottom Hochschild series of the unreduced row-colored HOMFLY homology of $L_g$, normalized as in \cref{subsec:bottom-normalization}. Put $(Q;q)_{\bfA}:=(Q;q)_{a_1,\dots,a_\ell}$. The three-way statement proposed in \cref{sec:conjecture} is
\begin{equation}\label{eq:threeway-intro}
(Q;qQ)_{\bfA}\,Z^{\Hilb}_{g;\bfA}(q,Q)
=(Q;qQ)_{\bfA}\,Z^{\Quot}_{E_{\bfA}}(qQ,Q)
=(Q;qQ)_{\bfA}\,B^{L_g}_{\bfA}\big(Q,(qQ)^{-1}\big),
\end{equation}
all three being polynomials in $q$ and $Q$ (\cref{conj:main}).

The Hilbert--link equality is the colored Oblomkov--Rasmussen--Shende prediction in the form used by Luan \cite[Conj.~1.1]{Luan}. The Hilbert--Quot equality is new, and \cref{thmA}(c) proves it, motivically, for the node. Equal ranks recover \cite[Conj.~1.12]{HJO}. All ranks equal to one recover the usual Hilb--Quot and bottom-row ORS predictions \cite{KT,ORS}. For $\bfA$ with $\ell=1$ the reduced/unreduced relation of \cite{GGS} turns the right side into the reduced quadruply graded framework of \cite{HJO} (\cref{subsec:knot-convention}).

For the node, the link side gives the following prediction.

\begin{maincor}[\Cref{cor:hopf,prop:gnsss}]\label{corD}
Assume the colored ORS conjecture for $C_{u,v}$ in the normalization of \cite[Cor.~1.3]{Luan}, in its virtual-weight form \eqref{eq:rawthreeway}. Then, up to an overall monomial factor, the $a=0$ part of the Poincar\'e series of the $(\mathrm{Sym}^u,\mathrm{Sym}^v)$-colored Khovanov--Rozansky homology of the Hopf link is
\[
F_{u,v}(T^{-1},Q)=\frac{\sum_{m}Q^m\,R_m(u,v;T^{-1})}{\prod_{i=0}^{u-1}(1-QT^{-i})\,\prod_{i=0}^{v-1}(1-QT^{-i})} .
\]
For $u=1$ this is the formula of Hogancamp--Mellit and Conners \cite{HM,Conners} quoted in \cite[Cor.~2.2]{Luan}. For all $u,v$ it coincides with the bottom Hochschild specialization of the independently colored Hopf-link expression proposed by Gukov, Nawata, Saberi, Sto\v{s}i\'c and Su{\l}kowski \cite[Eq.~(7.8)]{GNSSS}.
\end{maincor}

Thus the Quot calculation gives a geometric derivation, uniform in both colors, of an existing link-theoretic prediction. The identification of that prediction with rigorously constructed row-projector homology remains open for two nontrivial colors (\cref{sec:hopf}).

The link interpretation raises a normalization question that is invisible for knots. The full numerator $C_{u,v}$ has negative coefficients whenever $u,v>0$, although the series $F_{u,v}$ is coefficientwise nonnegative. We formulate two further conjectures for arbitrary row-colored algebraic links. The fully normalized bottom series should be a polynomial (\cref{conj:full-normalization}), and the series normalized at a single component should be coefficientwise nonnegative (\cref{conj:component-positivity}). The latter asks for a reduced homological model which is in general infinite-dimensional for links (\cref{question:reduced-row}). For the node both are verified, in the following sharp form.

\begin{mainthm}[\Cref{prop:minimal,prop:positive-max}]\label{thmE}
Let $1\le u\le v$. The reduced denominator of $F_{u,v}$ in $\QQ(q)(Q)$, normalized to have constant term one in $Q$, is exactly $(Q;q)_{u,v}$. A product $A$ of factors $1-q^jQ$, with multiplicities at most those in $(Q;q)_{u,v}$ and constant term one, satisfies $AF_{u,v}\in\NN[q][[Q]]$ if and only if $A\mid (Q;q)_v$. Thus $(Q;q)_v$ is the unique largest factor of the natural denominator that can be cleared while preserving coefficientwise positivity. Clearing it leaves an infinite positive series, whereas clearing $(Q;q)_{u,v}$ leaves the signed polynomial $C_{u,v}$.
\end{mainthm}

\subsection{Further results}
Below we list further results of the paper that are not needed for the main theorems, together with the higher-rank generalization stated in the last section.

\begin{itemize}[leftmargin=2em]
\item The elementary modification works for any germ. For an arbitrary $g\in\mm$, adding the fat
branch $\{x^u=0\}$ multiplies the zeta function by $1/(Q;Q\LL)_u$ and evaluates the contact-refined zeta function of $g$ at $z=(Q\LL)^u$ (\cref{cor:F}). This holds without any transversality assumption. For example it gives the zeta function of $x^u(x-y^m)$.
\item The Euler-characteristic specialization of \cref{thmA} is an identity for partitions
avoiding a box (\cref{cor:euler}). Its positive form has a bijective proof, but its rank form does not yet.
\item $\Hilb_n(C_{u,v},0)=\Hilb_n(\AAA^2,0)$ exactly for $n\le u+v$, and the first difference
$(\LL-1)\LL^{u+v-1}$ comes from curvilinear ideals (\cref{prop:threshold}). This occurs much earlier than the torus-fixed points detect, namely at $n=(u+1)(v+1)$.
\item For a general reduced germ with split normalization, the reduced denominator of the
Quot series of $E_{\bfA}$ over $\QQ(q^{1/2})$ is exactly $(Q;q)_{\bfA}$ (\cref{cor:general-denominator}). If \cref{conj:main} holds, this makes every one-component normalization of the bottom series of a multibranched algebraic link an infinite series.
\item The Hilbert-side formula has a generalization to the punctual Quot schemes of the free module of
rank $r$ over the fat node, in which the numerator becomes a plethystic evaluation of the modified Macdonald polynomial of the rectangle $(r^v)$ (\cref{thm:higher-rank}). It will be proved in the sequel \cite{FN2}.
\end{itemize}

\subsection{Organization}
\Cref{sec:quot} treats the Quot side, which is the easier of the two and uses only the preliminaries of \cref{sec:prelim}. \Cref{sec:modification,sec:fatline,sec:shadow} form the proof of the Hilbert-side formula. \Cref{sec:conjecture,sec:hopf}, which concern link homology, are not used in the proofs of the geometric results.

\Cref{sec:prelim} collects preliminaries on ideals and modules, Hilbert schemes, the torus action and $q$-calculus, including the rational function \eqref{eq:F}. \Cref{sec:quot} proves \cref{thmA}(b), namely the extension-fibre factorization of \cite{HJ} for a general germ with split normalization and unequal ranks, and the mixed-rank node, and derives \cref{thmA}(c). \Cref{sec:modification} contains the main geometric ingredient, the elementary modification along a branch, with the structure of its fibres, its version in families, and the $x$-chain. \Cref{sec:fatline} proves the bootstrap, the closed form for the fat line, and \cref{thmA}(a) for point counts.

\Cref{sec:consequences} collects consequences of the Hilbert-scheme formula. \Cref{sec:shadow} constructs the torus stratifications and proves \cref{thmC}, and \cref{thmA}(a) and \cref{thmB} in $\Kvar$.

\Cref{sec:positivity} proves \cref{thmE}. \Cref{sec:conjecture} states the branchwise Hilb--Quot--HOMFLY conjecture and the two normalization conjectures, with the evidence supplied by the node. \Cref{sec:hopf} discusses the colored Hopf link. Finally, \cref{sec:questions} states a generalization of the Hilbert-side formula to Quot schemes of free modules of arbitrary rank over the fat node, to be proved in the sequel \cite{FN2}, and lists open questions.

\subsection{Conventions}\label{subsec:conventions}
The following conventions are used throughout: the ring $R$ and colengths, varieties and reduced structures, partitions and monomial ideals, the variables of the generating functions, and the facts about the Grothendieck ring of varieties that we use.

Throughout, $\kk$ is a field, $R=\kk[[x,y]]$ and $\col(I)=\dim_{\kk}R/I$. A \emph{variety} is a reduced separated scheme of finite type over $\kk$. Locally closed subsets of schemes are given their reduced structure, and classes in $\Kvar$ depend only on the reduced structure.

A partition $\lambda$ is drawn with its boxes $(a,b)\in\NN^2$ standing for the monomials $x^ay^b$. Its rows run in the $x$-direction, row $b$ consisting of the boxes $(a,b)$ with $0\le a<\lambda_b$, and $I_\lambda$ is the monomial ideal whose standard monomials are the boxes of $\lambda$. The virtual-weight variable of \cref{subsec:geometric-conventions} is again $q$, with $W_c(\AAA^1;q)=q$, and $p(n)$ is the number of partitions of $n$.

\emph{The Grothendieck ring.} We work motivically. Here $\Kvar$ is the Grothendieck ring of varieties over $\kk$ and $\LL=[\AAA^1]$. We use the following standard facts. If a variety $Z$ is a finite disjoint union of locally closed subvarieties $Z_i$, then $[Z]=\sum_i[Z_i]$. If $Z\to S$ is Zariski-locally trivial with fibre $F$, then $[Z]=[F][S]$. In particular $[Z]=\LL^r[S]$ for an affine bundle of rank $r$, and $[\AAA^\delta\times\GG_m^\tau]=\LL^\delta(\LL-1)^\tau$.

The class $\LL$ is a zero divisor in $\Kvar$ \cite{Borisov}, so identities in $\Kvar$ cannot in general be obtained by cancelling powers of $\LL$. For $\kk=\FF_q$, counting $\FF_q$-points defines a ring homomorphism $\Kvar[\FF_q]\to\ZZ$ with $\LL\mapsto q$. In \cref{sec:fatline,sec:shadow} it is used to identify polynomials in $\LL$. Other realizations of classes, such as the Euler characteristic (\cref{cor:euler}) and the weight polynomials used for the comparison with link homology (\cref{sec:conjecture,sec:hopf}), only serve to draw consequences.

\emph{Variables.} The colength variable is $Q$ throughout, on the Hilbert side, on the Quot side and on the link side. It is the variable $z$ of \cite{HJO,KT} and the variable $Q$ of \cite{Luan}. The first variable $q$ of \eqref{eq:F} is a formal variable, which we specialize on the fly: the Hilbert side sits at $q=Q\LL$, the Quot side at $q=\LL$, and this is the whole content of the substitution $\LL\mapsto Q\LL$. All $q$-identities are proved for the formal variables $q,Q$ and then specialized. The letter $q$ also denotes the cardinality of a finite field, the counting measure being $\LL\mapsto q$; no confusion arises, since in every formula the formal variable has been specialized to $\LL$ or $Q\LL$ before points are counted.

On the link side we use the variables $(Q,T)$ of \cite{Luan}, with $q=T^{-1}$ (\cref{subsec:bottom-normalization}). In the Quot-side sections \ref{sec:quot}--\ref{sec:conjecture} the two ranks of the node are still written $u,v$, so that the module $E_{u,v}$ and the polynomials $C_{u,v}$, $F_{u,v}$ carry the same indices as the curve $C_{u,v}$. For a germ with $\ell$ branches the ranks are $\bfA=(a_1,\dots,a_\ell)$.

\subsection*{Acknowledgment} The authors thank fruitful discussions with Eugene Gorsky, Yuze Luan, Martin Olsson, Ken Ono, Vivek Shende, Sug Woo Shin, Peisheng Yu, Zhiwei Yun, and Peng Zhou.  AO research was partially supported by  NSF grant DMS-2200798. The authors acknowledge that the completion of the proof of \Cref{thm:A}
 is partly assisted by generative AIs, such as ChatGPT-5.6 Sol.  We also used generative AIs to improve exposition and to search references.

\section{Preliminaries}\label{sec:prelim}

In this section we collect the algebraic and geometric preliminaries. The most important notion is the \emph{contact length} $j_x(I)=\col(I+(x))$ of an ideal with the branch $\{x=0\}$ (\cref{def:contact}). It is the dimension of the fibres of the elementary modification of \cref{sec:modification}, and the whole Hilbert-side computation is a bookkeeping of colength and contact length.

The dictionary between ideals of finite colength and cyclic modules with commuting nilpotent operators (\cref{prop:adhm,lem:rankdict}) translates colengths and contact lengths of colon ideals into ranks of operators. It is also the language in which families of ideals are constructed (\cref{lem:points}). We then define the punctual Hilbert schemes as reduced projective varieties, record the torus action, and collect the $q$-calculus used in \cref{sec:fatline,sec:quot,sec:positivity}, including the node rational function \eqref{eq:F}, which governs all the generating functions of the paper.

Throughout, $\kk$ is a field, $R=\kk[[x,y]]$, $\mm=(x,y)$, and $\col(I)=\dim_{\kk}R/I$. For an ideal $I$ of finite colength we write $V=V(I)=R/I$, and $X,Y$ for the operators of multiplication by $x$ and $y$ on $V$.

\subsection{Ideals of finite colength}
Let us collect elementary facts about ideals of finite colength in $R=\kk[[x,y]]$. They contain a power of the maximal ideal, colength is additive along colon ideals, and the contact lengths with the two branches have several equivalent descriptions. The contact length $j_x(I)$ introduced here is the central invariant of the paper.

\begin{lemma}\label{lem:mn}
Let $I\subseteq R$ be an ideal with $\col(I)=n<\infty$. Then $\mm^n\subseteq I$.
\end{lemma}
\begin{proof}
Put $M_k=(\mm^k+I)/I\subseteq R/I$. Then $M_0\supseteq M_1\supseteq\cdots$ and $M_{k+1}=\mm M_k$. If $M_{k+1}=M_k$, Nakayama's lemma applied to the finitely generated $R$-module $M_k$ gives $M_k=0$. So the chain decreases strictly until it reaches $0$. Each strict step lowers the dimension by at least one and $\dim M_0=n$, so $M_n=0$, i.e.\ $\mm^n\subseteq I$.
\end{proof}

In particular every ideal of finite colength is $\mm$-primary, and it is generated by polynomials. Conversely, an ideal $I\supseteq\mm^N$ is determined by the subspace $I/\mm^N$ of the finite-dimensional space $R/\mm^N=\kk[x,y]/(x,y)^N$.

\begin{lemma}[colon and sum]\label{lem:colon}
Let $I\subseteq R$ be an ideal and $0\ne f\in R$. Then, in $\NN\cup\{\infty\}$,
\[
\col(I)=\col(I:f)+\col\big(I+(f)\big).
\]
\end{lemma}
\begin{proof}
Multiplication by $f$ induces an exact sequence of $R$-modules
\[
0\longrightarrow R/(I:f)\xrightarrow{\ \cdot f\ }R/I\longrightarrow R/\big(I+(f)\big)\longrightarrow0 .
\]
The map $r\mapsto fr$ is well defined on $R/(I:f)$ and injective by the definition $(I:f)=\{r:fr\in I\}$. Its image is $(fR+I)/I$, whose cokernel is $R/(I+(f))$. Dimension is additive in short exact sequences of vector spaces, also when some terms are infinite-dimensional.
\end{proof}

\begin{definition}[contact lengths]\label{def:contact}
For an ideal $I$ of finite colength and $f\in\mm$ put
\[
j_f(I):=\col\big(I+(f)\big),
\]
the \emph{contact length} of $I$ with the branch $\{f=0\}$. We only use $f=x$ and $f=y$. Since $R/(x)\cong\kk[[y]]$ and $R/(y)\cong\kk[[x]]$ are discrete valuation rings, $I+(x)=(x,y^{j_x(I)})$ and $I+(y)=(y,x^{j_y(I)})$.
\end{definition}

\begin{lemma}\label{lem:contact}
Let $I$ have finite colength $n$.
\begin{enumerate}[label=(\alph*)]
\item $j_x(I)=n-\rk X=\dim_{\kk}(I:x)/I$, which is the number of Jordan blocks of the nilpotent operator
$X$. Similarly $j_y(I)=n-\rk Y$.
\item $j_x(I)$ is the length of the scheme-theoretic intersection of $\Spec R/I$ with the line
$\{x=0\}$.
\item $j_x(I)=0$ if and only if $I=R$. If $y^v\in I$, then $j_x(I)\le v$.
\item For every $m\ge0$ one has $y^m\in I+(x)$ if and only if $m\ge j_x(I)$.
\end{enumerate}
\end{lemma}
\begin{proof}
(a) $X$ is an endomorphism of the $n$-dimensional space $V$ with kernel $(I:x)/I$ and cokernel $R/(I+(x))$. Kernel and cokernel have the same dimension $n-\rk X$. The operator is nilpotent because $x^n\in I$ (\cref{lem:mn}), and the kernel of a nilpotent operator has dimension equal to the number of its Jordan blocks.

(b) $R/(I+(x))=(R/I)\otimes_RR/(x)$ is the coordinate ring of the intersection. (c) If $I\ne R$, then $I+(x)\subseteq\mm$, so $j_x(I)\ge1$. If $y^v\in I$, then $(x,y^v)\subseteq I+(x)$, so $j_x(I)\le v$. (d) This holds because $I+(x)=(x,y^{j_x(I)})$.
\end{proof}

\subsection{Modules and ranks}
Next we translate ideals of finite colength into linear algebra. An ideal of colength $n$ is the same as a pair of commuting nilpotent operators on an $n$-dimensional space together with a cyclic vector, and the colengths and contact lengths of colon ideals are ranks of monomials in these operators (\cref{lem:rankdict}). This dictionary is used to construct families of ideals and to define the strata of \cref{sec:shadow}.

\begin{proposition}[ideals as cyclic modules]\label{prop:adhm}
Let $K$ be a field and $n\ge0$. Consider triples $(X,Y,e)$ with $X,Y$ commuting nilpotent endomorphisms of $K^n$ and $e\in K^n$ \emph{cyclic}, i.e.\ $K[X,Y]e=K^n$.
\begin{enumerate}[label=(\alph*)]
\item The map $(X,Y,e)\mapsto I=\{f\in K[[x,y]]:f(X,Y)e=0\}$ induces a bijection between
$GL_n(K)$-orbits of such triples and ideals of $K[[x,y]]$ of colength $n$. The action of $GL_n(K)$ on triples is free.
\item For $f\in K[[x,y]]$ one has $f\in I$ if and only if $f(X,Y)=0$.
\item A pair $(X,Y)$ of commuting nilpotent endomorphisms admits a cyclic vector if and only if
$\rk[X\,|\,Y]=n-1$, where $[X\,|\,Y]$ is the $n\times2n$ matrix with blocks $X$ and $Y$. In that case the cyclic vectors are exactly the vectors outside the hyperplane $XK^n+YK^n$.
\end{enumerate}
\end{proposition}
\begin{proof}
Since $X$ and $Y$ are commuting nilpotent operators, $f(X,Y)$ is defined for every power series $f$, because only finitely many monomials act non-trivially.

(a) The map $\phi:K[[x,y]]\to K^n$, $f\mapsto f(X,Y)e$, is $K[[x,y]]$-linear for the module structure in which $x,y$ act by $X,Y$. It is surjective because $e$ is cyclic, and its kernel is $I$. Hence $\col I=n$. Conversely, given $I$ of colength $n$, a basis of $R/I$ gives a triple consisting of the matrices of multiplication by $x$ and $y$ and the coordinates of $\bar1$. Two bases give triples in the same orbit, and the triple returns $I$. If $g\in GL_n(K)$ fixes $(X,Y,e)$, then $g$ commutes with $X,Y$ and fixes $e$, so it fixes every $f(X,Y)e$, i.e.\ $g=1$.

(b) If $f\in I$, then $f(X,Y)h(X,Y)e=h(X,Y)f(X,Y)e=0$ for all $h$, so $f(X,Y)=0$ by cyclicity. The converse is clear.

(c) $K^n$ is a module over the local ring $K[[x,y]]$ with $\mm K^n=XK^n+YK^n$. By Nakayama's lemma, $e$ generates $K^n$ if and only if $e\notin\mm K^n$, and a generator exists if and only if $\dim K^n/\mm K^n=1$. Since $X,Y$ are nilpotent and $n\ge1$, $\mm K^n\ne K^n$, and $\dim\mm K^n=\rk[X\,|\,Y]$. For $n=0$ the statement is trivial.
\end{proof}

\begin{lemma}[rank dictionary]\label{lem:rankdict}
Let $I$ have finite colength and $k\ge0$. Multiplication by $y^k$ induces an isomorphism of $R$-modules $R/(I:y^k)\cong Y^kV$, and multiplication by $x^k$ an isomorphism $R/(I:x^k)\cong X^kV$. Consequently $\col(I:y^k)=\rk Y^k$, $\col(I:x^k)=\rk X^k$, and
\begin{equation}\label{eq:rankdict}
\begin{aligned}
j_y\big((I:y^k)\big)&=\rk Y^k-\rk Y^{k+1}, & j_x\big((I:y^k)\big)&=\rk Y^k-\rk XY^k,\\
j_x\big((I:x^k)\big)&=\rk X^k-\rk X^{k+1}, & j_y\big((I:x^k)\big)&=\rk X^k-\rk X^kY.
\end{aligned}
\end{equation}
Moreover $\col\big(I+(xy)\big)=\col(I)-\rk XY$.
\end{lemma}
\begin{proof}
The $R$-linear map $R\to V$, $f\mapsto\overline{y^kf}$, has kernel $(I:y^k)$ and image $y^kV=Y^kV$, on which $x,y$ act by the restrictions of $X,Y$. For the module $M=Y^kV$ one has $j_y((I:y^k))=\dim M/yM=\rk Y^k-\rk Y^{k+1}$ and $j_x((I:y^k))=\dim M/xM=\rk Y^k-\rk XY^k$. The case of $x^k$ is symmetric. Finally $R/(I+(xy))=V/XYV$.
\end{proof}

\subsection{Hilbert schemes}\label{subsec:hilb}Let us  define the punctual Hilbert schemes $\Hilb_n(\{g=0\},0)$ as reduced closed subschemes of a Grassmannian, describe their functor of points in terms of cyclic modules, and record elementary facts about reduced schemes that are used to show that set-theoretic bijections between families of ideals are isomorphisms of varieties.

Let $W_n=\kk[x,y]/(x,y)^n$, a vector space of dimension $\binom{n+1}2$, and let $\mathrm{Gr}_n$ be the Grassmannian of $n$-dimensional quotients of $W_n$. For a $\kk$-scheme $S$, a quotient $q:W_n\otimes\mathcal O_S\to\mathcal F$, with $\mathcal F$ locally free of rank $n$, is called \emph{stable} if its kernel is stable under multiplication by $x$ and by $y$. This is a closed condition. It is the vanishing of the two composites $\ker q\to W_n\otimes\mathcal O_S\xrightarrow{x,\,y}W_n\otimes\mathcal O_S\to\mathcal F$. Let $G_n\subseteq\mathrm{Gr}_n$ be the closed subscheme of stable quotients, and for $g\in R$ let $G_n(g)\subseteq G_n$ be the closed subscheme where moreover $q(\bar g)=0$, $\bar g$ being the class of $g$ in $W_n$. These are projective $\kk$-schemes.

\begin{lemma}\label{lem:points}
\begin{enumerate}[label=(\alph*)]
\item For every field extension $K/\kk$, $G_n(g)(K)$ is the set of ideals of $K[[x,y]]$ of colength
$n$ containing $g$.
\item For a $\kk$-scheme $S$, giving a morphism $S\to G_n$ is equivalent to giving, up to
isomorphism, a locally free sheaf $\mathcal F$ of rank $n$ on $S$ with commuting endomorphisms $X,Y$ and a section $e$, such that every monomial of degree $n$ in $X,Y$ vanishes and $e$ generates $\mathcal F$ as an $\mathcal O_S[X,Y]$-module.
\end{enumerate}
\end{lemma}
\begin{proof}
(a) A $K$-point of $\mathrm{Gr}_n$ is a subspace of $W_n\otimes K$ of codimension $n$. It is stable if and only if it is an ideal of $K[x,y]/(x,y)^n$. Such ideals correspond to the ideals of $K[[x,y]]$ that contain $\mm^n$ and have colength $n$, and by \cref{lem:mn} these are all ideals of colength $n$.

(b) Given a stable quotient, multiplication by $x$ and $y$ descends to commuting endomorphisms of $\mathcal F$, and $e=q(\bar1)$ generates. Conversely, given $(\mathcal F,X,Y,e)$, the map $q(f)=f(X,Y)e$ is defined on $W_n\otimes\mathcal O_S$ because degree-$n$ monomials act by zero. It is surjective because $e$ generates, and its kernel is stable because $q$ intertwines $x,y$ with $X,Y$. The two constructions are mutually inverse.
\end{proof}

\begin{definition}\label{def:hilb}
For $g\in\mm$ we put $\Hilb_n(\{g=0\},0):=G_n(g)_{\mathrm{red}}$ and $\Hilb_n(\AAA^2,0):=(G_n)_{\mathrm{red}}$. Locally closed subsets of these schemes are given their reduced induced structures. We write $C_{u,v}=\{x^uy^v=0\}$.
\end{definition}

This agrees with \cite[Def.~3.1]{Luan}. Since only reduced structures enter, the various natural scheme structures on punctual Hilbert schemes play no role, and by \cref{lem:points}(a) the set of $K$-points of $\Hilb_n(\{g=0\},0)$ is the set of ideals of $K[[x,y]]$ of colength $n$ containing $g$, for every field $K\supseteq\kk$. On $G_n$ we have the \emph{tautological bundle} $\mathcal T$ of rank $n$, with fibre $R/I$ over $I$, together with its endomorphisms $X,Y$ and section $e$. For a monomial $M$ in $X,Y$ the function $I\mapsto\rk M$ is lower semicontinuous, so the sets $\{\rk M=r\}$ are locally closed.

We shall use two elementary facts about reduced schemes repeatedly.

\begin{lemma}\label{lem:reduced}
Let $S$ be a reduced scheme of finite type over $\kk$ and $Y$ a separated $\kk$-scheme.
\begin{enumerate}[label=(\alph*)]
\item If $f,g:S\to Y$ are morphisms such that $f\circ s=g\circ s$ for every field $K$ and every
$s\in S(K)$, then $f=g$.
\item If $f:S\to Y$ is a morphism whose image lies in a locally closed subset $Z\subseteq Y$, then
$f$ factors uniquely through $Z$ with its reduced induced structure.
\item A section of a locally free sheaf on $S$ that vanishes at every point is zero.
\end{enumerate}
\end{lemma}
\begin{proof}
(c) Locally the section is a tuple of functions vanishing at every point. Such functions are nilpotent, hence zero on a reduced scheme.

(a) The equalizer of $f$ and $g$ is a closed subscheme of $S$ (as $Y$ is separated) through which every point $\operatorname{Spec}\kappa(s)\to S$ factors. So it contains every point of $S$, and by (c), applied to its ideal sheaf locally, it equals $S$.

(b) Restrict to an open subscheme in which $Z$ is closed. A morphism from a reduced scheme into a scheme factors through the reduced induced structure on any closed subset containing its image \cite[Ex.~II.3.11(d)]{Hartshorne}.
\end{proof}

\subsection{The torus action}\label{subsec:torus}
The two-dimensional torus acts on the Hilbert schemes of $C_{u,v}$, with the monomial ideals as fixed points. This is used for the Euler-characteristic statements of \cref{sec:consequences}.

The torus $\mathbf T=\GG_m^2$ acts on $R$ by $(\alpha,\beta)\cdot x=\alpha x$, $(\alpha,\beta)\cdot y=\beta y$, hence on $W_n$ and on $G_n$. Since $x^uy^v$ is an eigenvector, $\mathbf T$ preserves $\Hilb_n(C_{u,v},0)$. For a partition $\lambda$ we write $I_\lambda$ for the monomial ideal whose standard monomials are the boxes of $\lambda$ (\cref{subsec:conventions}).

\begin{lemma}\label{lem:fixed}
Let $\kk$ be infinite. An ideal $I$ of finite colength is $\mathbf T$-invariant if and only if it is monomial, $I=I_\lambda$ with $|\lambda|=\col(I)$. Moreover $x^uy^v\in I_\lambda$ if and only if $(u,v)\notin\lambda$.
\end{lemma}
\begin{proof}
By \cref{lem:mn}, $I$ is determined by the subspace $I/\mm^n$ of the $\mathbf T$-representation $W_n$, in which the monomials span pairwise distinct characters. A $\mathbf T$-stable subspace of a representation with multiplicity-free characters is, over an infinite field, spanned by the character vectors it contains. So $I$ is monomial. The standard monomials of a monomial ideal form a Young diagram, and $x^uy^v\in I_\lambda$ exactly when $x^uy^v$ is not standard.
\end{proof}

\subsection{$q$-calculus}\label{subsec:qcalc}
The $q$-series identities used in the paper are the $q$-Pascal rules, the finite $q$-binomial theorem and its limit, a trinomial identity for Gaussian binomials, and Landsberg's count \cite{Landsberg} of matrices of given rank over a finite field. They are the only algebraic input of the bootstrap (\cref{sec:fatline}) and of the positivity statements (\cref{sec:positivity}).

We work with the formal variables $q,Q$ of \eqref{eq:qnotation}. The Gaussian binomials satisfy $\qbin{n}{k}{q}=\qbin{n}{n-k}{q}$, $\qbin{n}{0}{q}=\qbin{n}{n}{q}=1$, and we put $\qbin{n}{k}{q}=0$ for $k<0$ or $k>n$.

\begin{lemma}[$q$-Pascal rules]\label{lem:pascal}
For $1\le k\le n$,
\[
\qbin{n+1}{k}{q}=\qbin{n}{k}{q}+q^{\,n+1-k}\qbin{n}{k-1}{q}=q^{k}\qbin{n}{k}{q}+\qbin{n}{k-1}{q}.
\]
In particular $\qbin{n}{k}{q}\in\NN[q]$ is monic of degree $k(n-k)$.
\end{lemma}
\begin{proof}
Over the common denominator $(q;q)_{k}(q;q)_{n+1-k}$, the first right-hand side is
\[
\frac{(q;q)_{n}}{(q;q)_{k}(q;q)_{n+1-k}}\big((1-q^{n+1-k})+q^{n+1-k}(1-q^{k})\big)
=\frac{(q;q)_{n+1}}{(q;q)_{k}(q;q)_{n+1-k}}.
\]
The second identity is analogous. Positivity, the degree and monicity follow by induction on $n$.
\end{proof}

\begin{lemma}[finite $q$-binomial theorem]\label{lem:qbinomial}
For $m\ge0$, in $\ZZ[q,Q]$,
\[
(Q;q)_{m}=\prod_{i=0}^{m-1}(1-Qq^i)=\sum_{j=0}^{m}\qbin{m}{j}{q}(-1)^jq^{\binom j2}Q^j .
\]
\end{lemma}
\begin{proof}
Induction on $m$, using $(Q;q)_{m+1}=(Q;q)_{m}(1-Qq^{m})$. Since $\binom{j-1}{2}+m=\binom j2+m+1-j$, the coefficient of $Q^j$ in the product of the right-hand side with $1-Qq^m$ is
\[
(-1)^jq^{\binom j2}\qbin{m}{j}{q}+(-1)^{j}q^{m+\binom{j-1}2}\qbin{m}{j-1}{q}
=(-1)^jq^{\binom j2}\Big(\qbin{m}{j}{q}+q^{m+1-j}\qbin{m}{j-1}{q}\Big),
\]
and by \cref{lem:pascal} the bracket is $\qbin{m+1}{j}{q}$. See also \cite[Ch.~3]{Andrews}.
\end{proof}

\begin{corollary}[Euler]\label{cor:euler-id}
At $q=Q\LL$, coefficientwise in $Q$,
\[
(Q;q)_{\infty}=\sum_{m\ge0}\frac{(-1)^mq^{\binom m2}Q^m}{(q;q)_{m}} .
\]
\end{corollary}
\begin{proof}
For fixed $j$ we have $\qbin{m}{j}{q}=(q^{m-j+1};q)_{j}/(q;q)_{j}$, which tends to $1/(q;q)_{j}$ as $m\to\infty$, coefficientwise in $Q$, because $q^{m-j+1}$ has $Q$-order $m-j+1$. Similarly $(Q;q)_{m}\to(Q;q)_{\infty}$. Pass to the limit in \cref{lem:qbinomial}.
\end{proof}

\begin{lemma}\label{lem:trinomial}
For $k,i\ge0$ with $k+i\le v$,
\[
\qbin{v}{k}{q}\qbin{v-k}{i}{q}=\qbin{v}{k+i}{q}\qbin{k+i}{k}{q}=\frac{(q;q)_{v}}{(q;q)_{k}(q;q)_{i}(q;q)_{v-k-i}}.
\]
\end{lemma}

\begin{lemma}[matrices of given rank; Landsberg {\cite{Landsberg}}]\label{lem:rank}
For a prime power $q$ and $0\le m\le\min(u,v)$,
\begin{align*}
R_m(u,v;q)&:=\#\{M\in\Mat_{u\times v}(\FF_q):\rk M=m\}\\
&=\prod_{i=0}^{m-1}\frac{(q^u-q^i)(q^v-q^i)}{q^m-q^i}
=(-1)^m q^{\binom m2}(q;q)_m\qbin{u}{m}{q}\qbin{v}{m}{q}.
\end{align*}
\end{lemma}
\begin{proof}
A rank-$m$ map $\FF_q^v\to\FF_q^u$ is determined by its image $W$, which is one of the $\qbin{u}{m}{q}=\prod_{i=0}^{m-1}\frac{q^u-q^i}{q^m-q^i}$ subspaces of dimension $m$, together with a surjection $\FF_q^v\to W$, i.e.\ an $m\times v$ matrix of rank $m$. There are
\[
\prod_{i=0}^{m-1}(q^v-q^i)=(-1)^mq^{\binom m2}\frac{(q;q)_v}{(q;q)_{v-m}}
=(-1)^mq^{\binom m2}(q;q)_m\qbin{v}{m}{q}
\]
of these. See also \cite[\S1.7]{Morrison}.
\end{proof}

We regard $R_m(u,v;q)$ as a polynomial in $q$. Summing over $m$ counts all matrices.

\begin{corollary}\label{cor:allmatrices}
$\sum_{m=0}^{\min(u,v)}R_m(u,v;q)=q^{uv}$ as polynomials in $q$.
\end{corollary}
\begin{proof}
Both sides agree at every prime power $q$, hence at infinitely many points.
\end{proof}

\subsection{The node rational function}\label{subsec:nodefunction}
Let $q,Q$ be independent variables, and recall from \eqref{eq:F}
\[
C_{u,v}(q,Q)=\sum_{m=0}^{\min(u,v)}R_m(u,v;q)\,Q^m\in\ZZ[q,Q],\qquad F_{u,v}(q,Q)=\frac{C_{u,v}(q,Q)}{(Q;q)_{u,v}}\in\ZZ[q][[Q]].
\]
By \cref{lem:rank}, $C_{u,v}=C_{v,u}$, $C_{u,v}(q,0)=1$, and $C_{u,v}(q,Q)$ has the expansion \eqref{eq:NP} in Gaussian binomials. The two specializations that occur in this paper are
\begin{equation}\label{eq:two-specializations}
\begin{aligned}
F_{u,v}(Q\LL,Q)&=\frac{C_{u,v}(Q\LL,Q)}{(Q;Q\LL)_{u,v}}\qquad(\text{Hilbert side}),\\
F_{u,v}(\LL,Q)&=\frac{C_{u,v}(\LL,Q)}{(Q;\LL)_{u,v}}\qquad(\text{Quot side}),
\end{aligned}
\end{equation}
the first being \cref{thmA}(a) and the second \cref{thmA}(b). On the link side the same function is evaluated at $q=T^{-1}$ (\cref{sec:hopf}). All $q$-identities below are proved for the formal variables $q,Q$ and then specialized. Since $\LL$ is a free variable, $Q\LL$ and $Q$ are algebraically independent, so an identity in $\ZZ[q,Q]$ specializes to an identity in $\ZZ[\LL,Q]$ at $q=Q\LL$, and conversely.

\section{The Quot side and the Hilb--Quot identity}\label{sec:quot}

We begin with the Quot side, that is, with submodules of finite colength of the normalization module $E_{u,v}=\kk[[x]]^u\oplus\kk[[y]]^v$ over the reduced node. It is much easier than the Hilbert side. The extension-fibre argument of \cite{HJ,HJO} applies without change when the ranks on the two branches differ, so this section mainly cites \cite{HJ} and records the differences, and it holds for any reduced germ with split normalization. It also explains why the rank counts $R_k(u,v;\LL)$ appear at all, as the classes of the extension fibres; the Hilbert side has no such explanation yet.

In \cref{subsec:extension} we recall, for any reduced complete local ring with split normalization and any ranks on the branches, that the Quot series is a polynomial in $Q$, the generating function of the \emph{extension fibres} (the submodules that generate $E$ over the normalization), divided by a product of lattice zeta functions, one for each branch (\cref{prop:extension}). For the node the extension fibres are the subspaces of $\kk^u\oplus\kk^v$ surjecting onto both summands, a relative form of Goursat's lemma \cite{Goursat}, and their classes are the numbers of matrices of given rank (\cref{lem:goursat,thm:node-main}); this proves \cref{thmA}(b). Together with the Hilbert-side formula, proved in \cref{sec:modification,sec:fatline,sec:shadow}, it gives \cref{thmA}(c) (\cref{subsec:hilb-quot}).

The significance is twofold. The numerator of \cref{thmA} acquires an unconditional geometric meaning on the Quot side, and the substitution $\LL\mapsto Q\LL$ is seen to exchange two different geometric decompositions of the same rational function. Throughout this section the colength variable is again $Q$. The ranks of the node are written $u,v$ as before, and the ranks of a germ with $\ell$ branches are $\bfA=(a_1,\dots,a_\ell)$.

\subsection{Normalization modules and extension fibres}\label{subsec:extension}
Throughout this subsection the ring is an arbitrary reduced complete local ring with split normalization, and the ranks on the branches are arbitrary. We define the Quot schemes of the normalization modules and recall from \cite{HJ} the motivic zeta function of lattices over a discrete valuation ring (\cref{lem:dvr}) and the extension-fibre factorization (\cref{prop:extension}), with unequal ranks on the branches. Neither requires a plane embedding.

Let $R$ be a reduced complete local $\kk$-algebra with residue field $\kk$ whose finite normalization is split,
\[
\widetilde R=\prod_{i=1}^{\ell}\widetilde R_i,\qquad \widetilde R_i=\kk[[u_i]],
\]
and for $\bfA\in\ZZ_{\ge0}^{\ell}$ let $E=E_{\bfA}=\bigoplus_i\widetilde R_i^{\,a_i}$, an $R$-module through $R\to\widetilde R$. An $R$-submodule $L\subseteq E$ of finite colength $m$ contains $\mm_R^mE$, by the argument of \cref{lem:mn} applied to the chain $(\mm_R^kE+L)/L$. Since $R\subseteq\widetilde R$ is finite, there is an ideal
\[
\cc=\prod_iu_i^{c_i}\widetilde R_i\subseteq R,\qquad c_i\ge1,
\]
of $\widetilde R$ contained in $R$ (any such ideal inside the conductor), and $E'_m:=\cc\,\mm_R^mE$ is an $\widetilde R$-submodule of $E$ of finite colength contained in every $L$ of colength $m$.

\begin{definition}\label{def:quot}
$\Quot^R_m(E)$ is the closed subscheme of the Grassmannian of codimension-$m$ subspaces of the finite-dimensional space $E/E'_m$ on which the subspace is stable under $R$, taken with its reduced structure. Its $K$-points, for every field $K\supseteq\kk$, are the submodules of colength $m$ of $E_K=\bigoplus_iK[[u_i]]^{a_i}$ over the complete local ring $R_K=R\,\widehat\otimes_\kk K$, by the argument of \cref{lem:points}(a). For the node, $R_K=K[[x,y]]/(xy)$. We write
\[
\Zs^{\Quot}_E(Q):=\sum_{m\ge0}[\Quot^R_m(E)]\,Q^m\in\Kvar{}[[Q]].
\]
\end{definition}

Put $c(\bfA)=\sum_ia_ic_i=\dim_\kk E/\cc E$, and for $k\ge0$ let
\[
\cE_k(E):=\{L\subseteq E:\ L\text{ an $R$-submodule},\ \widetilde RL=E,\ \dim_\kk(E/L)=k\},
\]
the \emph{extension fibre} of colength $k$. Here $\widetilde RL$ is the $\widetilde R$-submodule generated by $L$. Since $\widetilde R=\prod_i\widetilde R_i$, it is $\bigoplus_i\widetilde R_i\,\mathrm{pr}_i(L)$.

\begin{lemma}[lattices in a discrete valuation ring]\label{lem:dvr}
Let $A=\kk[[u]]$ and $r\ge0$. For $\lambda\in\NN^r$ let $\Sigma_\lambda\subseteq\Quot^A_{|\lambda|}(A^r)$ be the set of lattices $M\subseteq A^r$ admitting a basis of the form
\begin{equation}\label{eq:triangular}
b_i=u^{\lambda_i}e_i+\sum_{j>i}h_{ji}(u)\,e_j,\qquad h_{ji}\in\kk[u],\ \deg h_{ji}<\lambda_j,\qquad 1\le i\le r.
\end{equation}
\begin{enumerate}[label=(\alph*)]
\item Every lattice lies in exactly one $\Sigma_\lambda$, and its basis \eqref{eq:triangular} is unique.
\item $\Sigma_\lambda$ is locally closed, and the coefficients of the $h_{ji}$ are regular functions on it which define an isomorphism $\Sigma_\lambda\cong\AAA^{\sum_j(j-1)\lambda_j}$. The $b_i$ are sections of the universal lattice $\mathcal M$ over $\Sigma_\lambda$, and for every $c\ge0$ the map $e_i\mapsto b_i$ is an isomorphism of bundles $(A/u^c)^r\otimes\mathcal O_{\Sigma_\lambda}\cong\mathcal M/u^c\mathcal M$, whose fibre at $M$ is the isomorphism $(A/u^c)^r\to M/u^cM$ given by the basis $b_1,\dots,b_r$.
\item Consequently
\[
\sum_{m\ge0}[\Quot^A_m(A^r)]\,Q^m=\sum_{\lambda\in\NN^r}\LL^{\sum_j(j-1)\lambda_j}Q^{\sum_j\lambda_j}=\frac1{(Q;\LL)_r}.
\]
\end{enumerate}
\end{lemma}
\begin{proof}
The $\Sigma_\lambda$ are Schubert cells of the affine Grassmannian of $A$, for the flag $V_1\supseteq V_2\supseteq\cdots$ with $V_i=\bigoplus_{j\ge i}Ae_j$, with the bases \eqref{eq:triangular} as normal forms, exactly as for the cells of \cite[\S5.1]{HJ}, and the formula in (c) is \cite[(3.8)]{HJ}. We only indicate the mechanism. The image of $M\cap V_i$ in $V_i/V_{i+1}\cong A$ is $u^{\lambda_i}A$, which defines $\lambda$, and any lift of $u^{\lambda_i}$ to $M\cap V_i$ is turned into the unique $b_i$ by division with remainder of its coordinates $e_j$, $j>i$, by $u^{\lambda_j}$. This procedure is linear in the coefficients, so it also applies to local sections of the universal lattice; this makes the $b_i$ regular sections and the $h_{ji}$ regular functions on $\Sigma_\lambda$, which is locally closed as the common level set of the upper semicontinuous functions $\bar M\mapsto\dim(\bar M\cap\bar V_i)$ in $(A/u^m)^r$. Conversely, arbitrary coefficients $h_{ji}$ define a free submodule with basis \eqref{eq:triangular}, of colength $|\lambda|$, which gives the inverse morphism $\AAA^{\sum_j(j-1)\lambda_j}\to\Sigma_\lambda$ (\cref{lem:reduced}). The trivialization in (b) holds because $e_i\mapsto b_i$ is a map of locally free sheaves that is bijective on every fibre. Summing $\LL^{\sum_j(j-1)\lambda_j}Q^{|\lambda|}$ over $\lambda$ gives $\prod_{j=0}^{r-1}(1-\LL^jQ)^{-1}$.
\end{proof}

\begin{proposition}[extension-fibre factorization]\label{prop:extension}
Each $\cE_k(E)$ is a locally closed subvariety of $\Quot^R_k(E)$, empty for $k>c(\bfA)$, and
\begin{equation}\label{eq:extension-factorization}
\Zs^{\Quot}_E(Q)=\frac{\displaystyle\sum_{k=0}^{c(\bfA)}[\cE_k(E)]\,Q^k}{\displaystyle\prod_i(Q;\LL)_{a_i}} .
\end{equation}
Consequently $\mathcal N^R_E(Q):=\prod_i(Q;\LL)_{a_i}\,\Zs^{\Quot}_E(Q)\in\Kvar{}[Q]$.
\end{proposition}
\begin{proof}
This is the extension-fibre argument of \cite[Thm.~4.4 and Cor.~5.3]{HJ}, where all the $a_i$ are equal; unequal ranks change nothing, and we only indicate the steps. If $\widetilde RL=E$, then $\cc E=\cc\widetilde RL\subseteq L$, since $\cc\subseteq R$. So $\cE_k(E)$ is the locus in the Grassmannian of codimension-$k$ subspaces of $E/\cc E$ cut out by $R$-stability, a closed condition, and by the condition that the subspace generates $E/\cc E$ over $\widetilde R$, an open condition which is equivalent to $\widetilde RL=E$ by Nakayama's lemma, since $\cc$ lies in the radical of $\widetilde R$. This proves the first assertion.

For an arbitrary $R$-submodule $L\subseteq E$ of colength $m$ put $M:=\widetilde RL=\bigoplus_iM_i$, a product of lattices $M_i\subseteq\widetilde R_i^{\,a_i}$. Then $L$ lies in the extension fibre of $M$ and $\col L=\col M+\col(M/L)$. Stratify $\Quot^R_m(E)$ by $k:=\col(M/L)$ and by the cells $\Sigma_\lambda=\prod_i\Sigma_{\lambda^{(i)}}$ of \cref{lem:dvr} containing $M$. On each stratum $Z_{\lambda,k}$ the map $L\mapsto M$ is a morphism to $\Sigma_\lambda$, because $\widetilde R\mathcal L$ is the image of a map of locally free sheaves of constant rank, and the triangular bases of \cref{lem:dvr}(b) identify the bundle $\mathcal M/\cc\mathcal M$ with the constant bundle $(E/\cc E)\otimes\mathcal O$, as $\widetilde R$-modules and hence as $R$-modules. Under this identification $L\mapsto(M,L/\cc M)$ is a morphism $Z_{\lambda,k}\to\Sigma_\lambda\times\cE_k(E)$, and $(M,L_0)\mapsto$ the inverse image of $L_0$ in $M$ is a morphism in the other direction. They are inverse bijections on points over every field, since $L_0$ generates $M/\cc M$ over $\widetilde R$ exactly when $\widetilde RL=M$, so they are inverse isomorphisms by \cref{lem:reduced}. Hence
\[
[\Quot^R_m(E)]=\sum_{|\lambda|+k=m}[\Sigma_\lambda]\,[\cE_k(E)],
\]
and multiplying generating series, with \cref{lem:dvr}(c) for each branch, gives \eqref{eq:extension-factorization}.
\end{proof}

\Cref{prop:extension} proves polynomiality in $Q$ over the Grothendieck ring. It does not say that the coefficients $[\cE_k(E)]$ are polynomials in $\LL$, and the extension fibres are in general not proper.

\subsection{The mixed-rank node}\label{subsec:node}
For the node, the extension fibres of $E_{u,v}$ are the subspaces of $\kk^u\oplus\kk^v$ surjecting onto both summands, which we classify by a relative form of Goursat's lemma \cite{Goursat} (\cref{lem:goursat}). Their classes are the numbers of matrices of given rank, and this proves \cref{thmA}(b) (\cref{thm:node-main}).

Let $R=R_{xy}=\kk[[x,y]]/(xy)$, $A=\kk[[x]]=R/(y)$, $B=\kk[[y]]=R/(x)$, and $E=E_{u,v}=A^u\oplus B^v$. Here $x$ acts on $B^v$ by zero and $y$ acts on $A^u$ by zero. The normalization is $\widetilde R=A\times B$, and $\cc=\mm=(x)\oplus(y)$ is an ideal of $\widetilde R$ contained in $R$, with $c_1=c_2=1$. So $c(u,v)=u+v$ and $E/\cc E=\kk^u\oplus\kk^v$, on which $R$ acts through $R/\mm=\kk$.

\begin{lemma}[subdirect residue spaces]\label{lem:goursat}
Let $U,V$ be vector spaces of dimensions $u,v$. The variety $G_{u,v;k}$ of codimension-$k$ subspaces $S\subseteq U\oplus V$ whose projections onto both summands are surjective is empty unless $0\le k\le\min(u,v)$. Otherwise it is a Zariski-locally trivial $\GL_k$-bundle over $\Gr(u-k,U)\times\Gr(v-k,V)$. In particular
\begin{equation}\label{eq:goursat-class}
[G_{u,v;k}]=\qbin{u}{k}{\LL}\,\qbin{v}{k}{\LL}\,\LL^{\binom k2}\prod_{j=1}^{k}(\LL^j-1)=R_k(u,v;\LL).
\end{equation}
\end{lemma}
\begin{proof}
Let $U_0=S\cap U$ and $V_0=S\cap V$. Surjectivity of $S\to V$ gives $\dim U_0=\dim S-\dim V=u-k$, and symmetrically $\dim V_0=v-k$. So $k\le\min(u,v)$. The two projections induce an isomorphism $U/U_0\xrightarrow{\sim}V/V_0$ whose graph is $S/(U_0\oplus V_0)$. Conversely, the inverse image of the graph of any isomorphism $U/U_0\to V/V_0$ is a subspace of the required kind.

On $G_{u,v;k}$ the kernels $U_0,V_0$ of the projections have constant dimension, hence are subbundles, and over $\Gr(u-k,U)\times\Gr(v-k,V)$ the fibre is the variety of isomorphisms between the two universal rank-$k$ quotient bundles. Trivializing the bundles proves local triviality. The formula follows from the classes of the Grassmannians and $[\GL_k]=\prod_{i=0}^{k-1}(\LL^k-\LL^i)$, and the last equality is \cref{lem:rank}.
\end{proof}

\begin{theorem}[\cref{thmA}(b)]\label{thm:node-main}
For all $u,v\ge0$ and every field $\kk$, in $\Kvar{}[[Q]]$,
\begin{equation}\label{eq:node-quot}
\sum_{m\ge0}[\Quot^{R_{xy}}_m(E_{u,v})]\,Q^m=\frac{C_{u,v}(\LL,Q)}{(Q;\LL)_{u,v}}=F_{u,v}(\LL,Q).
\end{equation}
Moreover $\cE_k(E_{u,v})=G_{u,v;k}$, so that the numerator $C_{u,v}(\LL,Q)=\sum_k[\cE_k(E_{u,v})]Q^k$ is the generating polynomial of the extension fibres.
\end{theorem}
\begin{proof}
Since $R$ acts on $E/\cc E=\kk^u\oplus\kk^v$ through $\kk$, every subspace is $R$-stable, and a subspace generates $E/\cc E$ over $\widetilde R=A\times B$ if and only if it surjects onto both summands. So $\cE_k(E_{u,v})=G_{u,v;k}$, as varieties, and \cref{prop:extension} together with \cref{lem:goursat} gives
\[
\Zs^{\Quot}_{E_{u,v}}(Q)=\frac{\sum_{k=0}^{\min(u,v)}R_k(u,v;\LL)\,Q^k}{(Q;\LL)_{u,v}}=F_{u,v}(\LL,Q). \qedhere
\]
\end{proof}

Concretely, a submodule $L\subseteq E_{u,v}$ of finite colength projects onto lattices $M=\mathrm{pr}_{A^u}(L)\subseteq A^u$ and $N=\mathrm{pr}_{B^v}(L)\subseteq B^v$ (these projections are already $A$-, resp.\ $B$-lattices, because $x$ and $y$ act on $A^u$ and $B^v$ through $A$ and $B$), and since $xy=0$,
\[
xM\oplus yN\subseteq L\subseteq M\oplus N .
\]
Thus $L$ is the inverse image of a subspace $S=L/(xM\oplus yN)\subseteq(M/xM)\oplus(N/yN)$ surjecting onto both summands, and $\col L=\col M+\col N+\operatorname{codim}S$. This is the relative Goursat lemma \cite{Goursat} behind \eqref{eq:node-quot}.

By \cref{lem:goursat}, the class of the extension fibre $G_{u,v;k}$ equals the class of the locus of rank-$k$ matrices in $\Mat_{u\times v}$, although the two spaces are not canonically identified without a choice of dualities. Thus the numerator of \cref{thmA} has an unconditional geometric interpretation on the Quot side, as the generating polynomial of the (nonproper) extension fibres, whereas on the Hilbert side the rank expansion is so far only an identity of generating functions.

\subsection{The Hilb--Quot identity for the fat node}\label{subsec:hilb-quot}
Part (c) of \cref{thmA} needs no separate proof. By \cref{thm:node-main} the Quot series is $F_{u,v}(\LL,Q)$, and by \cref{thm:torus-all}(b), proved in \cref{sec:shadow}, the Hilbert series is $F_{u,v}(Q\LL,Q)$. Both are series in $Q$ whose coefficients are polynomials in $\LL$ (\cref{lem:goursat,thm:torus-all}), and $F_{u,v}(Q\LL,Q)$ is obtained from $F_{u,v}(\LL,Q)$ by substituting $Q\LL$ for $\LL$ in these coefficients, as the substitution is applied to the expansion of the denominator as well. This is \cref{thmA}(c).

The identity is a statement in the Grothendieck ring of varieties, for every field, and it is not a consequence of a bijection between the two moduli spaces. Both are projective varieties (\cref{def:hilb,def:quot}), but the substitution $\LL\mapsto Q\LL$ mixes colength and dimension, sending the term $c\,\LL^iQ^j$ of the Quot series to $c\,\LL^iQ^{i+j}$, so that the coefficient of $Q^n$ on the Hilbert side collects the terms of the Quot side with $i+j=n$, and the two varieties of a given colength are in general of different dimensions.

The denominators match termwise, each branch contributing its Bejleri factor $(Q;Q\LL)_a$ \cite{Bej} on the Hilbert side (\cref{thm:B}) and its lattice factor $(Q;\LL)_a$ on the Quot side (\cref{lem:dvr}). The numerators match because the fat-line distribution of \cref{thm:D} and the subdirect residue spaces of \cref{lem:goursat} are both governed by the rank polynomial $C_{u,v}$. For $u=v=1$ the identity is the conjecture of Kivinen--Trinh \cite{KT} for the node, and for $u=v$ it is \cite[Conj.~1.12]{HJO}.

\section{Elementary modifications along a branch}\label{sec:modification}

This section contains the main geometric ingredient of the paper. Let $D=\{x=0\}$ be one branch of the node. For a punctual subscheme $Z=\Spec R/I$, the ideal $(I:x)$ defines the residual subscheme of $Z$ with respect to $D$. By \cref{lem:colon} with $f=x$ there is an exact sequence
\begin{equation}\label{eq:residual}
0\longrightarrow R/(I:x)\xrightarrow{\ \cdot x\ }R/I\longrightarrow R/\big(I+(x)\big)\longrightarrow0,
\end{equation}
so passing from $I$ to $(I:x)$ removes the layer $Z\cap D$, of length $j_x(I)$, and lowers the colength by $j_x(I)$. We call the inverse passage, from $J=(I:x)$ back to $I$, an \emph{elementary modification along $D$}.

The results of the section are the following.
\begin{itemize}[leftmargin=2em]
\item \Cref{thm:A} describes the fibres of $I\mapsto(I:x)$ completely. For a fixed ideal $J$ and a
fixed contact length $j$, the ideals $I$ with $(I:x)=J$ and $j_x(I)=j$ form an affine space of dimension $j_x(J)$ if $j\ge j_x(J)$, and there are none otherwise. Each such $I$ is obtained from $J$ by adding a single generator $y^j+xc$.
\item \Cref{prop:family} shows that this description is algebraic in families. Over a locally closed
family of ideals $J$ of constant contact length, the ideals $I$ form a Zariski-locally trivial affine bundle. In particular classes in $\Kvar$ are multiplied by $\LL^{j_x(J)}$.
\item \Cref{thm:B}, the \emph{$x$-chain}, iterates the modification. Since $x^uy^v\in I$ if and only if
$x^{u-1}y^v\in(I:x)$, $u$ modifications along $D$ reduce the fat node $C_{u,v}$ to the fat line $\{y^v=0\}$, and the generating function of $C_{u,v}$ is expressed through the contact distribution $g_v$ of the fat line.
\end{itemize}
The significance of these results is that they replace the problem of describing the ideals that contain $x^uy^v$, for which a direct analysis leads to a rapidly growing case distinction (\cref{subsec:whyhard}), by a recursion in which the only invariant to be tracked is the pair (colength, contact length).

The proof of the Hilbert-side formula uses nothing about $\Hilb_n(C_{u,v},0)$ beyond this section and the same statements with $x$ and $y$ exchanged, together with, for the lift to $\Kvar$, the finer analysis of the contact length in \cref{sec:shadow}.

\subsection{The fibres of $I\mapsto(I:x)$}
This subsection determines the fibres of the map $I\mapsto(I:x)$ over an arbitrary field. \Cref{thm:A} shows that they are affine spaces whose dimension is the contact length of the image, and that every ideal in a fibre is obtained by adding a single explicit generator. We first explain where this description comes from, then state and prove the theorem, isolate the features of it that are used later, and work out the first examples.

Let $I$ be an ideal of finite colength, and put $J=(I:x)$ and $j=j_x(I)$. Two inclusions are automatic, namely
\[
xJ=x(I:x)\subseteq I\subseteq(I:x)=J .
\]
So $I$ is sandwiched between $xJ$ and $J$, and only its image in $J/xJ$ has to be determined.

Next, $I+(x)=(x,y^j)$ by \cref{def:contact}, so $I$ contains an element of the form $y^j+xc$ with $c\in R$. The ideal $xJ+R(y^j+xc)$ is contained in $I$, and a colength count based on \eqref{eq:residual} shows that it is all of $I$. Thus $I$ is recovered from $J$, from $j$, and from a single \emph{tail} $c$. The tail is determined modulo $J$. Replacing $c$ by $c+a$ with $a\in J$ changes $y^j+xc$ by $xa\in xJ\subseteq I$. And the tail is subject to exactly one condition, $y^j+xc\in J$, which expresses $I\subseteq J$.

Writing $X$ for multiplication by $x$ on $R/J$, the condition reads
\[
X\bar c=-\bar y^{\,j}\qquad\text{in }R/J,
\]
an inhomogeneous linear equation for $\bar c\in R/J$. Its solution set is either empty or a torsor under $\ker X=(J:x)/J$, whose dimension is $j_x(J)$ by \cref{lem:contact}(a), and it is non-empty exactly when $\bar y^{\,j}\in X(R/J)$, that is, when $y^j\in J+(x)$, that is, when $j\ge j_x(J)$. The theorem makes this precise and shows conversely that every solution produces an ideal with the prescribed $J$ and $j$.

\begin{theorem}[elementary modification along $\{x=0\}$]\label{thm:A}
Let $J\subseteq R$ be an ideal of finite colength and $j\ge0$. Put
\[
F(J,j):=\{I\subseteq R:\ (I:x)=J,\ I+(x)=(x,y^j)\},\qquad
\cA(J,j):=\{\bar c\in R/J:\ x\,\bar c=-\bar y^{\,j}\ \text{in }R/J\}.
\]
\begin{enumerate}[label=(\alph*)]
\item $F(J,j)\ne\varnothing$ if and only if $j\ge j_x(J)$.
\item If $j\ge j_x(J)$, then $\cA(J,j)$ is an affine space over $\kk$ of dimension $j_x(J)$, a torsor
under the vector space $(J:x)/J$. The map
\[
\cA(J,j)\longrightarrow F(J,j),\qquad \bar c\longmapsto I_c:=xJ+R\,(y^j+xc),
\]
is a well-defined bijection. Here $c\in R$ is any lift of $\bar c$.
\item Every ideal $I$ of finite colength lies in exactly one $F(J,j)$, namely for $J=(I:x)$ and
$j=j_x(I)$. Moreover $\col(I)=\col(J)+j_x(I)$.
\item For $f\in R$ one has $xf\in I$ if and only if $f\in(I:x)$. In particular, for $u\ge1$,
$x^uy^v\in I$ if and only if $x^{u-1}y^v\in(I:x)$.
\end{enumerate}
\end{theorem}

\begin{proof}
\emph{Step 1. The colon of $I_c$.} Let $c\in R$ be arbitrary. We claim
\begin{equation}\label{eq:colonIc}
(I_c:x)=J+R\,(y^j+xc),\qquad I_c+(x)=(x,y^j).
\end{equation}
The inclusion $\supseteq$ in the first equation is clear, since $xJ\subseteq I_c$ and $x(y^j+xc)\in I_c$. Conversely, let $xf\in I_c$ and write $xf=xa+r(y^j+xc)$ with $a\in J$ and $r\in R$. Reducing modulo $x$ gives $ry^j\in(x)$. Since $x$ is a prime element of the factorial ring $R$ and $x\nmid y^j$, we get $r=xr'$ for some $r'\in R$, and cancelling $x$ gives $f=a+r'(y^j+xc)$. The second equation is clear, because $xJ$ and $xc$ lie in $(x)$.

\emph{Step 2. Elements of $\cA(J,j)$ give elements of $F(J,j)$.} The condition $\bar c\in\cA(J,j)$ means $y^j+xc\in J$. Then \eqref{eq:colonIc} gives $(I_c:x)=J$ and $I_c+(x)=(x,y^j)$, i.e.\ $I_c\in F(J,j)$.

\emph{Step 3. The map is well defined and injective.} Let $c,c'$ satisfy $y^j+xc,\,y^j+xc'\in J$. If $c-c'\in J$, then $x(c-c')\in xJ\subseteq I_c$, so $y^j+xc'=(y^j+xc)-x(c-c')\in I_c$ and $I_{c'}\subseteq I_c$. By symmetry $I_c=I_{c'}$. Conversely, if $I_c=I_{c'}$, then $x(c-c')=(y^j+xc)-(y^j+xc')\in I_c$, so $c-c'\in(I_c:x)=J$.

\emph{Step 4. The map is surjective.} Let $I\in F(J,j)$. Since $y^j\in I+(x)$, there is $c\in R$ with $y^j+xc\in I$. As $I\subseteq(I:x)=J$, we get $y^j+xc\in J$, i.e.\ $\bar c\in\cA(J,j)$. Moreover $xJ=x(I:x)\subseteq I$, hence $I_c\subseteq I$. By \cref{lem:colon} and Step 2, $\col(I_c)=\col(I_c:x)+\col(I_c+(x))=\col(J)+j$, and in the same way $\col(I)=\col(J)+j$. An inclusion of ideals of the same finite colength is an equality, so $I=I_c$.

\emph{Step 5. Part (a) and the dimension.} $\cA(J,j)\ne\varnothing$ if and only if $\bar y^{\,j}$ lies in the image of $X$ on $R/J$, i.e.\ $y^j\in J+(x)$, i.e.\ $j\ge j_x(J)$ by \cref{lem:contact}(d). By Steps 2--4 this is equivalent to $F(J,j)\ne\varnothing$. When it is non-empty, $\cA(J,j)$ is a coset of $\ker(X\,|\,R/J)=(J:x)/J$, which has dimension $j_x(J)$ by \cref{lem:contact}(a).

\emph{Step 6. Parts (c) and (d).} For any $I$ put $J=(I:x)$ and $j=j_x(I)$. Then $I+(x)=(x,y^{j})$, so $I\in F(J,j)$. Uniqueness is clear from the definition. The colength formula is \cref{lem:colon} with $f=x$. Part (d) is the definition of the colon ideal.
\end{proof}

Three features of \cref{thm:A} are used repeatedly.

\emph{Uniformity.} For \emph{every} ideal $J$ and \emph{every} $j\ge j_x(J)$ the fibre $F(J,j)$ is an affine space of dimension $j_x(J)$, a torsor under $(J:x)/J$, independently of all other invariants of $J$, such as its colength, the Jordan types of $X$ and $Y$ on $R/J$, or its position in the Hilbert scheme. This is what makes the recursion of \cref{subsec:xchain} close on the pair (colength, contact length). To know the classes of the ideals $I$ obtained from a family of ideals $J$ it suffices to know the classes of the $J$, refined by contact length.

\emph{Linearity.} The tail is a solution of a linear equation on the module $R/J$. When $J$ varies, $R/J$ is the fibre of the tautological bundle, and the equation becomes an equation between sections of vector bundles. This is why the bijection of \cref{thm:A} is an isomorphism of varieties in families (\cref{prop:family}).

\emph{Compatibility with the curve.} By part (d), $I\mapsto(I:x)$ maps $\Hilb(C_{u,v},0)$ to $\Hilb(C_{u-1,v},0)$, and $F(J,j)\subseteq\Hilb(C_{u,v},0)$ as soon as $J\in\Hilb(C_{u-1,v},0)$. Whether $I$ contains $x^uy^v$ depends only on $J$, not on the tail. More generally, for any $g\in R$ one has $xg\in I$ if and only if $g\in J$. This is used in \cref{cor:F}.

\emph{Relation with standard bases and with Luan's filtration.} Descriptions of punctual Hilbert schemes by standard bases go back to Brian\c{c}on \cite{Briancon} and Iarrobino \cite{Iarrobino}. Luan's vertical filtration of an ideal $I$ is $W^k=I\cap(y^k)=y^k(I:y^k)$ \cite[Def.~4.1]{Luan}, and \cite[Lem.~5.1]{Luan} shows that $W^k$ is obtained from $W^{k+1}$ by adding one generator $y^kx^{i}+y^{k+1}\vartheta$. After division by $y^k$ this is the statement $(I:y^k)=y(I:y^{k+1})+R(x^i+y\vartheta)$. So the vertical filtration is the iteration of the mirror image of \cref{thm:A}, in which $x$ and $y$ are exchanged. What the arguments below rely on is the uniform description of the fibres.

\begin{example}[the first fibres]\label{ex:nodeA}
\begin{enumerate}[label=(\arabic*),leftmargin=2em]
\item Take $J=R$. Then $j_x(J)=0$, $\cA(R,j)$ is a point, and $F(R,j)=\{(x,y^j)\}$ for every $j\ge0$. Indeed,
an ideal $I$ with $(I:x)=R$ contains $x$, and is then determined by its contact length.
\item Take $J=\mm$. Then $j_x(J)=1$ and $(J:x)/J=R/\mm=\kk$. For $j\ge1$ the equation
$x\bar c=-\bar y^{\,j}$ holds for every $\bar c\in R/\mm=\kk$, and
\[
F(\mm,j)=\{(x^2,\,xy,\,y^j+c_0x):\ c_0\in\kk\}\cong\AAA^1 .
\]
For $j=1$ these are the ideals $(y+c_0x,\,x^2)$. Together with $F(R,2)=\{(x,y^2)\}$ they exhaust the ideals of colength $2$, and exhibit $\Hilb_2(\AAA^2,0)\cong\PP^1$ as an affine line and a point.
\item Take $J=(y,x^a)$ with $a\ge1$. Then $j_x(J)=1$, and $(J:x)/J$ is spanned by $x^{a-1}$. For $j\ge1$
the theorem gives
\[
F(J,j)=\{(xy,\,x^{a+1},\,y^j+c_0x^a):c_0\in\kk\}\cong\AAA^1 .
\]
These, together with the ideals $(x,y^j)$ of (1), are all the ideals containing $xy$, because $xy\in I$ if and only if $y\in(I:x)$, and the ideals containing $y$ are $R$ and the $(y,x^a)$. Summing,
\[
\sum_{n}[\Hilb_n(C_{1,1},0)]Q^n=\sum_{j\ge0}Q^j+\LL\sum_{a,j\ge1}Q^{a+j}=\frac{1-Q+Q^2\LL}{(1-Q)^2},
\]
so $[\Hilb_n(C_{1,1},0)]=1+(n-1)\LL$ for $n\ge1$. This agrees with the classical description of $\Hilb_n(\{xy=0\},0)$, $n\ge2$, as a chain of $n-1$ projective lines \cite[Ex.~3.2]{Luan}, \cite{KGS}.
\end{enumerate}
\end{example}

\subsection{The modification in families}\label{subsec:families}
The description of \cref{thm:A} is algebraic in families of ideals of constant contact length (\cref{prop:family}). This is what allows us to compute classes in $\Kvar$. Over such a family, the ideals obtained by an elementary modification form a Zariski-locally trivial affine bundle. The main tool is an explicit description of the module $R/I$ in terms of $R/(I:x)$ and the tail (\cref{lem:basis}).

\Cref{thm:A} is a statement about points. Over every field $K\supseteq\kk$ it gives a bijection $\cA(J,j)\to F(J,j)$. To compute classes in $\Kvar$ we need more, namely that this bijection is an isomorphism of varieties when $J$ moves in a family. Two things have to be checked.

First, the forward map $(J,\bar c)\mapsto I_c$ has to be a morphism. A morphism to the Hilbert scheme is the same as a family of cyclic modules (\cref{prop:adhm,lem:points}), so we must write down the module $R/I_c$, with its operators and its cyclic vector, in terms of data that vary algebraically with $(J,\bar c)$. \Cref{lem:basis} does this. It says that $R/I_c$ is the module $x\cdot R/J\cong R/J$ extended by the $j$ vectors $1,y,\dots,y^{j-1}$, and that the operators are given by explicit formulas in which only the operators of $R/J$, its cyclic vector and the tail $\bar c$ appear. The only formula involving the tail, $Y(y^{j-1})=-\iota(\bar c)$, is the relation $y^j\equiv-xc$ modulo $I_c$.

Second, the inverse map $I\mapsto\big((I:x),\bar c(I)\big)$ has to be a morphism. Here we use that the submodule $x\cdot R/I\subseteq R/I$ is isomorphic to $R/(I:x)$ and is generated by the class of $x$ (\cref{lem:rankdict}). On a locus where its dimension is constant it is a subbundle of the tautological bundle, so it defines a family of cyclic modules, that is, a morphism $I\mapsto(I:x)$. The tail is then read off from the relation $x\bar c=-\bar y^{\,j}$ in $R/I$.

Both constructions need certain ranks to be constant. The rank of $X$ on $R/J$ is $\col(J)-j_x(J)$, so $\ker X$ and $\operatorname{Im}X$ are vector bundles exactly over families of constant contact length. On the Hilbert scheme of colength $a+j$ the locus $\{j_x(I)=j\}$ is the locus $\{\rk X=a\}$. This is why \cref{prop:family} is formulated over strata of constant contact length. We first record the explicit module structure.

\begin{lemma}\label{lem:basis}
Let $J$ have finite colength $a$, let $j\ge j_x(J)$, $\bar c\in\cA(J,j)$ and $I=I_c$. Let $X_J,Y_J$ be the operators on $R/J$, $e_J=\bar1$, and let $\iota:R/J\to R/I$, $f\mapsto xf$. Then $\iota$ is injective, $R/I=\iota(R/J)\oplus\langle1,y,\dots,y^{j-1}\rangle$, and
\begin{align*}
X\,\iota(w)&=\iota(X_Jw), & Y\,\iota(w)&=\iota(Y_Jw), && w\in R/J,\\
X\,(y^l)&=\iota(Y_J^l\,e_J), & Y\,(y^l)&=y^{l+1}\ (l<j-1), & Y\,(y^{j-1})&=-\iota(\bar c).
\end{align*}
\end{lemma}
\begin{proof}
The map $\iota$ is well defined and injective because $xf\in I$ if and only if $f\in(I:x)=J$ (\cref{thm:A}(b),(d)), and its image is $xR/I$. The quotient $(R/I)/\iota(R/J)=R/(I+(x))=\kk[y]/(y^j)$ has basis the images of $1,y,\dots,y^{j-1}$, which gives the direct sum decomposition. The formulas follow from the definitions. For $w=\bar f$ one has $X\iota(w)=x\cdot xf=\iota(\overline{xf})$ and $Y\iota(w)=y\cdot xf=\iota(\overline{yf})$. Next $X(y^l)=x\,y^l=\iota(\overline{y^l})=\iota(Y_J^le_J)$. For $l<j-1$, $Y(y^l)=y^{l+1}$ is again a basis vector. Finally $Y(y^{j-1})=y^j=(y^j+xc)-xc\equiv-xc=-\iota(\bar c)$ modulo $I$, because $y^j+xc\in I$.
\end{proof}

With the module structure in hand we state and prove the family version of \cref{thm:A}. In the statement, $\mathcal T$ is the tautological bundle on the Hilbert scheme, with fibre $R/J$ over $J$, and $X,Y,e$ are its operators and its cyclic section (\cref{subsec:hilb}).

\begin{proposition}\label{prop:family}
Let $\kk$ be any field, let $a,j'\ge0$ and $j\ge j'$, and let $S\subseteq\Hilb_a(\AAA^2,0)$ be a locally closed subvariety such that $j_x(J)=j'$ for all $J\in S$. Put
\[
F_S:=\{I\in\Hilb_{a+j}(\AAA^2,0):\ (I:x)\in S,\ j_x(I)=j\},\qquad
\cA_S:=\{(J,\bar c):\ J\in S,\ \bar c\in\cA(J,j)\}.
\]
Here $\cA_S$ is the closed subscheme of the total space of $\mathcal T|_S$ defined by $X\bar c=-Y^je$.
\begin{enumerate}[label=(\alph*)]
\item $F_S$ is locally closed in $\Hilb_{a+j}(\AAA^2,0)$.
\item $\cA_S\to S$ is a Zariski-locally trivial torsor under the vector bundle $\ker X\subseteq
\mathcal T|_S$ of rank $j'$. In particular $\cA_S$ is reduced.
\item There is an isomorphism $\varphi:\cA_S\to F_S$ which on points is
$(J,\bar c)\mapsto I_c$. Its inverse is $I\mapsto((I:x),\bar c(I))$, where $\bar c(I)$ is the class of any $c$ with $y^j+xc\in I$.
\item $[F_S]=\LL^{j'}[S]$ in $\Kvar$. If $S\subseteq\Hilb_a(\{x^{u-1}y^v=0\},0)$, then
$F_S\subseteq\Hilb_{a+j}(C_{u,v},0)$.
\end{enumerate}
The same holds for the mirror modification $I\mapsto(I:y)$, with $j$ replaced by $i$.
\end{proposition}

\begin{proof}
If $j=0$, then $j'=0$, so $S\subseteq\{R\}$ by \cref{lem:contact}(c), and $F_S\subseteq\{R\}$ for the same reason. All assertions are then trivial. We assume $j\ge1$ from now on. The proof follows the plan explained at the beginning of this subsection. Steps 1 and 2 show that $\cA_S$ is a torsor, Step 3 shows that $I\mapsto(I:x)$ is a morphism, Steps 4 and 5 construct morphisms $\varphi:\cA_S\to F_S$ and $\psi:F_S\to\cA_S$, from \cref{lem:basis} and from the relation $y^j+xc\in I$ respectively, and Step 6 shows that they are inverse to each other.

\emph{Step 1. Subbundles.} On $S$ the endomorphism $X$ of $\mathcal T|_S$ has constant rank $a-j'$ (\cref{lem:contact}(a)). Its cokernel is a coherent sheaf with fibres of constant dimension $j'$ on the reduced scheme $S$, hence locally free \cite[Ex.~II.5.8]{Hartshorne}. Therefore $\operatorname{Im}X$ is locally a direct summand of $\mathcal T|_S$, locally free of rank $a-j'$, and $X:\mathcal T|_S\to\operatorname{Im}X$ is a surjection onto a locally free sheaf, so $\ker X$ is locally a direct summand of rank $j'$.

\emph{Step 2. Part (b).} Let $\sigma=-Y^je\in\Gamma(S,\mathcal T|_S)$. At every point $J$ of $S$ the value $\sigma(J)$ lies in $\operatorname{Im}X$, because $y^j\in J+(x)$ by \cref{lem:contact}(d). So the image of $\sigma$ in $\operatorname{coker}X$ vanishes at every point, and it is zero by \cref{lem:reduced}(c). That is, $\sigma$ is a section of $\operatorname{Im}X$. Locally on $S$ choose a complement $\mathcal K'$ of $\ker X$. Then $X|_{\mathcal K'}:\mathcal K'\to\operatorname{Im}X$ is a morphism of locally free sheaves of the same rank that is an isomorphism on every fibre, hence an isomorphism, and $\cA_S=(X|_{\mathcal K'})^{-1}(\sigma)+\ker X$ locally. Trivializing $\ker X$ shows that $\cA_S$ is locally isomorphic to $S\times\AAA^{j'}$, as a scheme. Hence $\cA_S$ is reduced and is a torsor under $\ker X$.

\emph{Step 3. Part (a), and $I\mapsto(I:x)$ as a morphism.} Let $Z_j\subseteq\Hilb_{a+j}(\AAA^2,0)$ be the locally closed subvariety $\{j_x(I)=j\}=\{\rk X=a\}$. On $Z_j$ let $(\mathcal T',X',Y',e')$ be the tautological data. As in Step 1, $\operatorname{Im}X'$ is a subbundle of rank $a$, stable under $X',Y'$, and it is generated by $X'e'$ as an $\mathcal O[X',Y']$-module, because $e'$ generates $\mathcal T'$. At a point $I$ the fibre is $xR/I\cong R/(I:x)$ (\cref{lem:rankdict}), which is killed by $\mm^a$ (\cref{lem:mn}). By \cref{lem:reduced}(c) all monomials of degree $a$ in $X',Y'$ vanish on $\operatorname{Im}X'$.

By \cref{lem:points}(b) the data $(\operatorname{Im}X',X',Y',X'e')$ define a morphism $\rho:Z_j\to\Hilb_a(\AAA^2,0)$, which on points is $I\mapsto(I:x)$, since the annihilator of $X'e'=\bar x$ is $(I:x)$. Then $F_S=\rho^{-1}(S)$ is locally closed in $Z_j$, hence in $\Hilb_{a+j}(\AAA^2,0)$. By construction there is an isomorphism $\mu:\rho^*\mathcal T\to\operatorname{Im}X'$ which intertwines $X,Y$ with $X',Y'$ and sends $e$ to $X'e'$. On fibres it is $R/(I:x)\to R/I$, $f\mapsto xf$.

\emph{Step 4. The morphism $\varphi$.} On $\cA_S$ consider the pulled-back data $(\mathcal T,X_J,Y_J,e_J)$, the tautological section $\bar c$ of $\mathcal T$, with $X_J\bar c=-Y_J^je_J$, and put $\mathcal V=\mathcal T\oplus\mathcal O^j$, with basis $\varepsilon_0,\dots,\varepsilon_{j-1}$ of $\mathcal O^j$. Define $X,Y$ on $\mathcal V$ by the formulas of \cref{lem:basis}, with $\iota$ the inclusion of $\mathcal T$ and $y^l$ replaced by $\varepsilon_l$, that is,
\[
X|_{\mathcal T}=X_J,\quad Y|_{\mathcal T}=Y_J,\quad X\varepsilon_l=Y_J^le_J,\quad
Y\varepsilon_l=\varepsilon_{l+1}\ (l<j-1),\quad Y\varepsilon_{j-1}=-\bar c .
\]
These operators commute. On $\mathcal T$ this holds because $X_J$ and $Y_J$ commute. On $\varepsilon_l$ with $l<j-1$ both composites give $Y_J^{l+1}e_J$. On $\varepsilon_{j-1}$ we have $XY\varepsilon_{j-1}=-X_J\bar c=Y_J^je_J=YX\varepsilon_{j-1}$. The section $e=\varepsilon_0$ generates $\mathcal V$. By Nakayama's lemma it suffices to check this on fibres, where $Y^le=\varepsilon_l$ and $X^pY^qXe=X_J^pY_J^qe_J$, and $e_J$ generates $\mathcal T$.

On each fibre $X$ and $Y$ are nilpotent, since $X$ maps $\mathcal V$ into $\mathcal T$ and $Y^j$ maps $\mathcal O^j$ into $\mathcal T$. So each fibre is a cyclic module $K[[x,y]]/I$ of length $a+j$ (\cref{prop:adhm}), killed by $\mm^{a+j}$ (\cref{lem:mn}). As $\cA_S$ is reduced, all monomials of degree $a+j$ in $X,Y$ vanish on $\mathcal V$ (\cref{lem:reduced}(c)). By \cref{lem:points}(b) we obtain a morphism $\cA_S\to\Hilb_{a+j}(\AAA^2,0)$. By \cref{lem:basis} it sends a point $(J,\bar c)$ to $I_c$, so by \cref{thm:A} its image lies in $F_S$, and it factors through a morphism $\varphi:\cA_S\to F_S$ (\cref{lem:reduced}(b)).

\emph{Step 5. The morphism $\psi$.} On $F_S$ we have $\rho:F_S\to S$ and $\mu$ from Step 3. The section $Y'^je'$ of $\mathcal T'$ lies in $\operatorname{Im}X'$ at every point, because $y^j\in I+(x)$. By \cref{lem:reduced}(c) it is a section of $\operatorname{Im}X'$. Put $\bar c:=-\mu^{-1}(Y'^je')\in\Gamma(F_S,\rho^*\mathcal T)$. Then $\mu(X\bar c)=-X'Y'^je'=-Y'^jX'e'=\mu(-Y^je)$, so $X\bar c=-Y^je$ and $(\rho,\bar c)$ is a morphism $\psi:F_S\to\cA_S$. On points, $\mu(\bar c)=x\bar c=-\bar y^{\,j}$ in $R/I$, so $\bar c$ is the class of any $c$ with $y^j+xc\in I$.

\emph{Step 6. Parts (c) and (d).} On $K$-points, for every field $K\supseteq\kk$, $\psi\circ\varphi$ and $\varphi\circ\psi$ are the identity, by Steps 3 and 4 of the proof of \cref{thm:A} over $K$. Since $\cA_S$ and $F_S$ are reduced and the Hilbert schemes are separated, \cref{lem:reduced}(a) shows that both composites are identities. Hence $\varphi$ is an isomorphism, and $[F_S]=[\cA_S]=\LL^{j'}[S]$. The last assertion is \cref{thm:A}(d). The mirror statement follows by exchanging $x$ and $y$.
\end{proof}

In other words, over a stratum $S$ of constant contact length $j'$, the ideals obtained from $S$ by an elementary modification of contact length $j\ge j'$ form a Zariski-locally trivial affine bundle of rank $j'$ over $S$. This is all that is needed for the multiplicativity $[F_S]=\LL^{j'}[S]$ in $\Kvar$, on which the $x$-chain below rests. In \cref{sec:shadow} the mirror statement, for $I\mapsto(I:y)$, is combined with a finer analysis of the fibres.

\subsection{The $x$-chain}\label{subsec:xchain}
Iterating the elementary modification along the branch $\{x=0\}$ gives \cref{thm:B}, which expresses the contact-refined generating function of $C_{u,v}$ through that of $C_{u-1,v}$, and hence, after $u$ steps, through the contact distribution $g_v$ of the fat line $\{y^v=0\}$. We also derive the mirror chain, and recover the formulas of Bejleri \cite{Bej} and of Ellingsrud--Str{\o}mme \cite{ES}, and Luan's formula for $u=1$ \cite[Thm.~1.2]{Luan}.

For an ideal $I$ containing $x^uy^v$ consider the chain
\[
I=I^{(u)}\longmapsto I^{(u-1)}=(I:x)\longmapsto I^{(u-2)}=(I:x^2)\longmapsto\cdots\longmapsto I^{(0)}=(I:x^u).
\]
By \cref{thm:A}(d), $I^{(c)}$ contains $x^cy^v$, so $I^{(0)}$ defines a punctual subscheme of the fat line $\{y^v=0\}$. Each step $I^{(c-1)}\mapsto I^{(c)}$ is an elementary modification, and by \cref{thm:A} the only datum of $I^{(c-1)}$ that matters for the fibre over it is its contact length.

The contact lengths can only increase along the chain, $j_x(I^{(0)})\le j_x(I^{(1)})\le\dots\le j_x(I^{(u)})=j_x(I)$, by \cref{thm:A}(a), and $\col(I)=\col(I^{(0)})+j_x(I^{(1)})+\dots+j_x(I^{(u)})$ by \cref{thm:A}(c). So the ideals $I$ with a prescribed fat-line ideal $I^{(0)}$ and prescribed contact lengths form an iterated affine space of dimension $j_x(I^{(0)})+\dots+j_x(I^{(u-1)})$, and in families the class of this set is $\LL^{j_x(I^{(0)})+\dots+j_x(I^{(u-1)})}$ times the class of the possible $I^{(0)}$ (\cref{prop:family}). The generating functions refined by contact length are the natural bookkeeping device for this recursion.

For $u,v,n,j\ge0$ let $\Hilb_n(C_{u,v},0)_{j_x=j}\subseteq\Hilb_n(C_{u,v},0)$ be the locally closed subvariety of ideals $I$ with $j_x(I)=j$, and define the \emph{contact-refined motivic zeta functions}
\[
G^x_{u,v}(Q,z):=\sum_{n,j\ge0}[\Hilb_n(C_{u,v},0)_{j_x=j}]\,Q^nz^j\in\Kvar{}[z][[Q]],\qquad
g_v(Q,z):=G^x_{0,v}(Q,z).
\]
Since $j_x(I)\le n$, the coefficient of each $Q^n$ is a polynomial in $z$. By \cref{lem:contact}(c), $g_v$ is a polynomial in $z$ of degree at most $v$ with coefficients in $\Kvar{}[[Q]]$, because the contact strata of the fat line,
\[
\Hilb_n(C_{0,v},0)_{j_x=k}:=\Hilb_n(\{y^v=0\},0)_{j_x=k}=\{I:\ y^v\in I,\ \col I=n,\ I+(x)=(x,y^k)\},
\]
are empty for $k>v$. Thus $g_v(Q,z)=\sum_{k=0}^v\sum_n[\Hilb_n(C_{0,v},0)_{j_x=k}]Q^nz^k$. We also write $\Zs_{u,v}(Q)=G^x_{u,v}(Q,1)=\sum_n[\Hilb_n(C_{u,v},0)]Q^n$.

\begin{theorem}[the $x$-chain]\label{thm:B}
For $u\ge1$ and $v\ge0$, in $\Kvar{}[z][[Q]]$,
\[
G^x_{u,v}(Q,z)=\frac{G^x_{u-1,v}(Q,Q\LL z)}{1-Qz}.
\]
Consequently
\[
G^x_{u,v}(Q,z)=\frac{g_v(Q,(Q\LL)^uz)}{\poch{Qz}{u}},\qquad
\Zs_{u,v}(Q)=\frac{g_v(Q,(Q\LL)^u)}{(Q;Q\LL)_u}.
\]
\end{theorem}
\begin{proof}
Fix $n$ and $j$. By \cref{thm:A}(c),(d), $\Hilb_n(C_{u,v},0)_{j_x=j}$ is the disjoint union over $j'\le j$ of the sets $F_{S_{j'}}$ of \cref{prop:family}, where $S_{j'}=\Hilb_{n-j}(C_{u-1,v},0)_{j_x=j'}$. By \cref{prop:family} these are locally closed, with classes $\LL^{j'}[S_{j'}]$. Hence
\[
[\Hilb_n(C_{u,v},0)_{j_x=j}]=\sum_{j'\le j}\LL^{j'}\,[\Hilb_{n-j}(C_{u-1,v},0)_{j_x=j'}],
\]
and therefore
\begin{align*}
G^x_{u,v}(Q,z)&=\sum_{a,j'}[\Hilb_a(C_{u-1,v},0)_{j_x=j'}]\,Q^a\LL^{j'}\sum_{j\ge j'}(Qz)^j\\
&=\sum_{a,j'}[\Hilb_a(C_{u-1,v},0)_{j_x=j'}]\,Q^a\frac{(Q\LL z)^{j'}}{1-Qz}
=\frac{G^x_{u-1,v}(Q,Q\LL z)}{1-Qz}.
\end{align*}
Iterating $u$ times gives $G^x_{u,v}(Q,z)=G^x_{0,v}(Q,(Q\LL)^uz)\big/\prod_{i=0}^{u-1}(1-Q(Q\LL)^iz)$. Put $z=1$.
\end{proof}

Each step of the chain contributes one factor of Bejleri's series \eqref{eq:bejleri} \cite{Bej}. The $i$-th modification, $i=0,\dots,u-1$, contributes $(1-Q(Q\LL)^i)^{-1}=(1-Q^{i+1}\LL^{i})^{-1}$. After $u$ modifications along $\{x=0\}$ the fat branch $\{x^u=0\}$ has contributed its full factor $1/(Q;Q\LL)_u$, and what remains is the other branch, the fat line $\{y^v=0\}$, with its contact variable evaluated at $(Q\LL)^u$. Thus the whole interaction of the two branches is encoded in the single evaluation $g_v(Q,(Q\LL)^u)$.

\begin{remark}~
\begin{itemize}
    \item Theorem~\ref{thm:B} says the computation of $G^x_{u,v}(Q,z)$ can be reduced to the computation of $G^x_{0,v}(Q,z)$; however, since the contact length $j_x$ is recorded, the same method does not directly lead to a reduction from $G^x_{0,v}(Q,z)$ to $G^x_{0,0}(Q,z)$. Further peeling off the $\{y=0\}$ branch from $C_{0,v}=\{y^v=0\}$ requires a more careful treatment in Section~\ref{sec:shadow}.
    \item Since $\Zs_{u,v}(Q)$ records colength but not the contact length, we have the symmetry $\Zs_{u,v}(Q)=\Zs_{v,u}(Q)$, induced by the automorphism $x\leftrightarrow y$ of $R$, which maps ideals containing $x^uy^v$ to ideals containing $x^vy^u$ and preserves colength. This does not imply a relation between $G^x_{u,v}(Q,z)$ and $G^x_{v,u}(Q,z)$ other than their common $z=1$ specialization. 
\end{itemize}
\end{remark}
Applying Theorem~\ref{thm:B} to both sides of $\Zs_{u,v}(Q)=\Zs_{v,u}(Q)$, we get
\begin{corollary}[the mirror relation]\label{cor:mirror}
For $u,v\ge 0$, we have $g_v(Q,(Q\LL)^u)/(Q;Q\LL)_u=g_u(Q,(Q\LL)^v)/(Q;Q\LL)_v$.
\end{corollary}
This gives relations between various specializations of $g_v(Q,-)$, $v\ge 0$. In Section~\ref{sec:fatline}, these will be exploited to determine $g_v(Q,z)$ up to motives annihilated by a power of $\LL$.

\begin{example}[Bejleri and Ellingsrud--Str{\o}mme]\label{ex:bejleri}
For $u=0$, \cref{thm:B} and \cref{cor:mirror} give $\Zs_{0,v}(Q)=g_v(Q,1)=g_0(Q,(Q\LL)^v)/(Q;Q\LL)_v$. The only ideal containing $1$ is $R$, so $g_0=1$ and $\Zs_{0,v}(Q)=1/(Q;Q\LL)_v$ in $\Kvar{}[[Q]]$. This is Bejleri's formula \eqref{eq:bejleri} \cite{Bej}, obtained by peeling $y$ off $v$ times. For fixed $n$ and $N\ge n$ one has $\Hilb_n(\{y^N=0\},0)=\Hilb_n(\AAA^2,0)$ by \cref{lem:mn}. Letting $N\to\infty$ gives $\sum_n[\Hilb_n(\AAA^2,0)]Q^n=\prod_{i\ge1}(1-Q^i\LL^{i-1})^{-1}$, the class version of the formula of Ellingsrud and Str{\o}mme \cite{ES}.
\end{example}

\begin{example}[$u=1$]\label{ex:u1}
The ideals containing $y$ are $R$ and $(y,x^a)$, $a\ge1$, and $(y,x^a)+(x)=\mm$. So the contact distribution of the fat line $\{y=0\}$ is $g_1(Q,z)=G^x_{0,1}(Q,z)=1+\frac{Qz}{1-Q}$. The $x$-chain for $C_{v,1}$, which peels the branch $\{x=0\}$ $v$ times off $C_{v,1}$ and lands on $\{y=0\}$, gives $\Zs_{v,1}(Q)=G^x_{v,1}(Q,1)=g_1(Q,(Q\LL)^v)/(Q;Q\LL)_v=\frac1{(Q;Q\LL)_v}\big(1+\frac{Q^{v+1}\LL^v}{1-Q}\big)$, and $\Zs_{1,v}(Q)=\Zs_{v,1}(Q)$ by the swap of $x$ and $y$. This is \cite[Thm.~1.2]{Luan}. For $v=1$ it is the series computed by hand in \cref{ex:nodeA}(3).
\end{example}

For $u\ge2$, \cref{thm:B} reduces everything to the contact distribution $g_v$ of a fat line, or, via \cref{cor:mirror}, to $g_u$. The difficulty of computing $g_u$ directly is analysed in \cref{subsec:whyhard}. The next section determines $g_v$ without computing it directly.

\section{The bootstrap and the closed form for the fat line}\label{sec:fatline}

By \cref{thm:B}, the generating function $\Zs_{u,v}(Q)$ of the fat node is determined by the contact distribution $g_v(Q,z)=G^x_{0,v}(Q,z)$ of the fat line $\{y^v=0\}$. This section determines $g_v$, and hence proves the Hilbert-side formula, without describing a single ideal containing $y^v$ other than $R$. The idea (\cref{thm:C}) is that the curve $C_{u,v}$ can be peeled along either branch. \Cref{thm:B} and its mirror give two expressions for the same series, and these identities, for all pairs $u<v$, determine every $g_v$ by Lagrange interpolation. We then guess the solution from computer experiments (\cref{subsec:found}) and verify it by the $q$-binomial theorem (\cref{thm:D}); its symmetric form is the rank formula of \cref{thmA}(a).

Interpolation requires dividing by the differences of the interpolation nodes, which are powers of $Q\LL$ times units. We therefore work in the localization
\[
\Lambda:=\Kvar{}[\LL^{-1}]
\]
and in the rings $\Lambda[[Q]]$ and $\Lambda[[Q]][z]$. Since $\LL$ is a zero divisor in $\Kvar$ \cite{Borisov}, the results of this section determine $g_v$ in $\Lambda[[Q]][z]$, that is, up to coefficients annihilated by a power of $\LL$. They determine the point counts exactly, because the counting measure $\LL\mapsto q$ inverts $\LL$ (\cref{thm:main,thm:refined}). The classes themselves are identified in \cref{sec:shadow}, where it is shown that they lie in $\ZZ[\LL]$; once this is known, the formulas of this section hold exactly in $\Kvar$ (\cref{cor:gv-pieces}).

\subsection{Why we do not compute $g_v$ directly}\label{subsec:whyhard}
The condition that an ideal contains $x^u$ is a polynomial condition of degree $u$ in the tails of its generators, and this is what makes a direct computation of the contact distribution of a fat line difficult.

Luan's vertical method \cite[\S10]{Luan} is the mirror chain of \cref{cor:mirror}. It peels $y$ off $v$ times and is then left with the fat line $\{x^u=0\}$, i.e.\ with $g_u$. Computing $g_u$ directly means describing, in the coordinates given by the tails of the generators, which ideals contain $x^u$, and how this interacts with the contact length. If the ideal contains $q_2=x^2+\alpha x+\beta$ and $s_w=xy^w+y^w\sigma$ with $\alpha,\beta,\sigma\in\kk[[y]]$ and $w\le v$, then $x^uy^v\equiv y^v\Phi_u(\alpha,\beta,\sigma)$ modulo the ideal, where $\Phi_u$ is the remainder of $(-\sigma)^u$ modulo $\sigma^2-\alpha\sigma+\beta$, a polynomial of degree $u$ in the tails.

For $u=1$ the membership condition is linear and for $u=2$ it can be solved for $\beta$, which is why closed formulas are available in these cases. For $u\ge3$ it is a system of degree $u$ mixing the tails of different generators, and Luan remarks that the analysis then splits into a large number of cases \cite[p.~39]{Luan}. The bootstrap below avoids the computation altogether.

\subsection{The bootstrap}\label{subsec:bootstrap}
The contact distributions of the fat lines are determined by three properties, the essential one being the compatibility of the two ways of peeling the fat node (\cref{thm:C}). Recall that $(Q;Q\LL)_v$ is a unit of $\Kvar{}[[Q]]$, with constant term $1$.

\begin{theorem}[bootstrap]\label{thm:C}
Let $(h_v(z))_{v\ge0}$ be a family of polynomials with coefficients in $\Kvar{}[[Q]]$ such that
\begin{enumerate}[label=(\roman*)]
\item $h_0(z)=1$;
\item $\deg_z h_v\le v$ and $h_v(0)=1$;
\item $(Q;Q\LL)_v\,h_v\big(Q,(Q\LL)^u\big)=(Q;Q\LL)_u\,h_u\big(Q,(Q\LL)^v\big)$ for all $u,v\ge 0$.
\end{enumerate}
Then the images of $h_v(z)$ in $\Lambda[[Q]][z]$ are uniquely determined by (i)--(iii). Equivalently, if $(h'_v(z))$ is a second such family, then every coefficient of $h_v(z)-h'_v(z)$ is annihilated by a power of $\LL$. The family $h_v(z)=g_v(Q,z)$ of contact distributions of the fat lines satisfies (i)--(iii).
\end{theorem}
\begin{proof}
\emph{The contact distributions satisfy (i)--(iii).} Let $h_v(z)=g_v(Q,z)$. For (i), the only ideal containing $1$ is $R$. For (ii), $\deg_z h_v\le v$ by \cref{lem:contact}(c), and the coefficient of $z^0$ is the class of the set of ideals with $j_x=0$, which is $\{R\}$. For (iii), both sides equal $(Q;Q\LL)_{u,v}\,\Zs_{u,v}(Q)$ by \cref{thm:B,cor:mirror}.

\emph{Uniqueness.} We argue by induction on $v$, the case $v=0$ being (i). Suppose that $h_0(z),\dots,h_{v-1}(z)$ are known in $\Lambda[[Q]][z]$. Condition (iii) with $u=0,\dots,v-1$ determines the values of $h_v(z)$ at the $v$ nodes $z_i=(Q\LL)^{i-1}$, $1\le i\le v$, and (ii) gives $h_v(z_0)=1$ for $z_0=0$. Write $h_v=\sum_{k=0}^{v}c_kz^k$. The values at the $v+1$ nodes $z_0,\dots,z_v$ form a linear system for $(c_0,\dots,c_v)$ whose matrix is the Vandermonde matrix of the nodes, with determinant $\prod_{0\le i<j\le v}(z_j-z_i)$. Each factor is a power of $Q\LL$, or $(Q\LL)^i\big((Q\LL)^d-1\big)$ with $d\ge1$. Since $1-(Q\LL)^d$ is a unit of $\Kvar{}[[Q]]$ and $Q$ is not a zero divisor of $\Kvar{}[[Q]]$, after inverting $\LL$, the determinant is not a zero divisor of $\Lambda[[Q]]$, and the system has at most one solution over $\Lambda[[Q]]$. 
\end{proof}

The idea behind \cref{thm:C} is that the two ways of peeling the \emph{same} curve $C_{u,v}$ give, for each pair $u<v$, a linear relation between the unknown functions $g_u$ and $g_v$. These relations, together with the trivial normalizations, already determine all $g_v$. In particular, no ideal containing $y^v$ other than $R$ is ever described explicitly. The only geometric input is \cref{thm:A}, through \cref{thm:B,cor:mirror}, and the normalization (ii) concerns the ideal $R$ alone.

\begin{example}[computing $g_2$ from \cref{thm:C}]\label{ex:g2}
Write $g_2(Q,z)=1+\alpha z+\beta z^2$ and $P_2:=(Q;Q\LL)_2=(1-Q)(1-Q^2\LL)$. Condition (iii) with $u=0,1$ gives, using $g_0=1$ and $g_1(Q,z)=1+Qz/(1-Q)$ (\cref{ex:u1}),
\[
1+\alpha+\beta=\frac1{P_2},\qquad 1+\alpha\,Q\LL+\beta\,Q^2\LL^2=\frac{(1-Q)\,g_1(Q,Q^2\LL^2)}{P_2}=\frac{1-Q+Q^3\LL^2}{P_2}.
\]
Subtracting $1$ from both equations gives $\alpha+\beta=(Q+Q^2\LL-Q^3\LL)/P_2$ and $Q\LL(\alpha+\beta Q\LL)=(Q^2\LL-Q^3\LL+Q^3\LL^2)/P_2$. Dividing the second equation by $Q\LL$, which is possible in $\Lambda[[Q]]$, and subtracting it from the first, $\beta(1-Q\LL)=Q^2(1-Q\LL)/P_2$, so
\[
\beta=\frac{Q^2}{(1-Q)(1-Q^2\LL)},\qquad\alpha=\frac{Q(1+Q\LL)}{1-Q^2\LL}
\]
in $\Lambda[[Q]]$. As a check in low colength, note that in colength $1$ the only ideal is $\mm$, with $j_x=1$, matching the $Q$-coefficient $1$ of $\alpha$. In colength $2$, the ideals are $(ax+by)+\mm^2$. They have $j_x=1$ if $b\neq0$, a family $\AAA^1$, and $j_x=2$ for $(x,y^2)$, matching the $Q^2$-coefficients $\LL$ of $\alpha$ and $1$ of $\beta$. Finally \cref{thm:B} and \cref{cor:mirror} give $\Zs_{2,v}(Q)=g_2(Q,(Q\LL)^v)/(Q;Q\LL)_v$, and expanding,
\[
(Q;Q\LL)_{2,v}\,\Zs_{2,v}(Q)=(1-Q)(1-Q^2\LL)+(1-Q)(1+Q\LL)Q^{v+1}\LL^v+Q^{2v+2}\LL^{2v},
\]
which is \cite[Thm.~1.4]{Luan} at $T^2=\LL$.
\end{example}

\subsection{How the closed form was found}\label{subsec:found}
\Cref{thm:C} is effective. The interpolation can be carried out over the field $\QQ(Q,\LL)$, with $\LL$ a formal variable, and the uniqueness argument applies verbatim since the nodes $0,1,Q\LL,\dots,(Q\LL)^{v-1}$ are distinct. We did this by computer for $v\le6$. In each case the numerator $(Q;Q\LL)_v\,g_v(Q,z)$ turned out to be a polynomial in $(Q,\LL,z)$ whose coefficients factor completely. For instance
\begin{align*}
(Q;Q\LL)_3\,g_3(Q,z)={}&(1-Q)(1-Q^2\LL)(1-Q^3\LL^2)+Q(1-Q)(1-Q^2\LL)(1+Q\LL+Q^2\LL^2)\,z\\
&+Q^2(1-Q)(1+Q\LL+Q^2\LL^2)\,z^2+Q^3z^3 .
\end{align*}
Here $1+Q\LL+Q^2\LL^2=\qbin{3}{1}{Q\LL}=\qbin{3}{2}{Q\LL}$ and $(1-Q)(1-Q^2\LL)=(Q;Q\LL)_2$, so that
\[
(Q;Q\LL)_3\,g_3(Q,z)=\sum_{k=0}^{3}Q^k\qbin{3}{k}{Q\LL}(Q;Q\LL)_{3-k}\,z^k .
\]
The same pattern holds for all $v\le6$ and suggests \cref{thm:D}. That the guess is correct is then a statement about $q$-series, which we now prove.

\subsection{The closed form}\label{subsec:closedform}
The closed formula for the contact distribution of a fat line (\cref{thm:D}) is proved by verifying the conditions of \cref{thm:C} with the $q$-binomial theorem. Following the convention of \cref{subsec:conventions}, we first work with the formal variables $q,Q$ and substitute $q=Q\LL$ afterwards.

\begin{proposition}\label{prop:symmetric}
For $v\ge0$ put
\begin{gather*}
\widehat g_v(q;Q,z):=\sum_{k=0}^{v}\qbin{v}{k}{q}\frac{(Qz)^k}{(Qq^{v-k};q)_k}\in\ZZ[q][[Q]][z],\\(Q;q)_v\,\widehat g_v(q;Q,z)=\sum_{k=0}^{v}\qbin{v}{k}{q}(Q;q)_{v-k}\,(Qz)^k\in\ZZ[q,Q,z].
\end{gather*}
Then for all $u,v\ge0$
\[
(Q;q)_v\,\widehat g_v(q;Q,q^u)=\sum_{m=0}^{\min(u,v)}(-1)^m\,q^{\binom m2}\,(q;q)_m\qbin{u}{m}{q}\qbin{v}{m}{q}\,Q^m=C_{u,v}(q,Q).
\]
In particular $(Q;q)_v\,\widehat g_v(q;Q,q^u)$ is symmetric in $u$ and $v$.
\end{proposition}
\begin{proof}
The two expressions for $(Q;q)_v\widehat g_v$ agree because $(Q;q)_v=(Q;q)_{v-k}\,(Qq^{v-k};q)_k$. Expand $(Q;q)_{v-k}$ by \cref{lem:qbinomial} and use \cref{lem:trinomial} with $m=k+i$. This gives
\[
(Q;q)_v\,\widehat g_v(q;Q,z)=\sum_{k,i}\qbin{v}{k}{q}\qbin{v-k}{i}{q}(-1)^iq^{\binom i2}Q^{k+i}z^k=\sum_{m=0}^{v}\qbin{v}{m}{q}Q^m\sum_{i=0}^{m}\qbin{m}{i}{q}(-1)^iq^{\binom i2}z^{m-i}.
\]
At $z=q^u$ the inner sum equals
\[
q^{um}\sum_{i=0}^{m}\qbin{m}{i}{q}(-1)^iq^{\binom i2}(q^{-u})^i=q^{um}\,(q^{-u};q)_m=\prod_{i=0}^{m-1}(q^u-q^i),
\]
by \cref{lem:qbinomial} with $Q$ replaced by $q^{-u}$, an identity in $\ZZ[q,q^{-1}]$. This product vanishes for $m>u$, because of the factor $i=u$. For $m\le u$, using $q^u-q^i=-q^i(1-q^{u-i})$ and $\prod_{i=0}^{m-1}(1-q^{u-i})=(q;q)_u/(q;q)_{u-m}$,
\[
\prod_{i=0}^{m-1}(q^u-q^i)=(-1)^mq^{\binom m2}\frac{(q;q)_u}{(q;q)_{u-m}}=(-1)^mq^{\binom m2}(q;q)_m\qbin{u}{m}{q}.
\]
The last equality with $C_{u,v}(q,Q)$ is \cref{lem:rank}.
\end{proof}

\begin{theorem}[closed form of the fat-line distribution]\label{thm:D}
For every $v\ge0$, in $\Lambda[[Q]][z]$,
\[
g_v(Q,z)=\widehat g_v(Q\LL;Q,z)=\sum_{k=0}^{v}\qbin{v}{k}{Q\LL}\frac{(Qz)^k}{\big(Q(Q\LL)^{v-k};Q\LL\big)_k} .
\]
In $\Kvar{}[[Q]][z]$ the two sides differ by a polynomial each of whose coefficients is annihilated by a power of $\LL$; they are in fact equal, by \cref{cor:gv-pieces}.
\end{theorem}
\begin{proof}
We check that the family $\widehat g_v(Q\LL;Q,z)$ satisfies (i)--(iii) of \cref{thm:C}, as identities in $\ZZ[\LL][[Q]][z]$. (i) $\widehat g_0=1$. (ii) $\deg_z\widehat g_v=v$, the term $k=0$ gives $\widehat g_v(Q\LL;Q,0)=1$, and the coefficients lie in $\ZZ[\LL][[Q]]$ after expanding the factors $1/(1-Q(Q\LL)^i)$ as geometric series. (iii) By \cref{prop:symmetric} with $q=Q\LL$, both sides of (iii) equal $C_{u,v}(Q\LL,Q)$. Now apply the uniqueness in \cref{thm:C}.
\end{proof}

Since $q$ and $Q$ are independent variables, \cref{prop:symmetric} contains the following identities, which are used on the Quot side and in \cref{sec:positivity}.

\begin{corollary}[two expansions of $C_{u,v}$]\label{cor:positive-C}
For all $u,v\ge0$, in $\ZZ[q,Q]$ and in $\ZZ[q][[Q]]$ respectively,
\begin{equation}\label{eq:positive-C}
C_{u,v}(q,Q)=\sum_{k=0}^{v}\qbin{v}{k}{q}\,(Q;q)_{v-k}\,Q^kq^{uk},\qquad\frac{C_{u,v}(q,Q)}{(Q;q)_v}=\sum_{k=0}^{v}\qbin{v}{k}{q}\,\frac{Q^kq^{uk}}{(Qq^{v-k};q)_k}\in\NN[q][[Q]].
\end{equation}
\end{corollary}
\begin{proof}
The first identity is \cref{prop:symmetric} at $z=q^u$. Dividing by $(Q;q)_v=(Q;q)_{v-k}(Qq^{v-k};q)_k$ gives the second, whose terms have nonnegative series expansions because $\qbin{v}{k}{q}\in\NN[q]$ (\cref{lem:pascal}).
\end{proof}

\begin{corollary}[\cref{thmA}(a) for point counts]\label{thm:main}
Let $q$ be a prime power and $\kk=\FF_q$. For all $u,v\ge0$,
\[
\sum_{n\ge0}\#\Hilb_n(C_{u,v},0)(\FF_q)\,Q^n=F_{u,v}(qQ,Q)=\frac{C_{u,v}(qQ,Q)}{(Q;qQ)_{u,v}} .
\]
Consequently, writing $M_n(\LL)\in\ZZ[\LL]$ for the coefficient of $Q^n$ in $F_{u,v}(Q\LL,Q)$, one has $\#\Hilb_n(C_{u,v},0)(\FF_q)=M_n(q)$ for every prime power $q$.
\end{corollary}
\begin{proof}
The counting measure $\#:\Kvar[\FF_q]\to\ZZ$ sends $\LL$ to $q$, so it factors through $\Lambda$ and kills every element annihilated by a power of $\LL$. Apply it to $\Zs_{u,v}(Q)=g_v(Q,(Q\LL)^u)/(Q;Q\LL)_u$ (\cref{thm:B}) and use \cref{thm:D} and \cref{prop:symmetric}, which give $(Q;Q\LL)_v\,g_v(Q,(Q\LL)^u)=C_{u,v}(Q\LL,Q)$ in $\Lambda[[Q]]$. The coefficient of $Q^n$ in $F_{u,v}(Q\LL,Q)$ is a polynomial in $\LL$ with integer coefficients, since the numerator and the expansions of $1/(Q;Q\LL)_u$ and $1/(Q;Q\LL)_v$ have coefficients in $\ZZ[\LL]$.
\end{proof}

\begin{corollary}[contact-refined point counts]\label{thm:refined}
Let $q$ be a prime power and $\kk=\FF_q$. For all $u,v\ge0$,
\[
\sum_{I\ni x^uy^v}Q^{\col I}z^{j_x(I)}=\frac{1}{(Qz;qQ)_u}\sum_{k=0}^{v}\qbin{v}{k}{qQ}\frac{\big(Q(qQ)^uz\big)^k}{\big(Q(qQ)^{v-k};qQ\big)_k},
\]
the sum running over the ideals of $\FF_q[[x,y]]$ of finite colength containing $x^uy^v$. In particular, for the fat line $\{y^v=0\}$ and $0\le k\le v$,
\[
\sum_n\#\Hilb_n(C_{0,v},0)_{j_x=k}(\FF_q)\,Q^n=Q^k\qbin{v}{k}{qQ}\prod_{i=v-k}^{v-1}\frac1{1-Q(qQ)^{i}} .
\]
\end{corollary}
\begin{proof}
Combine $G^x_{u,v}(Q,z)=g_v(Q,(Q\LL)^uz)/(Qz;Q\LL)_u$ (\cref{thm:B}) with \cref{thm:D} and apply the counting measure as in the proof of \cref{thm:main}. The second statement is the coefficient of $z^k$ for $u=0$.
\end{proof}

These two corollaries need nothing beyond \cref{thm:A} and $q$-calculus. \Cref{sec:shadow} upgrades them from point counts to classes.

\begin{example}[the numerators for $v\le2$]\label{ex:Nv}
\begin{gather*}
  (Q;Q\LL)_2\,g_2(Q,z)=(1-Q)(1-Q^2\LL)+Q(1-Q)(1+Q\LL)\,z+Q^2z^2 ,\\
  (Q;Q\LL)_1\,g_1(Q,z)=(1-Q)+Q\,z,
\end{gather*}
and $(Q;Q\LL)_3\,g_3(Q,z)$ was displayed in \cref{subsec:found}. For example, the contact strata of $\{y^2=0\}$ have the generating functions $\frac{Q(1+Q\LL)}{1-Q^2\LL}$ for $k=1$ and $\frac{Q^2}{(1-Q)(1-Q^2\LL)}$ for $k=2$, in agreement with \cref{ex:g2}.
\end{example}

\section{Consequences of the formula}\label{sec:consequences}

This section collects consequences of the Hilbert-scheme formula, namely structural properties of the numerator $C_{u,v}(Q\LL,Q)$, the positivity of the coefficients, the Euler-characteristic specialization, which is an identity for partitions avoiding a box, the range of colengths in which the Hilbert scheme of $C_{u,v}$ coincides with that of the plane, and the extension of the $x$-chain to arbitrary germs.

Throughout, $C_{u,v}(Q\LL,Q)$ denotes the polynomial of \eqref{eq:main}, and $M_n(\LL)$ the coefficient of $Q^n$ in $C_{u,v}(Q\LL,Q)/(Q;Q\LL)_{u,v}$. By \cref{thm:torus-all} below, $[\Hilb_n(C_{u,v},0)]=M_n(\LL)$ in $\Kvar$ for every field $\kk$. The statements of this section about $M_n$ therefore apply to classes and to all their realizations.

\subsection{The numerator}
The basic properties of the numerator $C_{u,v}(Q\LL,Q)$ are integrality and symmetry, its degree and leading term, its value at $Q=1$, and its limit as $u,v\to\infty$.

\begin{corollary}\label{cor:structure}
For all $u,v\ge0$ the following hold.
\begin{enumerate}[label=(\arabic*)]
\item $C_{u,v}(Q\LL,Q)\in\ZZ[Q,\LL]$ and $C_{u,v}(Q\LL,Q)=C_{v,u}(Q\LL,Q)$. Moreover $C_{0,v}(Q\LL,Q)=1$,
$C_{1,v}(Q\LL,Q)=1-Q+Q^{v+1}\LL^v$, and $C_{2,v}(Q\LL,Q)=(1-Q)(1-Q^2\LL)+(1-Q)(1+Q\LL)Q^{v+1}\LL^v+Q^{2v+2}\LL^{2v}$.
\item $\deg_QC_{u,v}(Q\LL,Q)=uv+\min(u,v)$, and the monomial of top $Q$-degree is $Q^{uv+\min(u,v)}\LL^{uv}$,
with coefficient $1$.
\item $C_{u,v}(\LL,1)=\LL^{uv}$. If $u,v\ge1$, then $(Q;Q\LL)_{u,v}$ vanishes to order exactly $2$ at
$Q=1$, as a polynomial in $Q$ over $\QQ(\LL)$.
\item As $u,v\to\infty$, coefficientwise in $Q$, $C_{u,v}(Q\LL,Q)\to\poch{Q}{\infty}$ and
$\Zs_{u,v}(Q)\to1/\poch{Q}{\infty}=\prod_{i\ge1}(1-Q^i\LL^{i-1})^{-1}$.
\end{enumerate}
\end{corollary}
\begin{proof}
(1) Integrality and symmetry are visible in \eqref{eq:main}. For $u=0$ only $m=0$ contributes. For $u=1$, $C_{1,v}(Q\LL,Q)=1-Q\poch{Q\LL}{1}\qb{v}{1}=1-Q(1-(Q\LL)^v)$. For $u=2$ see \cref{ex:g2}.

(2) By \cref{lem:rank}, $R_m(u,v;q)$ is a monic polynomial in $q$ of degree $\sum_{i<m}(u+v-m)=m(u+v-m)$. So the $m$-th summand $Q^mR_m(u,v;Q\LL)$ has top $Q$-degree $d(m)=m(u+v+1-m)$, attained by the monomial $Q^{d(m)}\LL^{m(u+v-m)}$ with coefficient $1$. For $0\le m<\min(u,v)$ we have $d(m+1)-d(m)=u+v-2m>0$. Hence the maximum is attained only at $m=\min(u,v)$, where $d=\min(u,v)(\max(u,v)+1)=uv+\min(u,v)$ and the $\LL$-degree is $uv$.

(3) At $Q=1$ the specialization $q=Q\LL$ becomes $q=\LL$, so $C_{u,v}(\LL,1)=\sum_mR_m(u,v;\LL)=\LL^{uv}$ by \cref{cor:allmatrices}. The factor $1-Q$ occurs once in each of $(Q;Q\LL)_u$ and $(Q;Q\LL)_v$, and $1-Q^i\LL^{i-1}$ does not vanish at $Q=1$ for $i\ge2$.

(4) By the proof of \cref{cor:euler-id}, $\qb{u}{m}\to1/\poch{Q\LL}{m}$ coefficientwise in $Q$ as $u\to\infty$. Since the $m$-th term of $C_{u,v}(Q\LL,Q)$ is divisible by $Q^m$, we get \[C_{u,v}(Q\LL,Q)\to\sum_m(-1)^m(Q\LL)^{\binom m2}Q^m/\poch{Q\LL}{m}=\poch{Q}{\infty}\] by \cref{cor:euler-id}, and $\Zs_{u,v}(Q)\to\poch{Q}{\infty}/\poch{Q}{\infty}^2$.
\end{proof}

Since $uv=\dim_{\kk}R/(x^u,y^v)$ is the length of the intersection of the two branches, (3) says that at $Q=1$ the numerator becomes the class of the space $\Mat_{u\times v}$, of dimension $uv$, stratified by rank. The pole of order $2$ at $Q=1$ matches the number of branches, as for the reduced node (\cref{ex:nodeA}(3)).

\subsection{Positivity and Euler characteristics}
All coefficients of the Hilbert series are polynomials in $\LL$ with nonnegative coefficients (\cref{cor:positivity}), and at $\LL=1$ the formula becomes an identity for partitions avoiding a box. One side of that identity has a bijective proof, which we give.

\begin{corollary}[positivity]\label{cor:positivity}
All coefficients of $C_{u,v}(Q\LL,Q)/(Q;Q\LL)_{u,v}$, and of the right-hand side of \cref{thm:refined}, are polynomials in $\LL$ (resp.\ $\LL$ and $z$) with non-negative integer coefficients. In particular $[\Hilb_n(C_{u,v},0)_{j_x=j}]$ is a polynomial in $\LL$ with non-negative coefficients for all $n,j$.
\end{corollary}
\begin{proof}
In \cref{thm:refined} every factor has non-negative coefficients as a series in $Q,\LL,z$. Indeed, the Gaussian binomials lie in $\NN[Q\LL]$ (\cref{lem:pascal}), $(Q(Q\LL)^uz)^k$ is a monomial, and each factor $(1-Q(Q\LL)^iz^{\epsilon})^{-1}$ is a geometric series. Set $z=1$ for the first statement. The last statement follows from \cref{thm:torus-all}.
\end{proof}

\begin{corollary}[Euler characteristics]\label{cor:euler}
For all $u,v\ge0$,
\[
\sum_{\lambda:\ (u,v)\notin\lambda}Q^{|\lambda|}
=\frac{\sum_{m}(-1)^mQ^{\binom{m+1}2}(Q;Q)_m\qbin{u}{m}{Q}\qbin{v}{m}{Q}}{(Q;Q)_u\,(Q;Q)_v}
=\frac1{(Q;Q)_u}\sum_{k=0}^{v}\qbin{v}{k}{Q}\frac{Q^{(u+1)k}}{(Q^{v-k+1};Q)_k},
\]
where the sum on the left runs over all partitions not containing the box $(u,v)$.
\end{corollary}
\begin{proof}
For a complex variety the topological Euler characteristic equals the Euler characteristic with compact supports, which factors through $K_0(\mathrm{Var}_{\CC})$ with $\LL\mapsto1$. So by \cref{thm:torus-all}, $\chi_{\mathrm{top}}(\Hilb_n(C_{u,v},0)_{\CC})=M_n(1)$. On the other hand, a one-parameter subgroup of $\mathbf T$ whose weight separates the finitely many characters of $W_n$ has the monomial ideals as its fixed points on $\Hilb_n(C_{u,v},0)$ (\cref{lem:fixed}), and $\chi_{\mathrm{top}}(Z)=\chi_{\mathrm{top}}(Z^{\GG_m})$ for a complex variety with a $\GG_m$-action \cite{BB}. Hence $M_n(1)$ is the number of partitions of $n$ avoiding $(u,v)$. Now put $\LL=1$, so that $Q\LL=Q$, in \eqref{eq:main} and in \cref{thm:refined} at $z=1$.
\end{proof}

\begin{example}
For $u=v=1$ the partitions avoiding $(1,1)$ are the hooks, with generating function $1+\frac{Q}{(1-Q)^2}=\frac{1-Q+Q^2}{(1-Q)^2}$, which is the middle expression $\frac{1-Q(1-Q)}{(1-Q)^2}$.
\end{example}

The partitions avoiding $(u,v)$ are those contained in the ``fat hook'' $\{(a,b):a<u\ \text{or}\ b<v\}$. The last expression in \cref{cor:euler} has a bijective proof. Since $\qbin{v}{k}{Q}/(Q^{v-k+1};Q)_k=1/(Q;Q)_k$, it equals
\[
\frac1{(Q;Q)_u}\sum_{k=0}^{v}\frac{Q^{(u+1)k}}{(Q;Q)_k}.
\]
In the conventions of \cref{subsec:conventions}, $(u,v)\in\lambda$ if and only if row $v$ has length at least $u+1$. So $\lambda$ avoids $(u,v)$ if and only if at most $v$ of its rows have length $\ge u+1$. If exactly $k$ rows do, they are the first $k$ rows. Removing $u+1$ boxes from each of them leaves a partition with at most $k$ parts, with generating function $1/(Q;Q)_k$, and the remaining rows form a partition with parts of size at most $u$, with generating function $1/(Q;Q)_u$.

Conversely, prefixing $u+1$ boxes to each of at most $k$ rows and appending rows of length $\le u$ gives a partition with exactly $k$ rows of length $\ge u+1$. This proves the identity between the left side and the last expression, with the weight $Q^{(u+1)k}$ accounting for the removed boxes. We do not know a bijective proof of the alternating rank expansion in the middle of \cref{cor:euler}.

\subsection{Where the curve becomes visible}
At which colength does the Hilbert scheme of $C_{u,v}$ begin to differ from that of the plane? The two agree for $n\le u+v$, and the first difference consists of curvilinear ideals. This happens much earlier than the torus-fixed points can detect.

\begin{lemma}[curvilinear ideals]\label{lem:curvilinear}
Let $I$ have colength $n\ge2$ with $\mm^{n-1}\not\subseteq I$. Then exactly one of the following holds:
\begin{enumerate}[label=(\roman*)]
\item $I=(y-f(x),x^n)$ with $f=\sum_{i=1}^{n-1}c_ix^i$ uniquely determined by $I$.
\item $I=(x-g(y),y^n)$ with $g=\sum_{i=2}^{n-1}d_iy^i$ uniquely determined by $I$.
\end{enumerate}
\end{lemma}
\begin{proof}
Put $h_k=\dim(\mm^k+I)/(\mm^{k+1}+I)$. As in the proof of \cref{lem:mn}, the spaces $M_k=(\mm^k+I)/I$ decrease strictly until they vanish, and $\sum_kh_k=n$. Since $M_{n-1}\ne0$, there are $n$ strict inclusions $M_0\supsetneq\dots\supsetneq M_{n-1}\supsetneq M_n=0$ with total drop $n$, so $h_k=1$ for $0\le k\le n-1$. In particular the image of $I$ in $\mm/\mm^2$ is a line, spanned by the linear part of some $\ell=ax+by+(\text{order}\ge2)\in I$.

If $b\ne0$, the Weierstrass preparation theorem gives $\ell=U\cdot(y-\tilde f(x))$ with $U$ a unit and $\tilde f(0)=0$. Then $R/(\ell)\cong \kk[[x]]$, and $I/(\ell)$ is an ideal of colength $n$ of $\kk[[x]]$, i.e.\ $(x^n)$. So $I=(y-f,x^n)$ with $f=\tilde f\bmod x^n$. Two such ideals coincide iff the polynomials $f$ coincide, since $(y-f,x^n)\cap \kk[[x]]=(x^n)$. If $b=0$, the same argument with $x$ and $y$ exchanged gives $I=(x-g(y),y^n)$, and $g$ has no linear term because the image of $I$ in $\mm/\mm^2$ is spanned by $x$. The two cases are exclusive, because in case (i) this image is spanned by $y-c_1x$. In other words, case (ii) consists of the curvilinear ideals whose tangent line is $\{x=0\}$, and case (i) of all the others, the tangent line being $\{y=c_1x\}$.
\end{proof}

\begin{proposition}[threshold]\label{prop:threshold}
Let $u,v\ge1$.
\begin{enumerate}[label=(\alph*)]
\item $\Hilb_n(C_{u,v},0)=\Hilb_n(\AAA^2,0)$ for $n\le u+v$.
\item For $n=u+v+1$, the complement $\Hilb_n(\AAA^2,0)\setminus\Hilb_n(C_{u,v},0)$ is
$\{(y-f(x),x^n):\ f=\sum_{i=1}^{n-1}c_ix^i,\ c_1\ne0\}\cong\GG_m\times\AAA^{n-2}$, the set of curvilinear ideals whose tangent line is neither coordinate axis. Its class is $(\LL-1)\LL^{u+v-1}$. The set is symmetric in $x$ and $y$: by the Weierstrass preparation theorem each of its ideals is also of the form $(x-\tilde f(y),y^n)$ with $\tilde f=\sum_{i=1}^{n-1}\tilde d_iy^i$ and $\tilde d_1=c_1^{-1}\ne0$; it falls under case (i) of \cref{lem:curvilinear} only because that lemma places the ideals with tangent line $\{x=0\}$ in case (ii) and all others in case (i).
\end{enumerate}
\end{proposition}
\begin{proof}
(a) By \cref{lem:mn}, $x^uy^v\in\mm^{u+v}\subseteq\mm^n\subseteq I$.

(b) Let $n=u+v+1$ and $x^uy^v\notin I$. As $x^uy^v\in\mm^{n-1}$, we have $\mm^{n-1}\not\subseteq I$, so $I$ is as in \cref{lem:curvilinear}. In case (ii), $x^uy^v\equiv g(y)^uy^v$ modulo $x-g$, of $y$-order at least $2u+v\ge n$, so $x^uy^v\in I$, a contradiction. In case (i), $x^uy^v\equiv x^uf(x)^v$ has $x$-order $u+v\ord(f)$, which is less than $n$ if and only if $\ord f=1$, i.e.\ $c_1\ne0$.

Conversely, each $(y-f,x^n)$ with $c_1\ne0$ does not contain $x^uy^v$. The family $(c_1,\dots,c_{n-1})\mapsto(y-\sum c_ix^i,x^n)$ is a morphism from $\GG_m\times\AAA^{n-2}$ to the open subset $U\subseteq\Hilb_n(\AAA^2,0)$ on which $1,x,\dots,x^{n-1}$ is a basis of $R/I$. On $U$ the coordinates of $\bar y$ in this basis are regular functions, and on the image they are the $c_i$. So the family is an isomorphism onto the complement, which is bijective on points over every field by \cref{lem:curvilinear}.
\end{proof}

\begin{corollary}\label{cor:jump}
For $u,v\ge1$, as polynomials in $\LL$,
\[
[Q^n]\Big(\frac1{\poch{Q}{\infty}}-\frac{C_{u,v}(Q\LL,Q)}{(Q;Q\LL)_{u,v}}\Big)=
\begin{cases}0,& n\le u+v,\\ (\LL-1)\LL^{u+v-1},& n=u+v+1.\end{cases}
\]
\end{corollary}
\begin{proof}
By \cref{ex:bejleri} and \cref{thm:main}, the left-hand side evaluated at $\LL=q$ counts the $\FF_q$-points of $\Hilb_n(\AAA^2,0)\setminus\Hilb_n(C_{u,v},0)$, which \cref{prop:threshold} computes. Polynomials that agree at all prime powers are equal.
\end{proof}

Monomial ideals detect the curve only from $n=(u+1)(v+1)$ on. By \cref{lem:fixed}, $I_\lambda$ fails to contain $x^uy^v$ exactly when $(u,v)\in\lambda$, and then $\lambda$ contains the whole $(u+1)\times(v+1)$ rectangle. For example, $\Hilb_n(C_{3,6},0)$ differs from $\Hilb_n(\AAA^2,0)$ from $n=10$ on, not from $n=28$ on.

\subsection{Adding a fat branch to an arbitrary germ}
The $x$-chain works for an arbitrary germ $\{g=0\}$: adding a fat branch $\{x^u=0\}$ multiplies the zeta function by $1/(Q;Q\LL)_u$ and evaluates the contact-refined zeta function of $g$ at $(Q\LL)^u$. No transversality is needed, as the example of a smooth branch tangent to the fat line shows.

The proof of \cref{thm:B} uses nothing about the fat line $\{y^v=0\}$.

\begin{corollary}\label{cor:F}
Let $g\in\mm$ and $u\ge0$, and put $G^x_g(Q,z):=\sum_{n,k}[\Hilb_n(\{g=0\},0)_{j_x=k}]\,Q^nz^k$, where $\Hilb_n(\{g=0\},0)_{j_x=k}$ is the locally closed subvariety of ideals with $j_x=k$. Then, in $\Kvar{}[z][[Q]]$,
\[
\sum_{n,j}[\Hilb_n(\{x^ug=0\},0)_{j_x=j}]\,Q^nz^j=\frac{G^x_g(Q,(Q\LL)^uz)}{\poch{Qz}{u}}.
\]
\end{corollary}
\begin{proof}
By \cref{thm:A}(d) with $f=x^{u-1}g$, $x^ug\in I$ if and only if $x^{u-1}g\in(I:x)$. The proof of \cref{thm:B} applies verbatim.
\end{proof}

So adding a fat line $\{x^u=0\}$ as a branch to \emph{any} germ $\{g=0\}$ multiplies the zeta function by the Bejleri factor $1/(Q;Q\LL)_u$ \cite{Bej} and evaluates the contact-refined zeta function of $g$ at $z=(Q\LL)^u$. No transversality is needed, and $g$ may itself contain the branch $\{x=0\}$ or be tangent to it.

\begin{example}[a smooth branch tangent to the fat line]\label{ex:tangent}
Let $g=x-y^m$ with $m\ge1$. Since $R/(g)\cong \kk[[y]]$, the ideals containing $g$ are the $I_b=(x-y^m,y^b)$, $b\ge0$, and $I_b+(x)=(x,y^{\min(b,m)})$. Hence
\begin{gather*}
G^x_g(Q,z)=\sum_{b=0}^{m-1}(Qz)^b+\frac{(Qz)^m}{1-Q},\\
\sum_n[\Hilb_n(\{x^u(x-y^m)=0\},0)]\,Q^n=\frac1{(Q;Q\LL)_u}\Big(\sum_{b=0}^{m-1}Q^b(Q\LL)^{ub}+\frac{Q^m(Q\LL)^{um}}{1-Q}\Big).
\end{gather*}
For $m=1$ the curve is isomorphic to $C_{u,1}$, and we recover \cref{ex:u1}. For $u=1$ the curve $x(x-y^m)$ has two smooth branches with intersection multiplicity $m$. If $\operatorname{char}\kk\ne2$ it is the $A_{2m-1}$ singularity, and the formula becomes $\big(\sum_{b<m}Q^{2b}\LL^b+\frac{Q^{2m}\LL^m}{1-Q}\big)/(1-Q)$.

As a check, take the tacnode $x(x-y^2)$ and $n=3$. The formula gives $[Q^3]=1+\LL$. Directly, $[\Hilb_3(\AAA^2,0)]=1+\LL+\LL^2$ and $f=x^2-xy^2\in\mm^2$. The ideal $\mm^2$ contains $f$. The curvilinear ideals $(x-d_2y^2,y^3)$ contain $f\equiv(d_2^2-d_2)y^4\equiv0$. The curvilinear ideals $(y-c_1x-c_2x^2,x^3)$, which form an $\AAA^2$, do not, since $f\equiv x^2$ modulo them. Hence $[\Hilb_3]=1+\LL+\LL^2-\LL^2=1+\LL$.
\end{example}

\section{Torus stratifications and the motive}\label{sec:shadow}

Here is where the argument stands. \Cref{thm:B} reduces $C_{u,v}$ with the statistic $j_x$ to the fat line $C_{0,v}$ with the statistic $j_x$, and the same method reduces $C_{u,v}$ with all the statistics $j_x((I:x^k))$, $k\ge0$, to $C_{0,v}$ with $j_x$. Applied along the other branch, the method reduces $C_{0,v}$ with the statistic $j_y$ to the trivial input $C_{0,0}$, so $C_{0,v}$ with all the statistics $j_y((I:y^k))$ is understood as well: these are the rows of \cref{subsec:statistics}, and the strata are Bejleri's cells. What is left is $C_{0,v}$ with the statistic $j_x$, the stage reached in \cref{sec:modification}. \Cref{sec:fatline} settles it up to coefficients annihilated by a power of $\LL$, and exactly for point counts, but not in $\Kvar$, where $\LL$ is a zero divisor \cite{Borisov}.

The new input of this section is that $C_{0,v}$ can be understood with all the statistics $j_x((I:y^k))$ and $j_y((I:y^k))$ at the same time. The point is that along the $y$-chain the contact length $j_x((I:y^k))$ jumps by $0$ or $1$ at each step (\cref{prop:jump}), and that the jumps occur on fully understood strata: on each fibre of $I\mapsto(I:y)$, which is an affine space by \cref{thm:A}, the jump $1$ occurs on an affine hyperplane, on the whole fibre, or nowhere (\cref{prop:fibre}). We call the ideals with jump $1$ split and those with jump $0$ glued. The jump vector $\varepsilon(I)\in\{0,1\}^r$ and the contact lengths $j_x((I:y^k))$ determine each other, and the level sets of the jump vector together with the rows are the pieces of the stratification. The complement of a hyperplane in $\AAA^m$ is $\AAA^{m-1}\times\GG_m$, and iterating produces pieces $\AAA^\delta\times\GG_m^\tau$ (\cref{thm:torus-fatline}). Hence every coefficient of $g_v(Q,z)$ lies in $\ZZ[\LL]$ (\cref{cor:gv-pieces}), which makes the closed form of \cref{thm:D} exact in $\Kvar$. Running the $x$-chain over the pieces then stratifies every $\Hilb_n(C_{u,v},0)$ (\cref{thm:torus-all}) and proves \cref{thmA}(a) and \cref{thmB} in $\Kvar$. The image of an ideal in the reduced node $\{xy=0\}$, which explains the torus coordinates as gluing parameters, is discussed in \cref{subsec:gluing} and is not needed for the proofs.

\subsection{Statistics along the $y$-chain}\label{subsec:statistics}
The stratification is indexed by the following invariants. Let $I$ have finite colength $n$ and put $I_i:=(I:y^i)$ for $i\ge0$, so that $I_0=I$ and $I_i=R$ for $i\ge r(I)$, where $r(I)$ is the smallest $r$ with $y^r\in I$. The \emph{rows} of $I$ are $a_i(I):=j_y(I_i)$. By \eqref{eq:rankdict}, $a_i(I)=\rk Y^i-\rk Y^{i+1}$, which is the number of Jordan blocks of $Y$ of size $>i$. So $a(I)=(a_0\ge a_1\ge\cdots)$ is a partition of $n$ with $r(I)$ non-zero parts, and $y^v\in I$ if and only if $r(I)\le v$. For a monomial ideal $I_\lambda$ the rows are the rows of the Young diagram $\lambda$, in the sense of \cref{subsec:conventions}. Geometrically, peeling $Z=\Spec R/I$ along $\{y=0\}$ removes the layer $Z\cap\{y=0\}$, of length $a_0$, and leaves $\Spec R/I_1$; the rows are the lengths of the successive layers.

The second statistic is the contact length $j_x(I_i)$ of each $I_i$ with the other branch. Its behaviour along the chain is governed by the following lemma.

\begin{proposition}[contact jumps]\label{prop:jump}
Let $I\ne R$ have finite colength $n$. Then
\[
j_x(I)-j_x\big((I:y)\big)=j_y(I)-j_y\big((I:x)\big)=n-\rk X-\rk Y+\rk XY\in\{0,1\},
\]
and this number equals $1$ if and only if $y^{j_x(I)}\in I+(xy)$, if and only if $x^{j_y(I)}\in I+(xy)$.
\end{proposition}
\begin{proof}
By \cref{lem:contact}(a) and \eqref{eq:rankdict}, $j_x(I)-j_x((I:y))=(n-\rk X)-(\rk Y-\rk XY)$, and exchanging $X$ and $Y$ gives the same expression for $j_y(I)-j_y((I:x))$. For the remaining assertions put $j:=j_x(I)$. Since $I\ne R$ we have $j\ge1$, and $(I:y)+(x)\subseteq\big((I+(x)):y\big)$. The quotient $\big((I+(x)):y\big)/(I+(x))$ is the kernel of $y$ on $R/(I+(x))=\kk[y]/(y^j)$, which is the line spanned by $y^{j-1}$. Hence $0\le j_x(I)-j_x((I:y))\le1$, with equality if and only if $y^{j-1}\in(I:y)+(x)$. If $y^{j-1}=f+xg$ with $yf\in I$, then $y^j=yf+xyg\in I+(xy)$. Conversely, if $y^j=f+xyg$ with $f\in I$, then $y(y^{j-1}-xg)=f\in I$, so $y^{j-1}\in(I:y)+(x)$. Thus the jump is $1$ if and only if $y^{j_x(I)}\in I+(xy)$, and by the same argument with $x$ and $y$ exchanged, if and only if $x^{j_y(I)}\in I+(xy)$.
\end{proof}

\begin{definition}[split and glued ideals, jump vectors]\label{def:split}
Let $I\ne R$ have finite colength. We call $I$ \emph{split} if the jump of \cref{prop:jump} equals $1$, and \emph{glued} if it equals $0$. For $I$ of finite colength with $r=r(I)$ rows, the \emph{jump vector} of $I$ is
\[
\varepsilon(I)=(\varepsilon_0,\dots,\varepsilon_{r-1})\in\{0,1\}^r,\qquad\varepsilon_i:=j_x(I_i)-j_x(I_{i+1}),
\]
so that $\varepsilon_i=1$ if $I_i$ is split and $\varepsilon_i=0$ if $I_i$ is glued. Since $j_x(I_r)=j_x(R)=0$,
\begin{equation}\label{eq:contact-from-type}
j_x(I_i)=\sum_{i'\ge i}\varepsilon_{i'}\qquad(0\le i\le r),
\end{equation}
so the jump vector and the contact lengths $j_x(I_i)$ determine each other. We write $|\varepsilon|:=\sum_i\varepsilon_i=j_x(I)$. For a partition $a$ of $n$ with $r$ non-zero parts and a vector $\varepsilon\in\{0,1\}^r$ put
\[
C_a:=\{I\in\Hilb_n(\AAA^2,0):\ a(I)=a\},\qquad C_{a,\varepsilon}:=\{I\in C_a:\ \varepsilon(I)=\varepsilon\}.
\]
\end{definition}

Every monomial ideal is split, since $y^{j_x(I_\lambda)}\in I_\lambda$. The ideal $(x+\gamma y,y^2)$ with $\gamma\ne0$ is glued, with $j_x=j_y=1$, because $y\notin(x+\gamma y,\,y^2,\,xy)$.

By \eqref{eq:rankdict}, $a_i(I)=\rk Y^i-\rk Y^{i+1}$ and $j_x(I_i)=\rk Y^i-\rk XY^i$, so the pieces $C_{a,\varepsilon}$ are the non-empty level sets of $I\mapsto(\rk Y^i,\rk XY^i)_{i\ge0}$. Level sets of ranks of morphisms of vector bundles are locally closed, so the $C_{a,\varepsilon}$ are locally closed subvarieties of $\Hilb_n(\AAA^2,0)$, and those with $r\le v$ partition $\Hilb_n(\{y^v=0\},0)$. This is the only use we make of the rank interpretation; the pieces themselves are defined by contact lengths.

The sets $C_a$ with $r\le v$ are the cells of Bejleri's paving of the fat line \cite{Bej}, \cite[Ex.~3.3]{Luan}. By the mirror form of \cref{thm:A} and \cref{prop:family}, $I\mapsto I_1$ maps $C_a$ onto $C_{(a_1,a_2,\dots)}$ and makes it an affine bundle of rank $a_1$, so by induction $C_a$ is an iterated affine bundle over a point, and $C_a\cong\AAA^{n-a_0}$ by Quillen--Suslin \cite{Quillen,Suslin} and the vanishing of $H^1$ on affine schemes. Summing $Q^{|a|}\LL^{|a|-a_0}$ over the partitions with at most $v$ rows recovers Bejleri's series $1/(Q;Q\LL)_v$ \eqref{eq:bejleri} \cite{Bej}. The contact length $j_x$ is not constant on the cells of positive dimension (\cref{ex:pieces-v2}), which is why the finer pieces $C_{a,\varepsilon}$ are needed.

\subsection{The fibre rule}\label{subsec:ychain}
How does the jump vector behave in the fibres of $I\mapsto(I:y)$? By the mirror form of \cref{thm:A}, every $I$ with $(I:y)=J$ and $j_y(I)=a$ is of the form $I_{\bar c}=yJ+R(x^a+yc)$ for a unique point  $\bar c$ in the affine space
\[
\cA_y(J,a)=\{\bar c\in R/J:\ y\bar c=-\bar x^{\,a}\}\cong\AAA^{j_y(J)},\qquad a\ge j_y(J).
\]

\begin{proposition}[fibre rule]\label{prop:fibre}
Let $J$ have finite colength, let $a\ge\max(j_y(J),1)$, and let $\bar c\in\cA_y(J,a)$.
\begin{enumerate}[label=(\alph*)]
\item If $J=R$, then $I_{\bar c}=(y,x^a)$ is split.
\item Let $J\ne R$ and $j':=j_x(J)$. The image of $c$ in $R/(J+(x))=\kk[y]/(y^{j'})$ equals $\ell(\bar c)\,y^{j'-1}$ for an affine-linear function $\ell:\cA_y(J,a)\to\kk$, and
\[
I_{\bar c}+(x)=\big(x,\ y^{j'+1},\ \ell(\bar c)\,y^{j'}\big).
\]
Hence $I_{\bar c}$ is split, with $j_x(I_{\bar c})=j'+1$, if $\ell(\bar c)=0$, and glued, with $j_x(I_{\bar c})=j'$, otherwise.
\item If $J\ne R$ is split, the linear part of $\ell$ is non-zero.
\item If $J$ is glued, then $\ell$ is constant, and $I_{\bar c}$ is split for all $\bar c$ if $a>j_y(J)$, and glued for all $\bar c$ if $a=j_y(J)$.
\end{enumerate}
\end{proposition}
\begin{proof}
(a) We have $(y,x^a)+(x)=(x,y)$, so $j_x=1$, and $y\in(y,x^a)+(xy)$.

(b) Let $c_0\in\kk[y]/(y^{j'})$ be the image of $c$. From $yc\equiv-x^a$ modulo $J$ and $a\ge1$ we get $yc_0=0$, so $c_0\in\kk\,y^{j'-1}$. This defines $\ell$, which is affine-linear because $R/J\to R/(J+(x))$ is linear. Now $I_{\bar c}+(x)=(x)+y\big(J+(x)\big)+R\,yc$ and $y(J+(x))=(xy,y^{j'+1})$, while $yc\equiv\ell(\bar c)\,y^{j'}$ modulo $(x,y^{j'+1})$. This gives the displayed formula, so $j_x(I_{\bar c})=j'+1$ if $\ell(\bar c)=0$ and $j_x(I_{\bar c})=j'$ otherwise. Since $j_x((I_{\bar c}:y))=j_x(J)=j'$, these are the split and the glued case.

(c) The linear part of $\ell$ is the map $(J:y)/J=\ker(Y|R/J)\to R/(J+(x))$ induced by the inclusion. Its image is $\big((J:y)+(x)\big)/(J+(x))$, of dimension $j_x(J)-j_x((J:y))$, which is $1$ if $J$ is split and $0$ if $J$ is glued (\cref{prop:jump}). This proves (c) and the first claim of (d).

(d) Let $J$ be glued, so that $\ell$ is constant, and put $i':=j_y(J)$ and $K:=J+(xy)$. By (b), $yc\equiv\ell\,y^{j'}$ modulo $(x,y^{j'+1})$, and $y(J+(x))=(xy,y^{j'+1})\subseteq K$, so $yc\equiv\ell\,y^{j'}$ modulo $K$. Also $yc\equiv-x^a$ modulo $K$. Hence $\ell\,y^{j'}\equiv-x^a$ modulo $K$. Since $J$ is glued, neither $y^{j'}$ nor $x^{i'}$ lies in $K$ (\cref{prop:jump}). If $a>i'$, then $x^{a-1}\in J+(y)=(y,x^{i'})$ gives $x^a\in xJ+(xy)\subseteq K$, so $\ell\,y^{j'}\in K$ and $\ell=0$. If $a=i'$, then $\ell\,y^{j'}\equiv-x^{i'}\not\equiv0$ modulo $K$, so $\ell\ne0$.
\end{proof}

The fibre rule describes the fibre of $I\mapsto(I:y)$ over $J$. If $J$ is split, the fibre is the disjoint union of the hyperplane $\{\ell=0\}$ of split ideals and its complement, consisting of glued ideals. If $J$ is glued, all ideals in the fibre are of the same kind, split when the new row $a$ is longer than $j_y(J)$ and glued when $a=j_y(J)$. In particular a vector $\varepsilon$ can occur as the jump vector of an ideal with rows $a$ only if $\varepsilon_{r-1}=1$, by (a), and, whenever $\varepsilon_{i+1}=0$, $\varepsilon_i=1$ exactly when $a_i>a_{i+1}$, by (d).

\subsection{The stratification of a fat line}\label{subsec:fatline-strat}
Each piece is an iterated fibration with fibres $\AAA^m$ or $\AAA^m\times\GG_m$, which gives its class.

\begin{theorem}\label{thm:torus-fatline}
Let $a=(a_0\ge\dots\ge a_{r-1}\ge1)$ be a partition of $n$ and $\varepsilon\in\{0,1\}^r$. Put $a'=(a_1,\dots,a_{r-1})$ and $\varepsilon'=(\varepsilon_1,\dots,\varepsilon_{r-1})$.
\begin{enumerate}[label=(\alph*)]
\item $C_{a,\varepsilon}\ne\varnothing$ if and only if $\varepsilon_{r-1}=1$ and, for every $i\le r-2$ with $\varepsilon_{i+1}=0$, one has $\varepsilon_i=1\iff a_i>a_{i+1}$.
\item Let $\varepsilon$ be as in (a) and $r\ge2$. The map $I\mapsto(I:y)$ restricts to a Zariski-locally trivial morphism $C_{a,\varepsilon}\to C_{a',\varepsilon'}$, with fibre $\AAA^{a_1-1}$ if $(\varepsilon_0,\varepsilon_1)=(1,1)$, with fibre $\AAA^{a_1-1}\times\GG_m$ if $(\varepsilon_0,\varepsilon_1)=(0,1)$, and with fibre $\AAA^{a_1}$ if $\varepsilon_1=0$.
\item Consequently $[C_{a,\varepsilon}]=\LL^{\delta}(\LL-1)^{\tau}$ in $\Kvar$, where
\[
\delta=\delta(a,\varepsilon)=\sum_{i=0}^{r-2}\big(a_{i+1}-\varepsilon_{i+1}\big),\qquad\tau=\tau(a,\varepsilon)=\#\{i\le r-2:\ \varepsilon_{i+1}=1,\ \varepsilon_i=0\}.
\]
On $C_{a,\varepsilon}$ the contact lengths are constant, $j_x(I_i)=\sum_{i'\ge i}\varepsilon_{i'}$; in particular $j_x(I)=|\varepsilon|$.
\item Moreover $C_{a,\varepsilon}\cong\AAA^{\delta}\times\GG_m^{\tau}$.
\end{enumerate}
\end{theorem}
\begin{proof}
Write $I_i=(I:y^i)$. Then $I_r=R$, and $I_{r-1}=(y,x^{a_{r-1}})$ is split by \cref{prop:fibre}(a). Part (a) follows from the fibre rule by induction on $r$. Over a point $J$ of $C_{a',\varepsilon'}$ the fibre of $C_a\to C_{a'}$ is $\cA_y(J,a_0)\cong\AAA^{a_1}$, and \cref{prop:fibre}(b),(d) say which of its points are split: both kinds occur when $\varepsilon_1=1$, because $a_1\ge1$, and exactly one kind occurs when $\varepsilon_1=0$.

For (b), let $\Sigma:=C_{a',\varepsilon'}$, a locally closed subvariety on which $j_y(J)=a_1$ and $j_x(J)=j'\ge1$ are constant. By the mirror form of \cref{prop:family}, the set $F_\Sigma$ of all $I$ with $(I:y)\in\Sigma$ and $j_y(I)=a_0$ is locally closed and isomorphic to the torsor $\cA_\Sigma\to\Sigma$ under the vector bundle $\mathcal K:=\ker Y\subseteq\mathcal T|_\Sigma$ of rank $a_1$, where $\cA_\Sigma=\{Y\bar c=-X^{a_0}e\}$ is a closed subvariety of the total space of $\mathcal T|_\Sigma$. We first show that $\ell$ is a regular function on $\cA_\Sigma$. Since $\rk X=\col(J)-j'$ is constant on $\Sigma$, $\mathcal Q:=\mathcal T/X\mathcal T$ is a vector bundle of rank $j'$ (Step 1 of the proof of \cref{prop:family}), with fibre $R/(J+(x))=\kk[y]/(y^{j'})$ at $J$, and the images of $e,Ye,\dots,Y^{j'-1}e$ form a global frame of it. Let $\pi:\mathcal T|_\Sigma\to\mathcal O_\Sigma$ be the composite of $\mathcal T\to\mathcal Q$ with the last coordinate in this frame. It is a morphism of vector bundles, hence a regular function on the total space of $\mathcal T|_\Sigma$, linear on the fibres, and by \cref{prop:fibre}(b) its restriction to $\cA_\Sigma$ is $\ell$. So $\ell$ is regular and affine-linear on the fibres of $\cA_\Sigma$, with linear part $\lambda:=\pi|_{\mathcal K}:\mathcal K\to\mathcal O_\Sigma$.

If $\varepsilon_1=0$, then $\ell$ is constant on each fibre by \cref{prop:fibre}(d), so $C_{a,\varepsilon}=F_\Sigma\cong\cA_\Sigma$, which is Zariski-locally trivial over $\Sigma$ with fibre $\AAA^{a_1}$, being a torsor under a vector bundle. If $\varepsilon_1=1$, then $\lambda$ is surjective on every fibre by \cref{prop:fibre}(c), so its kernel $\mathcal K'$ is a subbundle of rank $a_1-1$, and locally on $\Sigma$ there is a line subbundle $\mathcal L\subseteq\mathcal K$ with $\mathcal K=\mathcal K'\oplus\mathcal L$ and $\lambda|_{\mathcal L}$ an isomorphism. On an open $U\subseteq\Sigma$ over which $\mathcal K'$ and $\mathcal L$ are trivial and $\cA_\Sigma$ has a section $\sigma_0$, the map
\[
U\times\AAA^{a_1-1}\times\AAA^1\to\cA_\Sigma|_U,\qquad(J,\kappa,s)\mapsto\sigma_0(J)+\kappa+(\lambda|_{\mathcal L})^{-1}\big(s-\ell(\sigma_0(J))\big),
\]
with $\kappa$ in the fibre of $\mathcal K'$, is an isomorphism under which $\ell$ becomes the coordinate $s$. Therefore $C_{a,(1,\varepsilon')}=\{\ell=0\}$ and $C_{a,(0,\varepsilon')}=\{\ell\ne0\}$ are Zariski-locally isomorphic over $\Sigma$ to $\Sigma\times\AAA^{a_1-1}$ and to $\Sigma\times\AAA^{a_1-1}\times\GG_m$.

(c) follows from (b) by induction on $r$ and the multiplicativity of classes in Zariski-locally trivial fibrations; the formula for $j_x(I_i)$ is \eqref{eq:contact-from-type}.

(d) By induction on $r$ we show that each piece is isomorphic to \[\AAA^{\delta'}\times\GG_m^{\tau'}=\Spec\kk[t_1,\dots,t_{\delta'},s_1^{\pm1},\dots,s_{\tau'}^{\pm1}].\] If $\Sigma$ is of this form, then $\mathcal K$ is trivial by Swan's theorem \cite{Swan}, which says that finitely generated projective modules over this ring are free, and $\cA_\Sigma$ has a global section because $H^1$ of a quasi-coherent sheaf on the affine scheme $\Sigma$ vanishes. So $\cA_\Sigma\cong\Sigma\times\AAA^{a_1}$. If $\varepsilon_1=1$, the surjection $\lambda:\mathcal O_\Sigma^{a_1}\to\mathcal O_\Sigma$ splits globally, its kernel is projective, hence free, and the open $U$ above can be taken to be $\Sigma$. Part (d) is not needed for the classes, for which the local structure of (b) suffices.
\end{proof}

\begin{corollary}[the contact distribution of a fat line, piece by piece]\label{cor:gv-pieces}
In $\Kvar{}[z][[Q]]$,
\[
g_v(Q,z)=\sum_{(a,\varepsilon)}Q^{|a|}\,z^{|\varepsilon|}\,\LL^{\delta(a,\varepsilon)}(\LL-1)^{\tau(a,\varepsilon)},
\]
the sum over the partitions $a$ with at most $v$ rows and the jump vectors $\varepsilon$ allowed by \cref{thm:torus-fatline}(a), the empty partition contributing $1$. In particular every coefficient of $g_v(Q,z)$ lies in $\ZZ[\LL]$, and the closed form of \cref{thm:D} holds exactly in $\Kvar{}[z][[Q]]$,
\[
g_v(Q,z)=\sum_{k=0}^{v}\qbin{v}{k}{Q\LL}\frac{(Qz)^k}{\big(Q(Q\LL)^{v-k};Q\LL\big)_k}.
\]
\end{corollary}
\begin{proof}
The pieces $C_{a,\varepsilon}$ with $r(a)\le v$ are locally closed, pairwise disjoint and cover $\Hilb_n(\{y^v=0\},0)$, their classes are $\LL^\delta(\LL-1)^\tau$ by \cref{thm:torus-fatline}(c), and $j_x$ is constant on them. This gives the first formula, and it shows that the coefficient $\Pi_{n,k}(\LL)$ of $Q^nz^k$ in $g_v(Q,z)$ is a polynomial in $\LL$ with integer coefficients which does not depend on the field $\kk$, since the index set and the exponents do not. The coefficient $c_{n,k}(\LL)$ of $Q^nz^k$ in the closed form also lies in $\ZZ[\LL]$. For $\kk=\FF_q$ both specialize under the counting measure to the number of $\FF_q$-points of $\Hilb_n(\{y^v=0\},0)_{j_x=k}$, by \cref{thm:refined}. Two polynomials in $\ZZ[\LL]$ that agree at every prime power are equal, so $\Pi_{n,k}=c_{n,k}$.
\end{proof}

The two formulas of \cref{cor:gv-pieces} are different expansions of the same series. The identity between them is proved only through the point counts.
 At $\LL=1$ both become $\sum_aQ^{|a|}z^{r(a)}$, the generating function of partitions with at most $v$ rows: only the pieces with $\tau=0$ survive, and these are the pieces with $\varepsilon=(1,\dots,1)$, since $\varepsilon_{r-1}=1$ and the last $0$ of any other vector is followed by a $1$ and contributes to $\tau$.

\begin{example}[the fat line $\{y^2=0\}$ in small colength]\label{ex:pieces-v2}
For $n=2$, $\Hilb_2(\AAA^2,0)\cong\PP^1$, and the cells are $C_{(2)}=\{(y,x^2)\}$ and $C_{(1,1)}=\{(x+\gamma y,\,y^2):\gamma\in\kk\}\cong\AAA^1$. In $C_{(1,1)}$ the piece $\varepsilon=(1,1)$ is $\{\gamma=0\}=\{(x,y^2)\}$, with $j_x=2$, and the piece $(0,1)$ is $\{\gamma\ne0\}\cong\GG_m$, with $j_x=1$. The contact strata are $\Hilb_2(C_{0,2},0)_{j_x=1}=\{(y,x^2)\}\sqcup\GG_m\cong\AAA^1$ and $\Hilb_2(C_{0,2},0)_{j_x=2}=\{(x,y^2)\}$. Neither is a union of cells. For $n=3$ the cells with at most two rows are $C_{(3)}=\{(y,x^3)\}$ and $C_{(2,1)}=\{(xy,\,y^2,\,x^2+c_0y):c_0\in\kk\}\cong\AAA^1$, the latter being the fibre $\cA_y(\mm,2)$ over $J=\mm$. Here $j_x(\mm)=1$ and $\ell(\bar c)=c_0$, so $\{c_0=0\}$ is the piece $(1,1)$, with $j_x=2$, and $\{c_0\ne0\}\cong\GG_m$ is the piece $(0,1)$, with $j_x=1$. For $n=4$ the cells $(4)$, $(3,1)$ and $(2,2)$ contribute $z$, $z^2+(\LL-1)z$ and $\LL z^2+\LL(\LL-1)z$ to the coefficient of $Q^4$. In the cell $(2,2)$ the base is $J=(y,x^2)$, the fibre is $\cA_y((y,x^2),2)=R/(y,x^2)\cong\AAA^2$ with coordinates $c=c_0+c_1x$, and $\ell=c_0$. Altogether
\[
g_2(Q,z)=1+Qz+Q^2(z^2+\LL z)+Q^3(z^2+\LL z)+Q^4\big((1+\LL)z^2+\LL^2z\big)+\cdots,
\]
which agrees with the expansion of the formula of \cref{ex:g2}.

In general, a cell $C_a$ with $r\ge2$ rows meets the contact stratum $\Hilb_n(C_{0,v},0)_{j_x=r}$, in the piece $\varepsilon=(1,\dots,1)$, and also strata $\Hilb_n(C_{0,v},0)_{j_x=k}$ with $k<r$, obtained by choosing $\varepsilon_{r-2}=0$, by \cref{thm:torus-fatline}(a),(c). So the contact length is not constant on the cells of positive dimension, and the $x$-chain of \cref{thm:B}, whose base needs constant contact length, cannot be run over the cells directly. The pieces $C_{a,\varepsilon}$ are the correct base.
\end{example}

\subsection{All fat nodes, and the motive}\label{subsec:K0}
It remains to run the $x$-chain over the pieces of the fat line. For $I\in\Hilb_n(C_{u,v},0)$ put $I^{(m)}:=(I:x^{u-m})$ and $j_m:=j_x(I^{(m)})$ for $0\le m\le u$, so that $I^{(u)}=I$ and $I^{(0)}=(I:x^u)\in\Hilb(\{y^v=0\},0)$ by \cref{thm:A}(d).

\begin{theorem}\label{thm:torus-all}
Let $\kk$ be any field and $u,v,n\ge0$.
\begin{enumerate}[label=(\alph*)]
\item The non-empty level sets of $I\mapsto\big(a(I^{(0)}),\varepsilon(I^{(0)});j_1,\dots,j_u\big)$ on $\Hilb_n(C_{u,v},0)$ are locally closed subvarieties isomorphic to $\AAA^d\times\GG_m^\tau$, with $d=\delta(a,\varepsilon)+j_0+\dots+j_{u-1}$ and $\tau=\tau(a,\varepsilon)$, where $j_0:=|\varepsilon|=j_x(I^{(0)})$. The data that occur are the pairs $(a,\varepsilon)$ of \cref{thm:torus-fatline}(a) with at most $v$ rows together with integers $j_0\le j_1\le\dots\le j_u$; the corresponding level set has $n=|a|+j_1+\dots+j_u$ and $j_x(I)=j_u$. The empty partition is allowed and corresponds to $I^{(0)}=R$, with $\delta=\tau=j_0=0$.
\item In $\Kvar{}[z][[Q]]$,
\begin{gather*}
\sum_{n,j}[\Hilb_n(C_{u,v},0)_{j_x=j}]\,Q^nz^j=\frac{1}{(Qz;Q\LL)_u}\sum_{k=0}^{v}\qbin{v}{k}{Q\LL}\frac{\big(Q(Q\LL)^uz\big)^k}{\big(Q(Q\LL)^{v-k};Q\LL\big)_k},\\  \sum_n[\Hilb_n(C_{u,v},0)]\,Q^n=\frac{C_{u,v}(Q\LL,Q)}{(Q;Q\LL)_{u,v}}.
\end{gather*}
In particular $[\Hilb_n(C_{u,v},0)]=M_n(\LL)\in\ZZ[\LL]$, with $M_n$ the coefficient of $Q^n$ in $F_{u,v}(Q\LL,Q)$.
\end{enumerate}
\end{theorem}
\begin{proof}
(a) The level sets are locally closed, because the data are determined by ranks: by \eqref{eq:rankdict} the ranks of $X^b$, $b\le u+1$, record the colengths and contact lengths of the $I^{(m)}$, and since $R/I^{(0)}\cong X^uV$ (\cref{lem:rankdict}), the ranks of $X^uY^i$ and $X^{u+1}Y^i$ are those of $Y^i$ and $XY^i$ on $R/I^{(0)}$, which determine $a(I^{(0)})$ and $\varepsilon(I^{(0)})$. Since $I^{(m-1)}=(I^{(m)}:x)$, \cref{thm:A}(d) shows that $I^{(m)}\in\Hilb(C_{m,v},0)$ if and only if $I^{(m-1)}\in\Hilb(C_{m-1,v},0)$. By induction on $m$, the set of ideals $I^{(m)}$ with prescribed data $(a,\varepsilon;j_1,\dots,j_m)$ is $F_S$ in the notation of \cref{prop:family}, where $S$ is the set of ideals $I^{(m-1)}$ with data $(a,\varepsilon;j_1,\dots,j_{m-1})$, and $j=j_m$, $j'=j_{m-1}$. It is non-empty if and only if $j_m\ge j_{m-1}$ (\cref{thm:A}(a)), and by \cref{prop:family} it is an affine bundle of rank $j_{m-1}$ over $S\cong\AAA^{d'}\times\GG_m^{\tau'}$, which is trivial by Swan's theorem \cite{Swan} and the vanishing of $H^1$ on affine schemes. This proves (a).

(b) By \cref{thm:B}, the left-hand side of the first identity is $g_v(Q,(Q\LL)^uz)/(Qz;Q\LL)_u$. Substituting $z\mapsto(Q\LL)^uz$ in the exact closed form of \cref{cor:gv-pieces} gives the first identity. Setting $z=1$ and using \cref{prop:symmetric} gives the second. Finally, by (a) the class $[\Hilb_n(C_{u,v},0)]$ is a sum of terms $\LL^d(\LL-1)^\tau$, so it lies in $\ZZ[\LL]$.
\end{proof}

\begin{corollary}\label{cor:pieces-node}
Let $\mathcal D_n$ be the set of data $(a,\varepsilon;j_1,\dots,j_u)$ of \cref{thm:torus-all}(a) with $|a|+j_1+\dots+j_u=n$, and write $d$ and $\tau$ for the exponents attached to a datum. Then
\[
M_n(\LL)=[\Hilb_n(C_{u,v},0)]=\sum_{\mathcal D_n}\LL^{d}(\LL-1)^{\tau}.
\]
Consequently $M_n(1+s)\in\NN[s]$, and $M_n(1)$ is the number of data with $\tau=0$. These are exactly the level sets containing a monomial ideal, and each of them contains exactly one. Hence $M_n(1)$ is the number of monomial ideals in $\Hilb_n(C_{u,v},0)$, that is, the number of partitions of $n$ not containing the box $(u,v)$.
\end{corollary}
\begin{proof}
The identity is \cref{thm:torus-all}(a) and the additivity of classes, and the data with $\tau=0$ are those with $\varepsilon=(1,\dots,1)$, as noted after \cref{cor:gv-pieces}. A monomial ideal $I$ has $(I:x^by^i)$ monomial for all $b,i$, hence split, so it lies in a level set with $\tau=0$. Conversely, let a datum with $\varepsilon=(1,\dots,1)$ be given. Let $I^{(0)}$ be the monomial ideal with rows $a$, and put $I^{(m)}:=xI^{(m-1)}+(y^{j_m})$ for $m=1,\dots,u$. Since $j_m\ge j_{m-1}$ we have $y^{j_m}\in I^{(m-1)}$, so $I^{(m)}$ is the ideal of \cref{thm:A} with tail $\bar c=0$, and it is monomial. Thus the level set contains the monomial ideal $I^{(u)}$. It contains no other, because the datum determines $I^{(0)}$, and a monomial ideal $I^{(m)}$ with $(I^{(m)}:x)=I^{(m-1)}$ and $j_x(I^{(m)})=j_m$ contains $y^{j_m}$ and therefore has tail $0$. The last assertion is \cref{lem:fixed}.
\end{proof}

\Cref{thm:torus-all} contains \cref{thmA,thmB}. \Cref{cor:pieces-node} recovers the first statement of \cref{cor:euler} without the fixed-point theorem of \cite{BB}, and it shows that the positivity of \cref{cor:positivity} holds in the stronger form $M_n(\LL)\in\NN[\LL-1]$. The theorem also extends Luan's \cite[Thm.~1.5]{Luan}, which gives stratifications into pieces $\GG_m^{m_1}\times\AAA^{m_2}$ for $u\in\{v,v-1,v-2\}$, to all $(u,v)$; our pieces are different from Luan's diagonal strata \cite[Thm.~1.5]{Luan}.

\subsection{The torus coordinates}\label{subsec:gluing}
This subsection explains the meaning of the torus factors; nothing in it is used elsewhere. Both contact lengths of $I$ are computed from $I+(x)$ and $I+(y)$, which contain $xy$, and so is the jump of \cref{prop:jump}. So everything above depends only on the image $I+(xy)$ of $I$ in the reduced node, and the ideals of the reduced node have a transparent classification.

\begin{lemma}[ideals of the node]\label{lem:node}
Let $J\subsetneq R$ be an ideal of finite colength with $xy\in J$, and put $i=j_y(J)$, $j=j_x(J)$. Then exactly one of the following holds:
\begin{enumerate}[label=(\alph*)]
\item $J=(xy,x^i,y^j)$, and $\col J=i+j-1$.
\item $J=(xy,\,x^i+\gamma y^j)$ for a unique $\gamma\in\kk^\times$, and $\col J=i+j$.
\end{enumerate}
\end{lemma}
\begin{proof}
Modulo $(xy)$ every $f\in R$ is uniquely $f\equiv f_0+\sum_{a\ge1}\alpha_ax^a+\sum_{b\ge1}\beta_by^b$. If $f\in J$, then $f(0,y)\in J+(x)=(x,y^j)$ forces $f_0=0$ and $\beta_b=0$ for $b<j$, and symmetrically $\alpha_a=0$ for $a<i$. Hence $J\subseteq M:=(xy,x^i,y^j)$. Since $x^i\in J+(y)$, there is $g\in J$ with $g\equiv x^i+\sum_{a>i}\alpha'_ax^a+\sum_{b\ge j}\beta'_by^b$, and similarly $h\in J$ with $h\equiv y^j+\sum_{a\ge i}\alpha''_ax^a+\sum_{b>j}\beta''_by^b$. Now $xg\equiv x^{i+1}\cdot(\text{unit})$ modulo $(xy)$, so $x^{i+1}\in J$, and likewise $y^{j+1}\in J$. Thus $M':=(xy,x^{i+1},y^{j+1})\subseteq J\subseteq M$, and $M/M'$ has basis $x^i,y^j$. The subspace $U=J/M'$ contains the images $(1,\beta'_j)$ of $g$ and $(\alpha''_i,1)$ of $h$, so $U\ne0$. If $U=M/M'$ we are in case (a). Otherwise $\dim U=1$, so $\alpha''_i\beta'_j=1$ and $U=\kk\cdot(1,\gamma)$ with $\gamma=\beta'_j\ne0$. Then $J=M'+(x^i+\gamma y^j)=(xy,x^i+\gamma y^j)$, because $x(x^i+\gamma y^j)\equiv x^{i+1}$ and $y(x^i+\gamma y^j)\equiv\gamma y^{j+1}$ modulo $(xy)$. This is case (b).
\end{proof}

By \cref{prop:jump}, $I$ is split exactly when $I+(xy)$ is the monomial ideal of case (a), and glued exactly when it is of the form (b); in the glued case we call $\gamma(I):=\gamma$ the \emph{gluing parameter} of $I$. In the proof of \cref{prop:fibre}(b), the image of $I_{\bar c}$ in the reduced node is $(xy,y^{j'+1},x^a+\ell(\bar c)y^{j'})$, so the coordinate $\ell$ on the complement of the hyperplane is the gluing parameter of the new ideal, and \cref{prop:fibre}(d) says that over a glued $J$ with $a=j_y(J)$ the gluing parameter is transported unchanged, $\gamma(I_{\bar c})=\gamma(J)$. Thus a factor $\GG_m$ appears in $C_{a,\varepsilon}$ exactly at the steps $1\to0$, read from $I_{i+1}$ to $I_i$, its coordinate is the new gluing parameter, and a step $0\to1$ forgets the parameter. The torus $\mathbf T$ acts on $\gamma(I)$ through the character $(\alpha,\beta)\mapsto\alpha^{-j_y(I)}\beta^{j_x(I)}$, so the pieces with $\tau\ge1$ contain no fixed points, in accordance with \cref{cor:pieces-node}.

For the reduced node itself this is the chain of projective lines of \cref{ex:nodeA}(3). The ideals of colength $n$ containing $xy$ are the $n$ monomial ideals $(xy,x^i,y^{n+1-i})$, which are the $\mathbf T$-fixed points, and the $n-1$ families $(xy,x^i+\gamma y^{n-i})$, $\gamma\in\kk^\times$, which are the open $\mathbf T$-orbits; as $\gamma\to0$ or $\gamma\to\infty$ the glued ideal degenerates to one of the two neighbouring monomial ideals. The class $n+(n-1)(\LL-1)=1+(n-1)\LL$ is the one computed in \cref{ex:nodeA}(3).

\section{Denominators and positivity of the node function}\label{sec:positivity}

All statements in this section concern the explicit rational function
\[
F_{u,v}(q,Q)=\frac{C_{u,v}(q,Q)}{(Q;q)_{u,v}}
\]
of \eqref{eq:F}, independently of its geometric and link-theoretic interpretations. These statements constitute \cref{thmE}, and they provide the node evidence for the two normalization conjectures of \cref{sec:conjecture}. Full normalization gives the signed polynomial $C_{u,v}$, while \cref{prop:positive-max} proves positivity after normalization at either component and determines the largest positive partial normalization. In short, the node already exhibits the phenomenon that distinguishes links from knots. The full numerator is a signed polynomial, while normalization at a single component gives an infinite positive series.

The starting point is \cref{cor:positive-C}, which we restate with the roles of $u$ and $v$ in the form used below. For all $u,v\ge0$,
\begin{equation}\label{eq:positive-partial}
\frac{C_{u,v}(q,Q)}{(Q;q)_u}=\sum_{k=0}^{u}\pbin{u}{k}\,\frac{Q^kq^{vk}}{(Qq^{u-k};q)_k}\in\NN[q][[Q]],
\qquad
C_{u,v}(q,Q)=\sum_{k=0}^{u}\pbin{u}{k}(Q;q)_{u-k}\,Q^kq^{vk}.
\end{equation}

\begin{proposition}[exact denominator]\label{prop:minimal}
For $u,v>0$, the reduced denominator of $F_{u,v}$ in $\QQ(q)(Q)$, normalized to have constant term one in $Q$, is exactly $(Q;q)_{u,v}$.
\end{proposition}
\begin{proof}
Its only possible factors are the $1-q^jQ$. If such a factor divided $C_{u,v}$, then $C_{u,v}(q,q^{-j})$ would vanish identically. For every real $q>1$, every coefficient $R_m(u,v;q)$ in the rank expansion of $C_{u,v}$ is positive, since each factor $(q^u-q^i)(q^v-q^i)/(q^m-q^i)$ with $i<m\le\min(u,v)$ is, and the constant coefficient is one. Therefore $C_{u,v}(q,q^{-j})>0$, a contradiction.
\end{proof}

\begin{proposition}[largest positive partial normalization]\label{prop:positive-max}
Suppose $1\le u\le v$. Let $A(q,Q)$ be a product of factors $1-q^jQ$, with multiplicities at most those in $(Q;q)_{u,v}$, and with constant term one. Then
\begin{equation}\label{eq:positive-iff}
A(q,Q)\,F_{u,v}(q,Q)\in\NN[q][[Q]]
\quad\Longleftrightarrow\quad A\mid (Q;q)_v .
\end{equation}
Thus $(Q;q)_v$ is the unique largest factor of the natural denominator that can be cleared while preserving coefficientwise positivity. Clearing $(Q;q)_v$ leaves an infinite series, whereas clearing $(Q;q)_{u,v}$ leaves the signed polynomial $C_{u,v}$.
\end{proposition}
\begin{proof}
Write $E=(Q;q)_{u,v}/A=\prod_j(1-q^jQ)^{e_j}$. From the rank expansion,
\[
[Q]\,C_{u,v}=R_1(u,v;q)=(q^v-1)\,[u]_q,\qquad [u]_q=1+q+\dots+q^{u-1},
\]
and $1/E=1+Q\sum_je_jq^j+O(Q^2)$. Consequently
\[
[Q]\,\frac{C_{u,v}}{E}=q^v[u]_q-[u]_q+\sum_je_jq^j .
\]
Since $v\ge u$, the monomials $q^j$ with $0\le j<u$ occur in this expression with coefficient $-1+e_j$, so nonnegativity forces $e_j\ge1$ for every $0\le j<u$. This is exactly $(Q;q)_u\mid E$, i.e.\ $A\mid (Q;q)_v$. Conversely, if $E=(Q;q)_uE'$, then \eqref{eq:positive-partial} shows that $C_{u,v}(q,Q)/(Q;q)_u$ is coefficientwise nonnegative, and $1/E'$ is a product of geometric series. Their product is $AF_{u,v}$.
\end{proof}

In particular $F_{u,v}\in\NN[q][[Q]]$, which at $q=Q\LL$ is the positivity of \cref{cor:positivity}, and at $q=\LL$ gives the coefficientwise nonnegativity of the Quot series of \cref{thm:node-main}. These are statements about the coefficients of an explicit series, not about the cohomology of the varieties.

\begin{example}\label{ex:11}
For $u=v=1$,
\[
F_{1,1}(q,Q)=\frac{1-Q+qQ}{(1-Q)^2},\qquad (1-Q)F_{1,1}=1+\frac{qQ}{1-Q}.
\]
The partial normalization is positive but infinite. The full numerator has a negative coefficient. In fact $F_{1,1}$ cannot be written as a quotient with a nonzero numerator in $\NN[q,Q]$ and a polynomial denominator in $\QQ[q,Q]$. Such a numerator would be divisible by the irreducible polynomial $1-Q+qQ$, which is coprime to $(1-Q)^2$. For $0<q<1$ it would then vanish at $Q=(1-q)^{-1}>0$, which is impossible for a nonzero polynomial with nonnegative coefficients.

On the geometric side, $F_{1,1}(\LL,Q)=\sum_m(1+m\LL)Q^m$ is the Quot series of the normalization of the node, and $F_{1,1}(Q\LL,Q)=1+\sum_{n\ge1}(1+(n-1)\LL)Q^n$ is the Hilbert series of the node (\cref{ex:nodeA}). This compares a chain of $n-1$ projective lines on the Hilbert side with the extension fibres $\{E\}\sqcup\GG_m$ over the lattices on the Quot side.
\end{example}

\section{The branchwise Hilb--Quot--HOMFLY conjecture}\label{sec:conjecture}

This section places \cref{thmA}(c) in a general conjectural framework. For a reduced plane curve germ with branches $g_1,\dots,g_\ell$ and ranks $a_1,\dots,a_\ell$ we compare three generating series, namely the Hilbert series of the thickened curve $\{\prod_ig_i^{a_i}=0\}$, the Quot series of the normalization module $E_{\bfA}$, and the bottom part of the row-colored HOMFLY homology of the link.

\Cref{conj:main} predicts that they agree after the substitutions dictated by the node. For the node, the Hilbert--Quot part is \cref{thmA}(c). The link interpretation also raises normalization questions that are invisible for knots. We formulate them as \cref{conj:full-normalization,conj:component-positivity} and verify them for the node using \cref{sec:positivity}. The geometric results of the paper do not depend on this section.

In this section $\kk=\CC$. Let $g=g_1\cdots g_\ell\in\CC[[x,y]]$ be reduced, with distinct irreducible factors, and put
\[
R_g=\CC[[x,y]]/(g),\qquad
\widetilde R_g=\prod_{i=1}^{\ell}\widetilde R_i,\qquad \widetilde R_i\simeq\CC[[u_i]].
\]
For $\bfA=(a_1,\dots,a_\ell)\in\ZZ_{>0}^{\ell}$ the two geometric objects are
\[
E_{\bfA}=\bigoplus_{i=1}^{\ell}\widetilde R_i^{\,a_i}\quad\text{as an $R_g$-module},\qquad
R_{g;\bfA}=\CC[[x,y]]/\Big(\prod_ig_i^{a_i}\Big),
\]
and the link $L_g$ of the germ has its $i$-th component colored by the one-row partition $(a_i)$. For the node, $g=xy$, $\bfA=(u,v)$, $E_{\bfA}=E_{u,v}$, $R_{g;\bfA}=\mathcal O_{C_{u,v},0}$, and $L_g$ is the Hopf link. Equal ranks recover the Hilb--Quot specialization of \cite[Conj.~1.12]{HJO}. All ranks equal to one recover the usual Hilb--Quot and bottom-row ORS predictions \cite{KT,ORS}.

The correspondence between thickened Hilbert schemes and row-colored link homology is already stated in \cite[Conj.~1.1 and \S1.2]{Luan}, with the branch multiplicities giving the row lengths. We use that existing colored ORS prediction. The additional conjectural identification here is with the normalization-module Quot series, allowing independent branch ranks. The colored stable-pair theorem of \cite[Thm.~1.2]{Maulik} is an earlier decategorified result with partition boundary data. It is distinct from the two-variable homological identity used here.

\subsection{Geometric series and their denominators}\label{subsec:geometric-conventions}
The comparison of geometry with link homology uses virtual weight polynomials, which are defined first, and the geometric generating series of the conjecture with their natural normalizations. The main point is that the reduced denominator of every Quot series is exactly the product of the branch factors (\cref{cor:general-denominator}).

To compare with link homology we need a numerical realization of classes. For a complex variety $Z$ let
\[
W_c(Z;q):=\sum_{i,w}(-1)^{i}\dim\mathrm{Gr}^W_wH^i_c(Z;\QQ)\,q^{w/2}\in\ZZ[q^{1/2}]
\]
be its compactly supported virtual weight polynomial, defined using Deligne's weight filtration \cite{DeligneIII}. It is additive and multiplicative, hence factors through $K_0(\mathrm{Var}_\CC)$, with $W_c(\AAA^1;q)=q$. For a variety whose class is a polynomial $M(\LL)$ it equals $M(q)$.

For a non-empty variety $Z$ of dimension $d$ the leading term of $W_c(Z;q)$ is $m\,q^d$, where $m>0$ is the number of top-dimensional irreducible components of $Z$. This holds because compactly supported cohomology has weights at most its cohomological degree, and $H^{2d}_c(Z)$ is pure of weight $2d$ \cite{DeligneIII}.

Put
\begin{align*}
Z^{\Quot}_{E_{\bfA}}(q,Q)&=\sum_{m\ge0}W_c\big(\Quot^{R_g}_m(E_{\bfA});q\big)\,Q^m,\\
Z^{\Hilb}_{g;\bfA}(q,Q)&=\sum_{n\ge0}W_c\big(\Hilb_n(\Spec R_{g;\bfA},0);q\big)\,Q^n,
\end{align*}
where $\Hilb_n(\Spec R_{g;\bfA},0)=\Hilb_n(\{\prod_ig_i^{a_i}=0\},0)$ in the notation of \cref{def:hilb}. Put
\begin{align}
(Q;q)_{\bfA}&=\prod_{i=1}^{\ell}(Q;q)_{a_i}=\prod_{i=1}^{\ell}(Q;q)_{a_i},\label{eq:branch-factors}\\
N_{E_{\bfA}}(q,Q)&=(Q;q)_{\bfA}\,Z^{\Quot}_{E_{\bfA}}(q,Q),\label{eq:quot-numerator}\\
\mathscr H_{g;\bfA}(q,Q)&=(Q;qQ)_{\bfA}\,Z^{\Hilb}_{g;\bfA}(q,Q).\label{eq:hilb-numerator}
\end{align}
\begin{corollary}[the exact denominator of a Quot series]\label{cor:general-denominator}
The series $N_{E_{\bfA}}(q,Q)=\sum_kW_c(\cE_k(E_{\bfA});q)\,Q^k$ is a polynomial in $Q$ with coefficients in $\ZZ[q^{1/2}]$, and the reduced denominator of $Z^{\Quot}_{E_{\bfA}}(q,Q)$ over $\QQ(q^{1/2})$, normalized to have constant term one in $Q$, is exactly $(Q;q)_{\bfA}$. The same holds, with the same proof, for any reduced complete local $\CC$-algebra with split normalization, as in \cref{subsec:extension}.
\end{corollary}
\begin{proof}
The first assertion is \eqref{eq:extension-factorization} after applying $W_c$.

The irreducible factors of $(Q;q)_{\bfA}$ in $\QQ(q^{1/2})[Q]$ are the $1-q^jQ$, and such a factor divides $N_{E_{\bfA}}$ if and only if $N_{E_{\bfA}}(q,q^{-j})=0$ identically. By the leading-term property of $W_c$, every nonzero coefficient $W_c(\cE_k(E_{\bfA});q)$ is positive for all sufficiently large real $q$, and the constant coefficient is $W_c(\cE_0(E_{\bfA});q)=1$, since $\cE_0(E_{\bfA})=\{E_{\bfA}\}$. Hence $N_{E_{\bfA}}(q,q^{-j})>0$ for all sufficiently large $q$, for each $j\ge0$, and no factor cancels.
\end{proof}

For the node, $N_{E_{u,v}}=C_{u,v}(q,Q)$ and $\mathscr H_{xy;(u,v)}=C_{u,v}(qQ,Q)$ by \cref{thmA}.

\subsection{The bottom-row normalization}\label{subsec:bottom-normalization}
On the link side we use the bottom Hochschild part of the unreduced row-colored HOMFLY homology, divided by its leading monomial, which gives the unknot the value $1/(Q;q)_r$.

Let $L_g$ be the oriented algebraic link of the germ. We use the unreduced row-projector HOMFLY homology of $L_g$, with component $i$ colored by $S^{a_i}$. See \cite{Hog,HM}. Write its Poincar\'e series as
\[
\mathcal P^{L_g,\mathrm{unr}}_{\bfA}(\alpha;Q,T).
\]
Here $\alpha$ records Hochschild degree, and $Q,T$ are the variables of the link-homological formulas of \cite{Luan}. In particular the even generators of an $r$-colored unknot have degrees $QT^{-j}$, $0\le j<r$.

For a nonzero series, $\Bot_\alpha$ denotes the extraction of its smallest $\alpha$-degree. We remove the remaining overall grading monomial by the following convention. If the lowest $Q$-term is a single monomial $Q^\mu T^\nu$, divide by that monomial. Denote this operation by $\Nor_Q$ and put
\begin{equation}\label{eq:B}
B^{L_g}_{\bfA}(Q,T)=\Nor_Q\Bot_\alpha\,\mathcal P^{L_g,\mathrm{unr}}_{\bfA}(\alpha;Q,T),
\qquad
\mathscr B^{L_g}_{\bfA}(q,Q):=B^{L_g}_{\bfA}(Q,q^{-1}).
\end{equation}
The existence of the indicated single leading monomial is included in the link part of the conjecture. This is an intrinsic grading convention, independent of the geometric generating series. It fixes the unknot value
\begin{equation}\label{eq:unknot}
B^{\bigcirc}_{(r)}(Q,T)=\frac1{(Q;T^{-1})_r},\qquad \mathscr B^{\bigcirc}_{(r)}(q,Q)=\frac1{(Q;q)_r},
\end{equation}
and it removes the monomial shift in Luan's computation for colors $((1),(v))$ \cite[Cor.~1.3]{Luan}.

\subsection{The conjectures and the positive-model question}\label{subsec:conjecture-package}
Here are the branchwise Hilb--Quot--HOMFLY conjecture (\cref{conj:main}), the two normalization conjectures for row-colored algebraic links (\cref{conj:full-normalization,conj:component-positivity}) and how they are related, and the question of a positive model (\cref{question:reduced-row}).

\begin{conjecture}[branchwise Hilb--Quot--HOMFLY]\label{conj:main}
The numerators $N_{E_{\bfA}}(q,Q)$ and $\mathscr H_{g;\bfA}(q,Q)$ lie in $\ZZ[q,Q]$, the normalized bottom series $B^{L_g}_{\bfA}$ is defined, and
\begin{equation}\label{eq:threeway}
\mathscr H_{g;\bfA}(q,Q)=N_{E_{\bfA}}(qQ,Q)=(Q;qQ)_{\bfA}\,B^{L_g}_{\bfA}\big(Q,(qQ)^{-1}\big).
\end{equation}
The motivic version additionally asserts that the classes $[\Quot^{R_g}_m(E_{\bfA})]$ and $[\Hilb_n(\Spec R_{g;\bfA},0)]$ are polynomials in $\LL$ and that their polynomial representatives satisfy the same identities with $q=\LL$.
\end{conjecture}

Before normalization, \cref{conj:main} says
\begin{equation}\label{eq:rawthreeway}
Z^{\Hilb}_{g;\bfA}(q,Q)=Z^{\Quot}_{E_{\bfA}}(qQ,Q)=B^{L_g}_{\bfA}\big(Q,(qQ)^{-1}\big).
\end{equation}
The equality is a statement about virtual weights. Purity and evenness, if available, give an additional cohomological interpretation. They do not follow merely from a polynomial with nonnegative coefficients. The equality of the Hilbert and link series is the two-variable colored ORS prediction used in \cite{Luan}. For the node it is the statement $(\mathrm{ORS}_{u,v})$ of \cref{sec:hopf}. The branchwise Quot identification is the additional part of \cref{conj:main}, and for the node it is \cref{thmA}(c), in its motivic form, for every $(u,v)$.

\begin{conjecture}[full normalization]\label{conj:full-normalization}
The fully normalized bottom series is a polynomial, that is,
\begin{equation}\label{eq:all-branch-poly}
\mathscr N^{L_g}_{\bfA}(q,Q):=(Q;q)_{\bfA}\,\mathscr B^{L_g}_{\bfA}(q,Q)\in\ZZ[q,Q].
\end{equation}
\end{conjecture}

\begin{conjecture}[normalization at a component]\label{conj:component-positivity}
For every component $i$, the series normalized at that component is coefficientwise nonnegative, that is,
\begin{equation}\label{eq:one-branch-positive}
R^{L_g}_{\bfA;i}(q,Q):=(Q;q)_{a_i}\,\mathscr B^{L_g}_{\bfA}(q,Q)\in\NN[q][[Q]].
\end{equation}
\end{conjecture}

Full normalization is a consequence of \cref{conj:main}. Under \eqref{eq:threeway}, $\mathscr N^{L_g}_{\bfA}(q,Q)=N_{E_{\bfA}}(q,Q)$, which is a polynomial by \cref{prop:extension} (and lies in $\ZZ[q,Q]$ by the first assertion of \cref{conj:main}). Componentwise positivity is an additional assertion. For a knot the two normalizations coincide, predicting a positive polynomial. For a link, the full numerator may be signed, whereas normalization at one component may give an infinite positive series. Neither normalization conjecture asserts minimality or maximality of a denominator. \Cref{prop:positive-max} verifies the positivity prediction for the node and, in that example, determines the largest factor preserving positivity (\cref{subsec:normalization-evidence}).

\begin{question}[a positive model]\label{question:reduced-row}
For a marked component $i$ of a row-colored algebraic link, what natural bigraded object has Hilbert series $R^{L_g}_{\bfA;i}(q,Q)$? Can it be realized as bottom HOMFLY homology reduced at component $i$, compatibly with the $\Bot_\alpha\widetilde{\mathcal P}$ candidate in the unibranched case?
\end{question}

For knots, the HJO/GGS framework supplies the candidate in \eqref{eq:knot-convention} below. For links, reduction at a marked component is a candidate in the framework of \cite{GNSSS}. The general comparison with row-projector homology remains to be established. The rigorous factorization theorem of \cite{ConReduced} concerns column colors. \Cref{subsec:positive-models} explains this distinction and gives an explicit candidate graded vector space for the node normalization.

\subsection{Compatibility with the knot convention in HJO}\label{subsec:knot-convention}
For knots, \cref{conj:main} is compatible with the conventions of the thickened Hilb--Quot conjecture of \cite{HJO} and with the reduced quadruply graded framework of \cite{GGS}, as we now check.

Suppose $\ell=1$, write $a_1=n$, and let $\delta=\dim_\CC(\widetilde R_g/R_g)$. In the reduced GGS framework used in \cite{HJO}, the expected reduced/unreduced relation gives
\begin{equation}\label{eq:knot-convention}
(Q;q)_n\,B^K_{(n)}(Q,q^{-1})=Q^{n\delta}\,\Bot_\alpha\widetilde{\mathcal P}^{K}_{S^n}\big(Q^{1/2},Q^{-1/2},q^{1/2}\big).
\end{equation}
This is a compatibility statement with that framework, not a theorem identifying every rigorously constructed row-colored invariant with the reduced quadruply graded GGS object. The reduced/unreduced relation appears in \cite[Appendix B]{GGS}.

GGS self-symmetry for a row of length $n$ has, in the variables $(\mathsf Q,t_r,t_c)$ of \cite{GGS}, the form
\[
\widetilde{\mathcal P}(\mathsf Q,t_r,t_c)=\widetilde{\mathcal P}(\mathsf Q^{-1}t_r^{-1}t_c^{-n},t_r,t_c).
\]
It replaces the arguments on the right of \eqref{eq:knot-convention} by $(q^{-n/2},Q^{-1/2},q^{1/2})$. Substituting $qQ$ for $q$ gives precisely the thickened-Hilbert slice in \cite[Conj.~1.12]{HJO}. Thus no fourth grading on an independently colored link is required to formulate \cref{conj:main}.

Equation~\eqref{eq:knot-convention} supplies the unibranched candidate in \cref{question:reduced-row}. If it is realized by a finite-dimensional reduced homology, it gives a coefficientwise nonnegative \emph{Laurent} polynomial after the monomial regrading. The stronger exponent support in \cref{conj:full-normalization,conj:component-positivity} is an additional assertion. For links, the two normalizations differ. The positive-model question is taken up in \cref{subsec:positive-models}.

\subsection{Evidence from the node}\label{subsec:normalization-evidence}
The node supplies the following evidence for the conjectures. The Hilbert--Quot part is a theorem, both numerators are polynomials, and the node rational function satisfies both normalization conjectures, with an explicit positive formula for the normalization at one component.

For the node, the Hilbert--Quot equality in \eqref{eq:threeway} is \cref{thmA}(c), and the polynomiality of both numerators holds, with $N_{E_{u,v}}=C_{u,v}(q,Q)$ and $\mathscr H_{xy;(u,v)}=C_{u,v}(qQ,Q)$ in $\ZZ[q,Q]$. The motivic version holds as well, by \cref{thm:torus-all} and \cref{thm:node-main}. The Hilbert--link equality is known for $\min(u,v)=1$ \cite[Cor.~1.3]{Luan}. For all $u,v$ the predicted link series coincides with the bottom specialization of the Hopf-link formula of \cite{GNSSS} (\cref{prop:gnsss}), but the identification of that formula with row-projector homology is open for two nontrivial colors (\cref{sec:hopf}).

\Cref{conj:main} identifies the full numerator with $\mathscr N^{L_g}_{\bfA}(q,Q)=N_{E_{\bfA}}(q,Q)$, and hence implies \cref{conj:full-normalization}. Normalization at a component makes the additional positivity prediction of \cref{conj:component-positivity}. Under the three-way correspondence its Quot formulation is
\[
(Q;q)_{a_i}\,Z^{\Quot}_{E_{\bfA}}(q,Q)\in\NN[q][[Q]],
\]
and it implies the analogous statement $(Q;qQ)_{a_i}Z^{\Hilb}_{g;\bfA}(q,Q)\in\NN[q][[Q]]$. At $\ell=1$ the two normalizations coincide, predicting a positive polynomial as in \eqref{eq:knot-convention}. If \cref{conj:main} holds, \cref{cor:general-denominator} additionally determines the exact denominator and makes each one-component normalization infinite when $\ell>1$.

Since $(Q;q)_r\mid (Q;q)_s$ for $r\le s$, and $(Q;q)_r/(Q;q)_s$ has a nonnegative power-series expansion, \cref{conj:component-positivity} is equivalent to its assertion for a component of maximal color. This equivalence concerns graded dimensions. A reduced homological object could still carry additional structure depending on the marked component. For arbitrary row-colored links, one can ask the same polynomiality and positivity questions with Laurent grading shifts. The conjectures above fix the stronger degree support supplied by the algebraic-link setting.

\begin{corollary}\label{cor:node-positive-both}
The node rational function satisfies the polynomiality and positivity assertions of \cref{conj:full-normalization,conj:component-positivity}, independently of its link interpretation. In particular,
\[
(Q;q)_vF_{u,v}(q,Q)=\frac{C_{u,v}(q,Q)}{(Q;q)_u}=\sum_{k=0}^{u}\pbin{u}{k}\frac{Q^kq^{vk}}{(Qq^{u-k};q)_k}\in\NN[q][[Q]],
\]
and $(Q;q)_uF_{u,v}(q,Q)\in\NN[q][[Q]]$ by symmetry.
\end{corollary}
\begin{proof}
The full numerator is $C_{u,v}\in\ZZ[q,Q]$. The displayed formula is \eqref{eq:positive-partial}, and symmetry gives the other component normalization.
\end{proof}

For $u\le v$, \cref{prop:positive-max} strengthens this evidence by showing that $(Q;q)_v$ is the largest divisor of $(Q;q)_{u,v}$ preserving positivity. This maximality describes an additional feature of the node calculation. It is not part of \cref{conj:component-positivity}, and no analogous maximality assertion is included for general links. For example, an unlink has $\mathscr B=1/(Q;q)_{\bfA}$, so clearing every factor gives the positive polynomial $1$.

\subsection{Positive models and what a numerator measures}\label{subsec:positive-models}
Finally let us discuss what kind of object could explain the positivity predicted by \cref{conj:component-positivity}. We distinguish the three positivity questions involved, review the available factorization results for colored link homology, and give an explicit candidate graded vector space for the node. Most of the general discussion in this subsection is not original and can be found in the cited sources; it is included for the reader who is not immersed in link homology.

\Cref{question:reduced-row} asks for a natural model explaining \cref{conj:component-positivity}. There are three different positivity questions. They concern the coefficients of the unreduced series, those of a one-component reduction, and those of the numerator obtained by clearing every branch factor.

The first is automatic for the actual homology in its specified expansion, since dimensions are nonnegative. The third is not automatic, as \cref{ex:11} shows. For knots, a finite-dimensional reduced theory and an unknot-factorization theorem give a positive finite numerator. For links, reduction at one component can leave an infinite-dimensional theory.

In the unibranched case, the candidate is the regraded $\Bot_\alpha\widetilde{\mathcal P}$ expression in \eqref{eq:knot-convention}. \Cref{subsec:knot-convention} explains its status in the HJO/GGS framework. For a multibranched link, the expected positive model for \eqref{eq:one-branch-positive} is the bottom part of homology reduced at component $i$. This is already addressed by the colored-homology framework of \cite[\S\S3.2, 3.4, Eq.~(3.17)]{GNSSS}. It postulates an unknot-factorization at the chosen component and allows the reduced link homology to be infinite-dimensional. Taking the bottom degree gives the one-component normalization above, after the grading change.

There is now a rigorous factorization theorem for intrinsic \emph{column}-colored homology. Conners constructs reduction at any marked component and proves
\[
HHH(L)\cong HHH(\bigcirc_{a_i})\otimes\overline{HHH}(L,i)
\]
in that theory \cite[Thms.~1.1 and 3.50]{ConReduced}. This extends Wedrich's construction at a component of minimal color for nonsplit links \cite{Wedrich}. Conners explicitly separates column colors from the unresolved reduced construction for other partitions \cite[\S1.3.2]{ConReduced}. Thus these results support the proposed positive model, but do not establish the row-projector statement in full generality.

A structural answer to \cref{question:reduced-row} would construct a natural bigraded reduced bottom homology $\overline M_i$, with Hilbert series $R^{L_g}_{\bfA;i}$, and a compatible factorization
\[
M\cong S_i\otimes\overline M_i,\qquad S_i=\CC[u_{i,0},\dots,u_{i,a_i-1}],\quad \deg u_{i,j}=Qq^j,
\]
where $M$ is the unreduced bottom homology. An $S_i$-linear factorization would say that $M$ is free over the polynomial algebra for component $i$. This is stronger than positivity of its Hilbert series. Freeness over each $S_i$ separately would not imply freeness over their tensor product, so it is compatible with a signed full numerator.

For a simple algebraic illustration, let $S=\CC[u_1,u_2]$, with both variables of degree $Q$. The bigraded module
\[
S/(u_1-u_2)\ \oplus\ S\{qQ\}
\]
is free over each $\CC[u_i]$ separately and has Hilbert series $F_{1,1}$. Here braces denote a grading shift. Its full $S$-numerator is nevertheless $1-Q+qQ$. This illustrates consistency of the two assertions, without identifying this module with a particular link-homology construction.

The positive node expression also supplies an explicit candidate graded vector space for reduction at the $v$-colored component, namely
\[
\mathcal V^{(v)}_{u,v}=\bigoplus_{k=0}^{u}H^{2\bullet}(\Gr(k,u),\CC)\otimes\CC[w_{u-k},\dots,w_{u-1}]\{Q^kq^{vk}\},
\qquad \deg w_j=Qq^j .
\]
Here $H^{2d}$ has degree $q^d$, and the polynomial ring is $\CC$ when $k=0$. Its Hilbert series is $(Q;q)_vF_{u,v}(q,Q)$, by \cref{cor:node-positive-both}. This gives a concrete target for a reduced model. No canonical identification or compatible link-homology action is constructed here.

\section{The colored Hopf link}\label{sec:hopf}

This section specializes to the node, whose link is the Hopf link with components colored by $\mathrm{Sym}^u$ and $\mathrm{Sym}^v$. Assuming the colored ORS conjecture, \cref{thmA} predicts the bottom part of the colored Khovanov--Rozansky homology of the Hopf link (\cref{cor:hopf}). We compare this prediction with Luan's established case $\min(u,v)=1$ \cite[Cor.~1.3]{Luan}, and with the independently colored Hopf-link formula of Gukov--Nawata--Saberi--Sto\v{s}i\'c--Su{\l}kowski \cite{GNSSS}, whose bottom specialization agrees with it (\cref{prop:gnsss}). This is the main link-theoretic evidence for \cref{conj:main}.

\subsection{The conjecture of Oblomkov--Rasmussen--Shende}
Recall the conjecture of Oblomkov, Rasmussen and Shende \cite{ORS} relating the Hilbert schemes of a plane curve singularity to the triply graded homology of its link, its colored version for non-reduced curves, and Luan's proof of the colored version for $C_{1,v}$ \cite[Cor.~1.3]{Luan}.

For a plane curve germ $C$ with link $L_C=C\cap S^3_\epsilon$, Oblomkov, Rasmussen and Shende conjectured that the cohomology of the punctual Hilbert schemes of $C$, suitably graded, is the $a=0$ part of the triply graded Khovanov--Rozansky homology $\mathrm{HHH}(L_C)$ \cite{ORS}. See \cite[Conj.~1.1]{Luan} for the form used below. The Euler-characteristic version, a statement about HOMFLY-PT polynomials, was proved by Maulik \cite{Maulik}, following Oblomkov and Shende \cite{OS}.

For a non-reduced curve whose components $C_i$ have multiplicities $k_i$, the components of the link are to be colored by the symmetric powers $\mathrm{Sym}^{k_i}$, i.e.\ by one-row Young diagrams \cite[\S1.2]{Luan}. See \cite{Cautis,HM,Conners} for colored triply graded homology. For $C_{u,v}$ the link is the Hopf link $T(2,2)$ with components colored by $\mathrm{Sym}^u$ and $\mathrm{Sym}^v$.

Luan proved this colored conjecture for $C_{1,v}$ \cite[Cor.~1.3]{Luan}, by comparing the Poincar\'e polynomials of $\Hilb_n(C_{1,v},0)$ with the computation of the $(\mathrm{Sym}^v,\mathrm{Sym}^1)$-colored Hopf link by Hogancamp--Mellit and Conners \cite{HM,Conners}.

To keep the two gradings apart we write $T_{\Hilb}$ for the homological variable on the Hilbert scheme side and $T$ for the one on the link side, and $\operatorname{Poin}(Z;T_{\Hilb})=\sum_ib_i(Z)T_{\Hilb}^i$ for the Poincar\'e polynomial. Then \cite[Cor.~1.3, Cor.~2.2]{Luan} state that, up to the factor $\prod_{i=1}^{v-1}T^i$,
\begin{equation}\label{eq:luan-hopf}
\begin{gathered}
\sum_{n\ge0}\operatorname{Poin}\big(\Hilb_n(C_{1,v},0);T_{\Hilb}\big)Q^n
=\mathcal P_{a=0}\big(T(2,2);\mathrm{Sym}^v,\mathrm{Sym}^1\big)\Big|_{T=(QT_{\Hilb}^2)^{-1}},
\qquad\text{where}\\
\mathcal P_{a=0}\big(T(2,2);\mathrm{Sym}^v,\mathrm{Sym}^1\big)\doteq\Big(1+\frac{QT^{-v}}{1-Q}\Big)\prod_{i=1}^{v}\frac1{1-QT^{1-i}} .
\end{gathered}
\end{equation}
Here the spaces are paved, so $\operatorname{Poin}$ is also their virtual Poincar\'e polynomial, and $\doteq$ means equality up to a monomial factor. Indeed, under $T=(QT_{\Hilb}^2)^{-1}$ one has $QT^{1-i}=Q^iT_{\Hilb}^{2(i-1)}$ and $QT^{-v}=Q^{v+1}T_{\Hilb}^{2v}$, and the right-hand side becomes the series of \cite[Thm.~1.2]{Luan}. For $u=2$ the link homology is not known, and Luan states the resulting prediction as a conjecture \cite[\S1.5.2, Conj.~10.4]{Luan}.

\subsection{The prediction}
The colored conjecture for the fat node $C_{u,v}$, combined with \cref{thmA}, yields a prediction for the bottom part of the $(\mathrm{Sym}^u,\mathrm{Sym}^v)$-colored homology of the Hopf link (\cref{cor:hopf}); examples follow.

We formulate the colored conjecture for $C_{u,v}$ in the normalization of \eqref{eq:luan-hopf}.

\smallskip\noindent $(\mathrm{ORS}_{u,v})$ \emph{Up to a monomial factor,
\[
\sum_{n\ge0}\operatorname{Poin}\big(\Hilb_n(C_{u,v},0);T_{\Hilb}\big)\,Q^n
=\mathcal P_{a=0}\big(T(2,2);\mathrm{Sym}^u,\mathrm{Sym}^v\big)\Big|_{T=(QT_{\Hilb}^2)^{-1}},
\]
where $\mathcal P_{a=0}(T(2,2);\mathrm{Sym}^u,\mathrm{Sym}^v)$ is the Poincar\'e series of the $a=0$ part of the $(\mathrm{Sym}^u,\mathrm{Sym}^v)$-colored Khovanov--Rozansky homology of the Hopf link.}

\smallskip Its virtual form $(\mathrm{ORS}^{\mathrm{vir}}_{u,v})$ is the same statement with $\operatorname{Poin}$ replaced by the virtual Poincar\'e polynomial of \cite{BP}, \cite[Def.~3.4]{Luan}, the additive invariant of complex varieties that sends $\LL$ to $T_{\Hilb}^2$. It is the Hilbert--link part of \eqref{eq:rawthreeway} for the node. The two forms agree whenever the compactly supported cohomology of $\Hilb_n(C_{u,v},0)$ vanishes in odd degrees and is pure of weight $2i$ in degree $2i$, since then no cancellation occurs in the virtual Poincar\'e polynomial. This holds for $\min(u,v)\le2$ by Luan's affine pavings \cite[Thms.~1.2, 1.4]{Luan}.

\begin{corollary}\label{cor:hopf}
Assume $(\mathrm{ORS}^{\mathrm{vir}}_{u,v})$. Then
\[
\mathcal P_{a=0}\big(T(2,2);\mathrm{Sym}^u,\mathrm{Sym}^v\big)\ \doteq\
\frac{\displaystyle\sum_{m=0}^{\min(u,v)}Q^m\,R_m\big(u,v;T^{-1}\big)}
{\displaystyle\prod_{i=0}^{u-1}\big(1-QT^{-i}\big)\prod_{i=0}^{v-1}\big(1-QT^{-i}\big)},
\]
with $R_m(u,v;q)=\prod_{i=0}^{m-1}\frac{(q^u-q^i)(q^v-q^i)}{q^m-q^i}$ as in \cref{lem:rank}.
\end{corollary}
\begin{proof}
By \cref{thm:torus-all}, $[\Hilb_n(C_{u,v},0)]=M_n(\LL)$, so the virtual Poincar\'e polynomial of $\Hilb_n(C_{u,v},0)$ is $M_n(T_{\Hilb}^2)$, and the left-hand side of $(\mathrm{ORS}^{\mathrm{vir}}_{u,v})$ is $C_{u,v}(Q\LL,Q)/(Q;Q\LL)_{u,v}$ at $\LL=T_{\Hilb}^2$. The substitution $T=(QT_{\Hilb}^2)^{-1}$ is equivalent to $T_{\Hilb}^2=(QT)^{-1}$. So $\LL=(QT)^{-1}$, hence $Q\LL=T^{-1}$ and $Q^i\LL^{i-1}=QT^{1-i}$. Therefore $(Q;Q\LL)_u$ becomes $(Q;T^{-1})_u=\prod_{i=0}^{u-1}(1-QT^{-i})$, likewise for $(Q;Q\LL)_v$, and $C_{u,v}(Q\LL,Q)=C_{u,v}(Q\LL,Q)$ becomes $C_{u,v}(T^{-1},Q)=\sum_mQ^mR_m(u,v;T^{-1})$, by \eqref{eq:NP}.
\end{proof}

\begin{example}\label{ex:hopf}
\begin{enumerate}[label=(\arabic*)]
\item For $u=1$ we have $R_1(1,v;q)=q^v-1$, so the right-hand side is
\[
\frac{1-Q+QT^{-v}}{(1-Q)\prod_{i=0}^{v-1}(1-QT^{-i})}=\Big(1+\frac{QT^{-v}}{1-Q}\Big)\prod_{i=1}^{v}\frac1{1-QT^{1-i}} .
\]
This is \eqref{eq:luan-hopf}, as it must be.
\item For $u=v=2$ and $s=T^{-1}$ the right-hand side is
\[
\frac{1+(s-1)(s+1)^2\,Q+s(s-1)^2(s+1)\,Q^2}{(1-Q)^2(1-sQ)^2}.
\]
For $u=2$ purity holds by Luan's affine pavings \cite[Thm.~1.4]{Luan}, so the prediction depends only on $(\mathrm{ORS}_{2,v})$. Luan states the case $u=2$ as \cite[Conj.~10.4]{Luan}, in the variables $(Q,T)$ of the Hilbert scheme. The formula above is the same prediction after the substitution of \cref{cor:hopf}.
\item For $u=v=3$ and $s=T^{-1}$ the numerator is
\[
1+(s-1)(s^2+s+1)^2Q+s(s-1)^2(s+1)(s^2+s+1)^2Q^2+s^3(s-1)^3(s+1)(s^2+s+1)Q^3,
\]
and the denominator is $(1-Q)^2(1-sQ)^2(1-s^2Q)^2$. The coefficients of $Q$, $Q^2$, $Q^3$ are the numbers of $3\times3$ matrices of rank $1,2,3$ over $\FF_s$.
\end{enumerate}
\end{example}

The prediction is symmetric in $u$ and $v$, as it must be, since the two components of the Hopf link play symmetric roles. In the notation of \eqref{eq:F} it is $F_{u,v}(T^{-1},Q)$, the third specialization of the node function, at $q=T^{-1}$. In the normalization of \cref{subsec:bottom-normalization} it reads as follows.

\subsection{The row-colored Hopf conjecture and the GNSSS formula}\label{subsec:gnsss}
In the normalization of \cref{sec:conjecture} the prediction coincides with the bottom specialization of the independently colored Hopf-link formula of Gukov, Nawata, Saberi, Sto\v{s}i\'c and Su{\l}kowski \cite{GNSSS} (\cref{prop:gnsss}).

\begin{conjecture}[row-colored Hopf bottom series; cf.~\cite{GNSSS}]\label{conj:hopf}
For the positive Hopf link and all $u,v\ge1$,
\begin{equation}\label{eq:hopf}
B^{\mathrm{Hopf}}_{(u),(v)}(Q,T)=\frac{C_{u,v}(T^{-1},Q)}{(Q;T^{-1})_{u,v}}=F_{u,v}(T^{-1},Q),
\qquad\text{i.e.}\qquad
\mathscr B^{\mathrm{Hopf}}_{(u),(v)}(q,Q)=F_{u,v}(q,Q).
\end{equation}
\end{conjecture}

This is the specialization of \cref{conj:main} supplied by \cref{thmA}. Given the Hilbert--Quot equality \cref{thmA}(c), the link part of \eqref{eq:threeway} for the node is exactly \eqref{eq:hopf}. It is also the bottom specialization of an existing independently colored Hopf formula. The following proposition identifies the expressions. Its assertion is algebraic and does not assume that the proposed expression computes row-projector homology.

\begin{proposition}\label{prop:gnsss}
The bottom Hochschild specialization of the expression in \cite[Eq.~(7.8)]{GNSSS}, in its displayed monomial normalization, equals $F_{u,v}(q,Q)$ under
\[
r_1=v,\qquad r_2=u,\qquad Q=\mathsf q^2,\qquad q=\mathsf q^2t_c^2 .
\]
\end{proposition}
\begin{proof}
Every Pochhammer factor involving the Hochschild variable in that expression has constant term one. Its bottom part is therefore
\[
\frac1{(\mathsf q^2;\mathsf q^2t_c^2)_v}
\sum_{k=0}^{u}\qbin{u}{k}{\mathsf q^2t_c^2}
\frac{\mathsf q^{2(v+1)k}\,t_c^{2vk}}{(\mathsf q^{2(u-k+1)}t_c^{2(u-k)};\mathsf q^2t_c^2)_k}.
\]
The indicated substitution turns this into
\begin{equation}\label{eq:gnsss-bottom}
\frac1{(Q;q)_v}\sum_{k=0}^{u}\pbin{u}{k}\frac{Q^kq^{vk}}{(Qq^{u-k};q)_k}
=\frac{C_{u,v}(q,Q)}{(Q;q)_{u,v}},
\end{equation}
where the equality is \eqref{eq:positive-partial}.
\end{proof}

GNSSS obtain their expression using conjectural colored-homology structure, including color sliding. Agreement with that expression is consequently evidence for the link identification, not a proof of it. \Cref{ex:hopf}(1) is Luan's established comparison \cite[Cor.~1.3]{Luan} when one color is $(1)$. Thus in the phrase ``the link side is unknown,'' the missing result is the identification with rigorously constructed row-projector homology.

An explicit conjectural expression is available for every $u,v$, and by \cref{thmA} it is now the \emph{theorem} on both geometric sides. For example, the proposal for colors $((2),(v))$, $v\ge2$, is
\[
B^{\mathrm{Hopf}}_{(2),(v)}(Q,T)=\frac{(1-Q)(1-qQ)+(1+q)(1-Q)Qq^v+Q^2q^{2v}}{(Q;q)_2\,(Q;q)_v},
\qquad q=T^{-1},
\]
which is \cref{ex:hopf}(2) for $v=2$. The numerator and its geometric comparisons are explicit. The general row-colored link identification is still an assertion to be proved.

It would be interesting to compare the prediction with a direct computation. The $(\wedge^a,\wedge^b)$-colored Hopf link, colored by one-column diagrams, was computed by Hogancamp, Rose and Wedrich in terms of ideals generated by Haiman determinants \cite{HRW}, and Conners' row--column mirror symmetry \cite{Conners} relates row-colored and column-colored invariants of positive torus knots. An extension of the latter to the Hopf link, combined with \cite{HRW}, might give access to the $(\mathrm{Sym}^u,\mathrm{Sym}^v)$-colored Hopf link. A theorem transferring the column-colored results to the row-projector normalization \eqref{eq:B} is a separate input, and Conners' reduced-homology theorem \cite{ConReduced} is also a theorem for column colors.

\section{Questions and further directions}\label{sec:questions}

We close with two kinds of further directions. In \cref{subsec:higher-rank} we state a formula for the Quot schemes of free modules of arbitrary rank over the fat node, which extends the point-count form of \cref{thmA}(a) and will be proved in the sequel \cite{FN2}. In \cref{subsec:questions} we list open questions raised by the results of the paper.

\subsection{The fat node in higher rank}\label{subsec:higher-rank}
The following formula for the punctual Quot schemes of the free module of rank $r$ over the fat node will be proved in \cite{FN2}. For $r=1$ these Quot schemes are the Hilbert schemes $\Hilb_n(C_{u,v},0)$ studied in this paper, and the formula reduces to the point-count form of \cref{thmA}(a). For general $r$ the numerator is a plethystic evaluation of the modified Macdonald polynomial of a rectangle. After the geometric definition of the generating function and the notation for Macdonald polynomials, we state the theorem and check that for $r=1$ it agrees with \cref{thmA}(a).

\emph{The generating function.} Let $\kk=\FF_q$, $R=\FF_q[[x,y]]$, let $r\ge1$, and let $\mathcal F=R^{\oplus r}$. For $u,v\ge0$ put
\begin{equation}\label{eq:Zr}
Z^{(r)}_{u,v}(Q):=\sum_{N\ge0}\#\bigl\{L\subseteq\mathcal F\ \text{an $R$-submodule}:\ x^uy^v\mathcal F\subseteq L,\
\dim_{\FF_q}\mathcal F/L=N\bigr\}\,Q^N .
\end{equation}
The submodules counted by the coefficient of $Q^N$ are exactly the submodules of colength $N$ of the free module $\mathcal O_{C_{u,v},0}^{\oplus r}$ over the local ring of the fat node, that is, the quotients $\mathcal O_{C_{u,v},0}^{\oplus r}\twoheadrightarrow M$ of length $N$. So $Z^{(r)}_{u,v}$ is the generating function of the numbers of $\FF_q$-points of the punctual Quot schemes $\Quot_N(\mathcal O_{C_{u,v}}^{\oplus r},0)$.

These are projective varieties, defined exactly as the Hilbert schemes of \cref{subsec:hilb}. By the argument of \cref{lem:mn}, applied to the chain $(\mm^k\mathcal F+L)/L$, every such $L$ contains $\mm^N\mathcal F$, so $\Quot_N(\mathcal O_{C_{u,v}}^{\oplus r},0)$ is the reduced closed subscheme of the Grassmannian of codimension-$N$ subspaces of $\mathcal F/\mm^N\mathcal F$ that are stable under $x$ and $y$ and contain the image of $x^uy^v\mathcal F$.

As in \cref{prop:adhm}, a point is the same as a pair of commuting nilpotent operators $X,Y$ on $\kk^N$ with $X^uY^v=0$, together with $r$ vectors that generate $\kk^N$ as a $\kk[X,Y]$-module, up to the free action of $\GL_N$. For $r=1$ we recover $\Hilb_N(C_{u,v},0)$, and for $u=0$ we obtain the Quot schemes of the free module of rank $r$ over the fat line $\{y^v=0\}$.

\emph{Notation.} Put $s:=q^rQ$, so that $s=qQ$ for $r=1$, and
\[
\Pi^{(r)}_v(Q):=\prod_{i=0}^{v-1}\prod_{k=0}^{r-1}\frac1{1-q^{ri+k}Q^{i+1}} .
\]
For $r=1$ this is $1/(Q;qQ)_v$, the denominator of \cref{thmA}(a) at $\LL=q$.

Let $\widetilde H_\mu[X;q,t]$ be the modified Macdonald polynomial of Garsia and Haiman \cite{GH}, normalized as in \cite[Def.~2.5]{Mellit}. It is the unique symmetric function such that $\widetilde H_\mu[(t-1)X]$ lies in the span of the monomial symmetric functions $m_\nu$ with $\nu\trianglelefteq\mu$, $\widetilde H_\mu[(q-1)X]$ lies in the span of the $m_\nu$ with $\nu\trianglelefteq\mu'$, and $\widetilde H_\mu[1]=1$. In this normalization $\widetilde H_{(n)}[X;q,t]=(q;q)_n\,h_n[X/(1-q)]$ and $\widetilde H_{(1^n)}[X;q,t]=(t;t)_n\,h_n[X/(1-t)]$ \cite[Rem.~2.6]{Mellit}.

We use plethystic notation. For a symmetric function $f$ and an expression $A$ in monomials, $f[A]$ is obtained by writing $f$ as a polynomial in the power sums and substituting $p_k\mapsto p_k[A]$, where $p_k[a]=a^k$ for a monomial $a$ and $p_k$ is additive. For instance $f[1-Q+Qw]$ is obtained by the substitution $p_k\mapsto1-Q^k+Q^kw^k$. Finally, $(r^v)$ denotes the rectangular partition with $v$ parts equal to $r$.

\begin{theorem}[to be proved in \cite{FN2}]\label{thm:higher-rank}
For all $r\ge1$, $u,v\ge0$ and every prime power $q$, with $s=q^rQ$,
\begin{equation}\label{eq:higher-rank}
Z^{(r)}_{u,v}(Q)=\Pi^{(r)}_u(Q)\,\Pi^{(r)}_v(Q)\;\widetilde H_{(r^v)}\bigl[1-Q+Qs^u;\,q,\,s\bigr]
=\frac{\widetilde H_{(r^v)}[1-Q+Qs^u;\,q,s]}{\widetilde H_{(r^u)}[1-Q;\,q,s]\;\widetilde H_{(r^v)}[1-Q;\,q,s]} .
\end{equation}
\end{theorem}

The formula is compatible with \cref{thmA}(a). For $r=1$ we have $s=qQ$ and $\Pi^{(1)}_v=1/(Q;qQ)_v$. By the $q$-binomial theorem and \cref{cor:positive-C}, for a formal variable $t$,
\[
\widetilde H_{(1^v)}[1-Q+Qt^u;q,t]=(t;t)_v\,h_v\Big[\frac{1-Q+Qt^u}{1-t}\Big]=\sum_{k=0}^{v}\qbin{v}{k}{t}(Q;t)_{v-k}\,Q^kt^{uk}=C_{u,v}(t,Q),
\]
and at $t=s=qQ$ the right-hand side is $C_{u,v}(qQ,Q)$. Thus \cref{thm:higher-rank} for $r=1$ is \cref{thm:main}, the point-count form of \cref{thmA}(a).

\subsection{Questions}\label{subsec:questions}
Let us list open questions raised by the results of the paper.


\begin{question}[pavings and Betti numbers]\label{q:paving}
Does $\Hilb_n(C_{u,v},0)$ admit an affine paving, or at least pure cohomology concentrated in even degrees, for all $u,v,n$? Either would identify the coefficients of \cref{thmA}(a) at $\LL=T^2$ with the Betti numbers, as is the case for $\min(u,v)\le2$ \cite{Luan}, and would make $(\mathrm{ORS}^{\mathrm{vir}}_{u,v})$ in \cref{sec:hopf} equivalent to $(\mathrm{ORS}_{u,v})$. The stratification of \cref{thm:torus-all} has pieces $\AAA^\delta\times\GG_m^\tau$ rather than affine spaces.

The termwise positive expression of \cref{thm:refined}, a sum of products of Gaussian binomials $\qb{v}{k}=\sum_{\lambda\subseteq k\times(v-k)}t^{|\lambda|}$ \cite[Thm.~3.1]{Andrews} and geometric series, suggests cells indexed by $\big(k,\lambda,(m_i)_{v-k\le i<v},(\ell_i)_{0\le i<u}\big)$ with $\lambda\subseteq k\times(v-k)$, of colength $(u+1)k+|\lambda|+\sum(i+1)m_i+\sum(i+1)\ell_i$ and dimension $uk+|\lambda|+\sum im_i+\sum i\ell_i$, on which the contact length is $k+\sum\ell_i$. By \cref{prop:family} it would suffice to construct such cells for the contact strata of a fat line ($u=0$).
\end{question}



\begin{question}[the colored Hopf link]
Compute the $a=0$ part of the $(\mathrm{Sym}^u,\mathrm{Sym}^v)$-colored Khovanov--Rozansky homology of the Hopf link and compare with \cref{cor:hopf}. The first case not covered by \cite{Luan} is $u=v=2$.
\end{question}

\begin{question}[other germs]
By \cref{cor:F}, the zeta function of $\{x^ug=0\}$ is determined by the contact-refined zeta function $G^x_g$ of $g$. Compute $G_g$ for unibranch germs such as $y^p=x^q$, and for germs with several branches. Is there an analogue of the bootstrap when two fat branches are not transverse, for instance for $x^u(x-y^m)^v$?
\end{question}


\begin{question}[the Quot side of the geometry]
The torus stratification of \cref{sec:shadow} has no counterpart yet on the Quot side. Do the Quot schemes $\Quot^{R_{xy}}_m(E_{u,v})$ admit affine pavings, or at least pure cohomology, and is there a relation between the extension fibres $G_{u,v;k}$ and the contact strata of the fat line that explains the identity \cref{thmA}(c) geometrically rather than through the function $F_{u,v}$?
\end{question}


\begin{question}[higher rank]\label{q:higher-rank}
Is there a motivic version of \cref{thm:higher-rank}, with $q$ replaced by $\LL$? Do the elementary modification and the contact-length stratification of \cref{sec:modification,sec:shadow} give stratifications of the Quot schemes $\Quot_N(\mathcal O_{C_{u,v}}^{\oplus r},0)$ into pieces $\AAA^\delta\times\GG_m^\tau$, as they do for $r=1$? Is there a counterpart of the Hilb--Quot identity \cref{thmA}(c) in rank $r$?
\end{question}

\end{document}